\documentclass[reqno,12pt]{amsart}

\usepackage{amsfonts}
\usepackage{amsmath,amssymb,amsthm}
\usepackage{tikz-cd}
\usetikzlibrary{arrows.meta}
\usepackage{hyperref}
\usepackage{enumerate}
\usepackage{lettrine}
\usepackage{mathrsfs}

\everymath{\displaystyle}

\newcommand{\R}{\mathbb R}

\newcommand{\M}{\mathbb M}

\allowdisplaybreaks[4]
\theoremstyle{definition}
\newtheorem{theorem}{Theorem}[section]
\newtheorem{cor}[theorem]{Corollary}
\newtheorem{lem}[theorem]{Lemma}
\newtheorem{prop}[theorem]{Proposition}

\newtheorem{defn}[theorem]{Definition}
\newtheorem{rem}[theorem]{Remark}
\newtheorem{question}[theorem]{Question}
\numberwithin{equation}{section}

\begin{document}
	
	
	\title[A comparison with applications to geometric inequalities]{A comparison theorem with applications to sharp geometric inequalities for submanifolds}
	
	\author{Shengliang Pan, Chengyang Yi$^*$}
	
	\address{School of Mathematical Sciences, Tongji University.
		No.1239, Siping Road, Shanghai, China.}
	\email{slpan@tongji.edu.cn}
	\address{School of Mathematical Sciences, Tongji University.
		No.1239, Siping Road, Shanghai, China.}
	\email{409196288@qq.com}
	\thanks{$^*$Corresponding author: 409196288@qq.com}
	\subjclass[2020]{53C40}
	\keywords{Heintze-Karcher comparison theorem, Chern-Lashof inequality, Willmore-Chen inequality}
	
	\begin{abstract}
		In this paper, we derive an explicit expression  for the Jacobian determinant of the normal exponential map on a submanifold, establishing a relationship with its ambient counterpart. This formula leads to a new comparison theorem which is closely related to the comparison theorem of Heintze-Karcher. As applications, we obtain a Fenchel-Borsuk-Chern-Lashof-type inequality and a Willmore-Chen-type inequality on closed submanifolds in complete noncompact manifolds with nonnegative curvature and Euclidean volume growth.
	\end{abstract}

	\maketitle

	\tableofcontents

\section{Introduction}
In the classical theory of submanifolds in the Euclidean space, the total curvature plays an important role. In 1929, Fenchel \cite{F} proved an inequality for closed curves in $\R^{3}$ which states that if $\gamma$ is a $C^{2}$ closed curve  parametrized by its arc length $s$ in $\R^{3}$ and $|\gamma''(s)|$ is its curvature, then the following inequality holds:
		\begin{equation*}
\int_{\gamma}|\gamma''(s)|ds\geq 2\pi.
	\end{equation*}
The equality above holds if and only if $\gamma$ is a closed convex plane curve. This result was generalized to closed curves in higher dimensional Euclidean spaces by Borsuk \cite{Borsuk} in 1947. Further research on the total absolute curvature of curves can refer to \cite{BriHsi, Fa, Mi}. In order to explore analogous inequalities for higher-dimensional submanifolds, we need appropriate invariants to generalize the concept of curvature of space curves. Several natural candidates have emerged, including the total absolute curvature and the total mean curvature.

In 1957, Chern-Lashof \cite{ChernLa1957} introduced the  total absolute curvature. Let $\Sigma$ be a   closed manifold of dimension $n$ immersed in $\R^{n+m}$, they defined the total absolute curvature $K^{*}(x)$ at $x\in\Sigma$ by
\begin{equation*}
	K^{*}(x):= \int_{S_{x}^{m-1}}	\left|\mathrm{det}\langle h(\cdot,\cdot),\xi \rangle    \right|d\xi,
\end{equation*}
where $S_{x}^{m-1}:=\{ \xi\in T_{x}^{\perp}\Sigma:|\xi|=1 \}$, $h$ is the second fundamental form and $d\xi$ denotes the standard volume element of the unit sphere $S_{x}^{m-1}$. Indeed, the quantity $\left|\mathrm{det}\langle h(\cdot,\cdot),\xi \rangle    \right|$ is the norm of the Jacobian determinant of the Gauss map,
\begin{equation*}
	N:U\Sigma\rightarrow  \mathbb{S}^{n+m-1} \ \  \mathrm{with}\  N(x,\xi):=\xi,
\end{equation*}
at $(x,\xi)\in U\Sigma$, where $U\Sigma$ is the unit normal bundle (see Section \ref{102} for details) and $\mathbb{S}^{n+m-1}$ denotes the $(n+m-1)$-dimensional standard unit sphere. Therefore, $K^{*}(x)$ is the average value of these quantities on the unit normal sphere at $x$ up to a constant. Applying the area formula to the Gauss map, they obtained 
\begin{equation*}
	\int_{\Sigma}K^{*}(x)d\mathrm{vol}_{\Sigma}(x)\geq 2 \left| \mathbb{S}^{n+m-1} \right|.
\end{equation*}
Moreover, the equality above holds if and only if $\Sigma$ is embedded as a convex hypersurface in an
$(n+1)$-dimensional affine subspace of $\R^{n+m}$.

Another natural object is the total mean curvature. In 1968, Willmore \cite{Wi} proved an inequality for closed surfaces in $\R^{3}$. Let $\Sigma$ be a closed surface in $\R^{3}$, then its mean curvature $H$ satisfies
\begin{equation*}
	\int_{\Sigma}\left|H(x)\right|^{2}d\mathrm{vol}_{\Sigma}(x)\geq 4\pi,
\end{equation*}
with equality holding if and only if $\Sigma$ is a sphere in $\R^{3}$. Again employing the Gauss map, Chen \cite{Chen70,Chen71} generalized this result to closed submanifolds of arbitrary codimension in higher-dimensional Euclidean spaces. For a closed manifold $\Sigma$ of dimension $n$ immersed in $\R^{n+m}$, it holds that
		\begin{equation*}
	\begin{array}{lllll}
		\displaystyle\int_{\Sigma}\left|\mathbf{H}\right|^{n} d\mathrm{vol}_{\Sigma}\geq \left| \mathbb{S}^{n} \right|,
	\end{array}
\end{equation*}
where $\mathbf{H}$ denotes the mean curvature vector of $\Sigma$. Moreover, the equality holds if and only if $\Sigma$  is embedded as a hypersphere in an
$(n+1)$-dimensional affine subspace of $\R^{n+m}$.

We next turn to the Riemannian setting. Unless otherwise stated, $M^{N}$ denotes an $N$-dimensional manifold and throughout this paper, all ambient manifolds are assumed to be connected. Any notation not explicitly defined here can be found in Section \ref{102}. Let $(M,\bar{g})$ be a complete noncompact Riemannian manifold of dimension $N$  with nonnegative Ricci curvature. The asymptotic volume ratio of $M$ is defined as
\begin{equation*}
	\mathrm{AVR}(\bar{g}):=\lim_{r\rightarrow\infty}\frac{|B_{p}^{M}(r)|}{\omega_{N}r^{N}}
\end{equation*}
for some (any) fixed point $p\in M$, where $B_{p}^{M}(r)$ denotes the geodesic ball in $M$, $|B_{p}^{M}(r)|$ its volume and $\omega_{N}$ the volume of the unit ball in $\mathbb{R}^{N}$. By the Bishop-Gromov relative volume comparison theorem (see \cite{Pet}, Lemma 7.1.4), the limit exists with $0\leq \mathrm{AVR}(\bar{g})\leq1$. We assume that  $M$ has dimension $N \geq 3$ and Euclidean volume growth, i.e., $\mathrm{AVR}(\bar{g})>0$. Let $\Omega\subset M$ be a bounded open subset with smooth boundary. In 2020, by studying the properties of the solution to a certain harmonic equation, Agostiniani-Fogagnolo-Mazzieri~\cite{AgFoMa} obtained a Willmore-Chen type inequality,
		\begin{equation}\label{85}
	\begin{array}{lllll}
		\displaystyle\int_{\partial\Omega}\left|\mathbf{H}\right|^{N-1} d\mathrm{vol}_{\partial\Omega}\geq \mathrm{AVR}(\bar{g})\left| \mathbb{S}^{N-1} \right|,
	\end{array}
\end{equation}
in this setting. Moreover, the equality holds if and only if $(M\setminus \Omega,\bar{g})$ is isometric to 
\begin{equation*}
	\left( [r_{0},+\infty)\times\partial \Omega,dr\otimes dr+(r/r_{0})^{2}g_{\partial\Omega}   \right)
\end{equation*}
with
\begin{equation*} r_{0}=\left(\frac{|\partial\Omega|}{\mathrm{AVR}(\bar{g})\left| \mathbb{S}^{N-1} \right|}\right)^{\frac{1}{N-1}}.
\end{equation*}
In particular, $\partial\Omega$ is a totally umbilic  connected hypersurface with constant mean curvature. It is noteworthy that this does not yield any information about the interior of $\Omega$. Moreover, it remains unknown whether the equality case forces $M$ to be isometric to the Euclidean space $\R^{N}$. 
In 2023, Wang \cite{Wangxiao} provided a powerful proof of the above conclusion by applying the area formula to the normal exponential map and using standard comparison methods in Riemannian geometry, with inspiration from the Heintze-Karcher method.

In fact, in 1978, Heintze-Karcher \cite{HeKa} established a comparison theorem, which are used by them to derive a volume estimate for tubular neighborhoods of closed submanifolds in complete Riemannian manifolds. The same estimate also yields the inequality mentioned above. However, due to the limitations of the historical context at that time, the notion of the asymptotic volume ratio stated above was not yet available. This is mainly because it was not until 1980, when Gromov improved Bishop's volume comparison theorem (see \cite{Pet}, Lemma 7.1.4), that such a notion came into existence. Very recently, Brendle \cite{Bre2026} pointed out this connection and provided a detailed proof of the inequality by combining the Heintze-Karcher method with the Alexandrov-Bakelman-Pucci (ABP) technique. Moreover, using the Heintze-Karcher comparison theorem (or the Heintze-Karcher estimate directly), one can extend this inequality to the higher-codimensional case, provided that the ambient manifold has nonnegative sectional curvature. Likewise, the same comparison theorem (or its estimate) also allows one to generalize the Fenchel-Borsuk-Chern-Lashof inequality from Euclidean space to all closed submanifolds in a complete noncompact Riemannian manifold with nonnegative sectional curvature and Euclidean volume growth. It is particularly worth noting that the Heintze-Karcher method requires the closed submanifold to be merely immersed, not necessarily embedded.

Denote by  $\M^{N}_{\delta}$  the $N$-dimensional space form of constant sectional curvature $\delta$. Heintze-Karcher proved the following Willmore-Chen-type inequality (\cite{HeKa}, Theorem 2.2). If $\delta>0$, $(M^{n+m}, g)$ is a complete Riemannian manifold satisfying
		\[
\mathrm{Ric}^{M} \ge n\delta \quad (m=1),
\qquad
\sec^{M} \ge \delta \quad (m\ge 2),
\]
and $\Sigma$ is a closed immersed $n$-dimensional submanifold of $M$ with mean curvature vector $\mathbf{H}$, then
	\begin{equation}\label{89}
	\displaystyle\int_{\Sigma}(\delta+\left|\mathbf{H}\right|^{2})^{n/2} d\mathrm{vol}_{\Sigma}\geq  \frac{\left| \mathbb{M}^{n}_{1} \right|}{\left| \mathbb{M}^{n+m}_{\delta} \right|}|M|.
	\end{equation}
 The equality implies that both $M$ and $\Sigma$ are of constant curvature (for a more detailed characterization of the equality case, one can refer to \cite{HeKa}, pp.~465--466). Nevertheless, in both the setting of nonnegative curvature and the setting with a positive lower bound on curvature, the requirement of a lower bound on sectional curvature for higher codimensions appears too strong for these two Willmore-Chen inequalities; this point will be returned to shortly. Therefore, the primary goal of this paper is to establish a Heintze-Karcher-type comparison theorem under weaker curvature conditions, in order to obtain the desired geometric inequalities.

Now we revisit some known comparison theorems, including the aforementioned Heintze-Karcher  comparison theorem. On the one hand, before proving his volume comparison theorem, Bishop first established a volume distortion comparison theorem (cf. \cite{BiCr}, p. 253). When considering volume distortion for $k$-vector forms  (which are not necessarily of the highest degree), Bishop and Crittenden introduced in their book a very suitable curvature notion, the $k$-Ricci curvature, which lies between the sectional curvature and the Ricci curvature. Let $M^{N}$ be an $N$-dimensional Riemannian manifold, $p\in M$, $v\in T_{p}M$ a unit tangent vector and $P\subset T_{p}M$ a $k$-dimensional linear subspace such that $v\bot P$. The $k$-Ricci curvature of $(v,P)$ is defined by
	\[
	\mathrm{Ric}^{M}_{k}(v,P):=\sum_{i=1}^{k}	\bar{R}(v,e_{i},e_{i},v),
	\]
	where $\{e_{i}  \}_{i=1}^{k}$ is an orthonormal basis of $P$. We say $\mathrm{Ric}^{M}_{k}\geq k\delta$ if $\mathrm{Ric}^{M}_{k}(v,P)\geq k\delta$ for all $p\in M$, unit tangent vector $v\in T_{p}M$ and $k$-dimensional linear subspace $P$ satisfying $v\bot P$. Note that  $\text{Ric}^{M}_{1}\geq \delta$
	is equivalent to the sectional curvature being bounded below by $\delta$ and $\mathrm{Ric}^{M}_{N-1}\geq (N-1)\delta$ is equivalent to the Ricci curvature being bounded below by $(N-1)\delta$. It also follows immediately that for $1\leq k \leq l \leq N-1$, the condition $\mathrm{Ric}^{M}_{k} \geq k\delta$ implies $\mathrm{Ric}^{M}_{l} \geq l\delta$. The geometry and topology of manifolds with bounded $k$-Ricci curvature have been studied by many authors; see, for example, \cite{DomGonMou, KetMon, LeeRi, MaWu, Mou, ReisWr, Shen, Wangkai}. Inspired by the work of Bishop, as well as by subsequent results on $k$-Ricci curvature, we believe that the concept of $k$-Ricci curvature may be very appropriate when studying certain properties of $k$-dimensional submanifolds.
	
		On the other hand, Heintze-Karcher \cite{HeKa} extended Bishop's comparison theorem to the case of submanifolds.  However, in the higher codimension setting, their comparison theorem requires, in some situations, lower bounds on the sectional curvature of the ambient manifold. If one of the two higher‑codimension submanifolds being compared is a submanifold of a space form,  inspired by the work of Bishop, we believe that the sectional curvature condition required for the other submanifold is too strong. In 2020, Chahine \cite{Chah} obtained a Heintze-Karcher estimate for manifolds with lower bounded $k$-Ricci curvature, where the constant $k$ depends on the codimension. Nevertheless, perhaps due to the authors' limited ability, we have not succeeded in using the proof method of Heintze and Karcher to replace the sectional curvature lower bound condition in their comparison theorem for the case of $k$-dimensional submanifolds of higher codimension with what we consider to be the most suitable condition, namely a lower bound on the $k$-Ricci curvature. Therefore, we must seek a new approach.

	 In 2001, Cordero-Erausquin, McCann and Schmuckenschläger \cite{CoMcSc} used optimal transport to generalize the Borell-Brascamp-Lieb inequality to Riemannian manifolds. In doing so, they needed to compute the Jacobian determinant of the McCann map (see \cite{Mc}; see also \cite{Vi})	$\tilde{\Phi}:M^{N}\rightarrow M^{N}$ given by 
		\begin{equation*}
			\tilde{\Phi}(p):=\mathrm{Exp}_{p}(\bar{\nabla} \tilde{u}(p)), \qquad p\in M.
		\end{equation*}
This computation reduces to studying the coefficient matrix of $N$ Jacobi fields in a suitable frame, which in turn requires analyzing a family of geodesic variations. A key observation is the identity
\begin{equation*}
	t\bar{\nabla}\tilde{u}(c(s)) =-\frac{\bar{\nabla}d_{\gamma(t)}^{2}(c(s))}{2}+\frac{\bar{\nabla}d_{\gamma(t)}^{2}(c(s))}{2}+ t\bar{\nabla}\tilde{u}(c(s)),
\end{equation*}
which leads to a subtle result upon differentiating $\tilde{\Phi}$. The notation here differs slightly from the original; we refer the reader to pp.~229--230 and 241--244 in \cite{CoMcSc} for details. To the best of our knowledge, the above-mentioned technique does not offer significant advantages in the full-dimensional case. Similarly, in their proof of the comparison theorem, Heintze-Karcher mainly focused on the differential of the normal exponential map and the associated Jacobi field estimates. In order to obtain more general conclusions in this paper, we add a gradient field term to the submanifold. In Section \ref{103}, we will investigate the Jacobian determinant of the map $\Phi:T^{\perp}\Sigma\rightarrow M^{n+m}$ given by 
\begin{equation*}
	\Phi(x,y)=\mathrm{exp}_{x}\left(\nabla^{\Sigma}u(x)+y\right),
\end{equation*}
for all $(x,y)\in T^{\perp}\Sigma$ (see Section \ref{103} for full details). Conventional approaches focus on the behaviour of component matrices associated with $n+m$ Jacobi fields under specific frames (see \cite{Bre2023}, \cite{Chah} and pp.~20--22 in \cite{DaWe}; see also \cite{HeKa}). When we applied the above technique of Cordero-Erausquin, McCann and Schmuckenschläger to study the differential of the map $\Phi$, we found that, in many situations, the method exhibits a distinct advantage over the traditional approach in the submanifold setting, although their original research did not concern submanifold-related problems. The matrix arising from this computation can be decomposed into the product of two well-structured matrices. One corresponds to the Jacobi matrix of the exponential map in the ambient space, while the other admits a desirable block structure containing a zero matrix and an identity matrix (see Lemma~\ref{4}). This facilitates the construction of comparison theorems and allows us to relax the restrictions imposed on curvature conditions.
		An explicit Jacobian determinant formula for $\Phi$ follows immediately from this matrix equation, establishing a fundamental relationship with its ambient counterpart. To the best of our knowledge, this formula seems to be overlooked in the submanifold theory. It is worth mentioning that this formula is closely related to Brendle’s estimate (\cite{Bre2023}, Corollary 4.7), which is established using the Alexandrov-Bakelman-Pucci method and can be used to study certain Sobolev-type inequalities and isoperimetric inequalities. 
		 This formula contains a wealth of information, but in this paper we shall only use it to revisit the Heintze-Karcher comparison theorem. We hope that the formula and the technique of Cordero-Erausquin, McCann and Schmuckenschläger will attract more attention, especially from those working in submanifold theory.

		 For a real number $\delta$, we use the abbreviation (cf. \cite{HeKa}, p. 453)
		 \begin{align*}
		 	&\mathfrak{s}_{\delta}(r):=	\left\{\begin{aligned}
		 		&\delta^{-1/2}\sin(\delta^{1/2}r)&\ \mathrm{if\ }\delta>0,\\
		 		&r&\ \mathrm{if\ }\delta=0,\\
		 		&|\delta|^{-1/2}\sinh(|\delta|^{1/2}r)&\ \mathrm{if\ }\delta<0,
		 	\end{aligned}\right.\\
		 	& \mathfrak{c}_{\delta}(r):=\mathfrak{s}_{\delta}'(r).
		 \end{align*}
		 In contrast to the Heintze-Karcher method, we establish the following comparison theorem by exploiting the two features of the matrix equation discussed above and the Jacobian determinant formula for the normal exponential map.
		 \begin{theorem}\label{57}
		 	Let $f: \Sigma^n \rightarrow M^{n+m}$ be an isometric immersion of an $n$-manifold into a complete Riemannian $(n+m)$-manifold. Let $(x,\xi)\in U\Sigma$ and define the geodesic $\sigma(t):=\mathrm{exp}^{\perp}_{x}\left(t\xi\right)$ for $0\leq t < \infty$. Consider a second such situation $\underline{f}:\underline{ \Sigma}^n\rightarrow \M^{n+m}_{\delta}$ etc. Then the following two assertions hold.

		 	\smallskip
		 	{\rm (i)} If the sectional curvatures of $M$ are bounded below by $\delta$ and $\kappa_{i}(x,\xi)\leq \underline{\kappa}_{i}(\underline{x},\underline{\xi})$ for each $1\leq i \leq n$, then 
		 	\begin{align}\label{43}
		 		|\mathrm{det}\ (\mathrm{exp} ^{\perp})_{*(x,t\xi)}|\leq \left(\mathfrak{s}_{\delta}(t)/t\right)^{m-1}\prod_{i=1}^{n}\left( \mathfrak{c}_{\delta}(t)+\mathfrak{s}_{\delta}(t)\underline{\kappa}_{i}(\underline{x},\underline{\xi}) \right)
		 	\end{align}
		 	for each positive number $t$ not larger than $\tilde{\tau}_{f}(x,\xi)$ and
		 	\begin{align*}
		 		\frac{|\mathrm{det}\ (\mathrm{exp} ^{\perp})_{*(x,t\xi)}|}{|\mathrm{det}\ (\underline{\mathrm{exp}} ^{\perp})_{*(\underline{x},t\underline{\xi})}|}\leq \frac{|\mathrm{det}\ (\mathrm{exp} ^{\perp})_{*(x,s\xi)}|}{|\mathrm{det}\ (\underline{\mathrm{exp}} ^{\perp})_{*(\underline{x},s\underline{\xi})}|},
		 	\end{align*}
		 	for $0<s\leq t<\tilde{\tau}_{f}(x,\xi)$. Moreover, the equality in \eqref{43} holds for some positive number $t_{0}$ not larger than $\tilde{\tau}_{f}(x,\xi)$
		 	if and only if the sectional curvature of $M$ is equal to $\delta$ for all tangent planes containing $\sigma'(a)$ for each $a \in [0, t_{0}]$ and $\kappa_{i}(x,\xi)= \underline{\kappa}_{i}(\underline{x},\underline{\xi})$ for each $1\leq i \leq n$.

		 	\smallskip
		 	{\rm (ii)} If $M$ satisfies $\mathrm{Ric}^{M}_{n}\geq n\delta$, $\underline{\Sigma}$ is  umbilical at $\underline{x}$ for the normal $\underline{\xi}$ and $\langle 	\mathbf{H}(x),\xi \rangle\geq  \langle 	\underline{\mathbf{H}}(\underline{x}),\underline{\xi} \rangle$, then 
		 	\begin{align}\label{44}
		 		|\mathrm{det}\ (\mathrm{exp} ^{\perp})_{*(x,t\xi)}|\leq \left(\mathfrak{s}_{\delta}(t)/t\right)^{m-1}\left( \mathfrak{c}_{\delta}(t)-\mathfrak{s}_{\delta}(t)\langle \underline{\mathbf{H}}(\underline{x}),\underline{\xi} \rangle \right)^{n}
		 	\end{align}
		 	for each positive number $t$ not larger than $\tilde{\tau}_{f}(x,\xi)$ and
		 	\begin{align*}
		 		\frac{|\mathrm{det}\ (\mathrm{exp} ^{\perp})_{*(x,t\xi)}|}{|\mathrm{det}\ (\underline{\mathrm{exp}} ^{\perp})_{*(\underline{x},t\underline{\xi})}|}\leq \frac{|\mathrm{det}\ (\mathrm{exp} ^{\perp})_{*(x,s\xi)}|}{|\mathrm{det}\ (\underline{\mathrm{exp}} ^{\perp})_{*(\underline{x},s\underline{\xi})}|},
		 	\end{align*}
		 	for $0<s\leq t<\tilde{\tau}_{f}(x,\xi)$. Moreover,  the equality in \eqref{44} holds for some positive number $t_{0}$ not larger than $\tilde{\tau}_{f}(x,\xi)$
		 	if and only if  the sectional curvature of $M$ is equal to $\delta$ for all tangent planes containing $\sigma'(a)$ for each $a \in [0, t_{0}]$ and $\kappa_{i}(x,\xi)=-\langle 	\underline{\mathbf{H}}(\underline{x}),\underline{\xi} \rangle $ for each $1\leq i \leq n$.
		 \end{theorem}
		 
	The notation used in the theorem above is explained in Sections \ref{102} and \ref{103}. The first assertion in Theorem \ref{57} has already been proved by Heintze-Karcher;  the second assertion was weakened to a lower bound on the $n$-Ricci curvature. With this comparison theorem, a careful inspection of the Heintze-Karcher proof allows us to replace, in the higher‑codimensional case of \eqref{89}, the sectional curvature lower bound $\delta$ by the weaker condition $\mathrm{Ric}^{M}_{n}\geq n\delta$; we do not prove this improvement here. Likewise, our comparison theorem allows us to weaken, in many of the remaining higher-codimension cases considered by Heintze-Karcher(specifically, Theorem 2.1, Theorem 2.3 and Corollary 2.3.1 in \cite{HeKa}),  the sectional curvature lower bound to a $k$-Ricci curvature lower bound. These results follow by combining our comparison theorem with a detailed examination of the original argument, and we omit the details here.

		 The interval of $t$ for which Theorem \ref{57} holds is contained in the corresponding interval of the original Heintze-Karcher comparison theorem, since $\tilde{\tau}_{f}$ is bounded above by the first focal distance $\rho$ (see Sections~\ref{102} and~\ref{103}). This restriction does not affect its use in many geometric inequality settings.
		 On this possibly smaller interval, we  establish a stronger monotonicity property (see Theorem \ref{38} and the proof of Theorem \ref{57}) than that obtained from the Heintze-Karcher comparison theorem. This monotonicity plays a crucial role, in particular in characterizing the equality cases of the two types of geometric inequalities that we will consider in this paper. The proof is based on a completely new approach, using the Jacobian determinant formula for the normal exponential map mentioned above. In fact, this explicit formula allows us to reduce the problem to a careful study of the Hessian/Laplacian comparison theorem. As a byproduct, before proving Theorem \ref{38}, we first establish a monotonicity property in the Hessian/Laplacian comparison theorem (see Theorem \ref{17}). Although this monotonicity is not used elsewhere in the paper, to the best of our knowledge it has not appeared in the literature, we present it separately. While Theorem \ref{57} can be applied to establish various geometric inequalities, we present here the Fenchel-Borsuk-Chern-Lashof-type inequality and the Willmore-Chen-type inequality.

		 It is worth mentioning that the Heintze-Karcher comparison theorem also yields the famous Heintze-Karcher inequality (see \cite{Ro} and \cite{MoRo1991}). For further researches on the Heintze-Karcher inequality, we recommend the references \cite{Bre2013, ChenHeQi, FoPi, HuWeiXiaZh, JiXiaZh, LiXia}.

By analogy with the set of elliptic points of a surface in $\R^{3}$ (cf. \cite{MoRo2009}, p. 83), we introduce the following set $\Sigma^{+}$ for a submanifold $\Sigma$ of arbitrary codimension in a general Riemannian manifold:
	\begin{align*}
	\Sigma^{+}:=&\left\{ x\in \Sigma:\exists \xi\in S_{x}^{m-1}\,\mathrm{s.t.}\, \langle h(v,v),-\xi \rangle>0\right.\\
	&\left.\ \ \mathrm{for\ each\ unit\ vector}\ v\in T_{x}\Sigma\right\}.
\end{align*}
 Now, assume that the image of the second fundamental form $h:T_{x}\Sigma\times T_{x}\Sigma\rightarrow T_{x}^{\perp}\Sigma$ is exactly a $1$-dimensional linear subspace of $T_{x}^{\perp}\Sigma$ for each $x\in\Sigma^{+}$. Denote by $\{-\xi_{x}\}$ the set
 \[
 \{h(v,v)/|h(v,v)|: v\in T_{x}\Sigma,\,|v|=1   \},
 \]
define a subset of the normal bundle by
	\begin{equation}\label{84}
\mathcal{U}:=\{(x,y)\in T^{\perp}\Sigma: x\in \Sigma^{+},\, y\in T_{x}^{\perp}\Sigma,\,\langle \xi_{x},y \rangle >0  \},
\end{equation}
and choose an orthonormal basis $\{a_{i}\}_{i=1}^{n}$ of $T_{x}\Sigma$ with dual basis $\{\omega^{i}\}_{i=1}^{n}$ of  $T_{x}^{*}\Sigma$  such that
\[
\langle h(a_{i},a_{j}),-\xi_{x} \rangle=\kappa_{i}(x,\xi_{x})\delta_{ij},
\]
  for $1\leq i,j \leq n$ and 
\[
0<\kappa_{1}(x,\xi_{x})\leq \kappa_{2}(x,\xi_{x})\leq \cdots\leq \kappa_{n}(x,\xi_{x}),
 \]
where $\kappa_{i}$ are precisely the principal curvatures (see Definition \ref{83}). Using these notations, we state the following Fenchel-Borsuk-Chern-Lashof type inequality.
\begin{theorem}\label{67}
	Let  $(M^{n+m},\bar{g})$ be a complete noncompact Riemannian manifold with nonnegative sectional curvature and Euclidean volume growth. Let $\Sigma$ be a closed $n$-dimensional Riemannian manifold  and $f: \Sigma^{n} \rightarrow M$ be an isometric immersion, then we have
	\begin{equation}\label{54}
		\int_{\Sigma}K^{*}(x)d\mathrm{vol}_{\Sigma}(x)\geq 2 \mathrm{AVR}(\bar{g}) \left| \mathbb{S}^{n+m-1} \right|.
	\end{equation}
	Moreover, equality in \eqref{54} holds if and only if

	\smallskip	
	{\rm (i)} For each $x\in\Sigma^{+}$, the image of the second fundamental form $h:T_{x}\Sigma\times T_{x}\Sigma\rightarrow T_{x}^{\perp}\Sigma$ is exactly a $1$-dimensional linear subspace of $T_{x}^{\perp}\Sigma$;
	
	\smallskip
	{\rm (ii)}  The normal exponential map $\mathrm{exp}^{\perp}|_{\mathcal{U}}:\mathcal{U}\rightarrow \mathrm{exp}^{\perp}(\mathcal{U})$ is a diffeomorphism with
	\begin{equation*}
		[(\mathrm{exp}^{\perp})^{*}\bar{g}](x,y)=	\sum_{i=1}^{n}\left[\left(1+\kappa_{i}(x,\xi_{x})\langle \xi_{x},y \rangle \right)^{2}-1\right]\omega^{i}\otimes \omega^{i}+g_{T^{\perp}\Sigma}(x,y),
	\end{equation*}
	for each $(x,y)\in \mathcal{U}$;
	
	\smallskip
	{\rm (iii)} For each $x\in\Sigma\setminus\Sigma^{+}$ and $\xi\in S_{x}^{m-1}$, the linear transformation $\langle h(\cdot,\cdot),\xi\rangle$ necessarily has a zero eigenvalue.
\end{theorem}


\begin{rem}
	The equality in \eqref{54} can also yield that $\Sigma$ is connected and $f|_{\Sigma^{+}}$ is an embedding. In fact, the inequality in the theorem above can be derived directly from the original Heintze-Karcher comparison theorem. Moreover, the proof of the theorem above establishes the stronger inequality
	\begin{equation*}
		\int_{\Sigma^{+}}K^{*}(x)d\mathrm{vol}_{\Sigma}(x)\geq 2 \mathrm{AVR}(\bar{g}) \left| \mathbb{S}^{n+m-1} \right|,
	\end{equation*}
	where equality holds if and only if conditions (i) and (ii) in the theorem are fulfilled.
\end{rem}
Therefore, the following corollary is a direct result.

\begin{cor}
Let $M$ be a complete noncompact Riemannian manifold with nonnegative sectional curvature and Euclidean volume growth, and let $\Sigma$  be a closed submanifold of $M$. Then the set $\Sigma^{+}$  must be nonempty.
\end{cor}

Before turning to the generalized Willmore-Chen inequality, let us first consider an example. In 1979, Eguchi-Hanson \cite{EgHa} (see also \cite{And}, p. 270) constructed a  Ricci flat metric on
$T\mathbb{S}^{2}$ such that $T\mathbb{S}^{2}$ has Euclidean volume growth and contains a totally geodesic closed submanifold $\mathbb{S}^{2}$. This shows that, in general, nonnegative Ricci curvature condition is not sufficient to directly extend the Willmore-Chen-type inequality established by Agostiniani-Fogagnolo-Mazzieri for hypersurfaces to all higher codimensions. By virtue of Theorem \ref{57}, we obtain the following Willmore-Chen-type inequality.
\begin{theorem}\label{68}
	Let $(M^{n+m},\bar{g})$ $(n\geq2)$ be a complete noncompact Riemannian manifold with $\mathrm{Ric}_{n}^{M}\geq 0$ and Euclidean volume growth. Let $\Sigma$ be a closed $n$-dimensional  Riemannian manifold  and $f: \Sigma^{n} \rightarrow M$ be an isometric immersion, then we have
		\begin{equation}\label{55}
		\begin{array}{lllll}
			\displaystyle\int_{\Sigma}\left|\mathbf{H}\right|^{n} d\mathrm{vol}_{\Sigma}\geq \mathrm{AVR}(\bar{g})\left| \mathbb{S}^{n} \right|.
		\end{array}
	\end{equation}
Let $\mathcal{W}:=\{(x,y)\in T^{\perp}\Sigma: x\in \Sigma,\, y\in T_{x}^{\perp}\Sigma,\, \langle y,-\mathbf{H}(x) \rangle > 0  \}$. Then the equality in \eqref{55} holds if and only if the normal exponential map $\mathrm{exp}^{\perp}|_{\mathcal{W}}:\mathcal{W}\rightarrow \mathrm{exp}^{\perp}(\mathcal{W})$ is a diffeomorphism with
\[
	[(\mathrm{exp}^{\perp})^{*}\bar{g}](x,y)=	\left[\left(1-\langle \mathbf{H}(x),y \rangle \right)^{2}-1\right]g_{\Sigma}(x)+g_{T^{\perp}\Sigma}(x,y),
\]
	for each $(x,y)\in \mathcal{W}$. In particular, $f$ is an embedding and $\Sigma$ is connected and totally umbilic with $D^{\perp}$-parallel mean curvature vector, where $D^{\perp}$ is the normal connection on $T^{\perp}\Sigma$. Moreover, $|\mathbf{H}|\equiv (\mathrm{AVR}(\bar{g})\left| \mathbb{S}^{n} \right| /  |\Sigma|)^{\frac{1}{n}}$.
\end{theorem}
As an application of the above theorem, we can get the following result immediately.
\begin{cor}\label{90}
	Let $M$ be an $N(N\geq 3)$-dimensional complete noncompact Riemannian manifold with nonnegative $n$-Ricci curvature and Euclidean volume growth. Then there is no $k$-dimensional closed minimal submanifold for all $n\leq k \leq N-1$.
\end{cor}

In fact, we believe that the curvature conditions in most of the results obtained above are optimal. However, a rigorous justification of this optimality would seem to require a substantial number of examples, which we are currently unable to construct. We therefore pose the following question.

\begin{question}
In the higher‑codimension setting, is it possible to improve the condition $\mathrm{Ric}^{M}_{n}\geq n\delta$ in (ii) of Theorem \ref{57} to $\mathrm{Ric}^{M}_{n+1}\geq (n+1)\delta$? Similarly, can the condition $\mathrm{Ric}^{M}_{n}\geq 0$ appearing in (ii) of Theorem \ref{68} and in Corollary \ref{90} be relaxed to $\mathrm{Ric}^{M}_{n+1}\geq 0$?
\end{question}

Although Heintze-Karcher already gave a complete characterization of the equality case in inequality \eqref{89} under a positive lower curvature bound, in our study of the equality conditions for the above two geometric inequalities under nonnegative curvature, we overcome obstacles that are entirely different, arising from higher codimension and the immersed condition.

Regarding further investigations, we recommend \cite{ChernLa1958, ChoChRi, Ho, Pe, We} for the Chern-Lashof-type inequalities and \cite{CeMi, JiWaXiaZh, WuWu} for the Willmore-Chen-type inequalities.

This paper is organized as follows. In Section~\ref{102}, we introduce some notations and preliminaries. In Section~\ref{103}, we prove a differential formula, and then use it to establish two comparison theorems. In Section~\ref{104}, we prove the inequality~\eqref{54} in Theorem~\ref{67}. In Section~\ref{105}, we prove the necessity part of Theorem~\ref{67}. In Section~\ref{106}, we prove the sufficiency part of Theorem~\ref{67}. In Section~\ref{107}, we prove the inequality~\eqref{55} in Theorem~\ref{68}. In Section~\ref{108}, we prove the necessity part of Theorem~\ref{68}. Since the proof of the sufficiency part of Theorem~\ref{68} is very similar to that of Theorem~\ref{67}, we omit it.

{\bf Acknowledgements.} We would like to thank Prof. Guofang Wei for bringing Chahine's relevant work to our attention. We also thank Prof. Yunlong Yang and Dr. Shuangqi Liu for their valuable discussions. We are grateful to Jihye Lee and Fabio Ricci for providing the Eguchi-Hanson example. This work is supported by the National Natural Science Foundation of China (No. 12571062).

\section{Preliminaries}\label{102}

Let $(M, \bar{g})$ be a complete  $(n+m)$-dimensional Riemannian manifold and $f: \Sigma^{n} \rightarrow M$ be an isometric immersion. Then $ (\Sigma, f^{*}\bar{g})$ is an $n$-dimensional Riemannian submanifold of $M$ endowed with the induced metric $g_{\Sigma}:=f^{*}\bar{g}$ simply by $g$. Denote by $\bar{D}$ the Levi-Civita connection of $M$ and by $D^{\Sigma}$ or $D_{\Sigma}$ that of $\Sigma$. Let $\bar{\nabla}$ and $\nabla^{\Sigma}$ be the gradient operators of $M$ and $\Sigma$, respectively. We distinguish between $x$ and $f(x)$, as well as between the following different exponential maps. However, we canonically identify $f_{*}(w)$ with $w$ for all $w\in T\Sigma$, and regard $T_{x}^{\perp}\Sigma$ as the orthogonal complement of $T_{x}\Sigma$ in $T_{f(x)}M$. We shall make use of the following convention on 
the ranges of indices:
\begin{align*}
	1\leq i,j,k,\cdots,\leq n;\ n+1\leq \alpha,\beta,\cdots,\leq n+m;\ 1\leq A,B,C,\cdots,\leq n+m,
\end{align*}
and adopt the Einstein summation convention. The Kronecker delta is defined by
\begin{align*}
	\delta_{AB}=\delta_{A}^{B}=\delta^{AB}:=	\left\{\begin{aligned}
		&1\ \mathrm{if\ }A=B,\\
		&0\ \mathrm{if\ }A\neq B.
	\end{aligned}\right.
\end{align*}

For $	X,Y,Z,W\in \Gamma( T\Sigma)$, the $(0,4)$-type Riemann curvature tensor of $M$ is defined by
\begin{equation*}
	\bar{R}(X,Y,Z,W):=\langle \bar{D}_{X} \bar{D}_{Y}Z-\bar{D}_{Y}\bar{D}_{X}Z-\bar{D}_{[X,Y]}Z,W \rangle.
\end{equation*}
The Ricci curvature tensor of $M$ at $p\in M$ is given by 
\begin{equation*}
\mathrm{Ric}^{M}(X,Y):=\sum_{A=1}^{n+m}	\bar{R}(X,e_{A},e_{A},Y),
\end{equation*}
where $\{e_{A}  \}_{A=1}^{n+m}$ is an orthonormal basis of $T_{p}M$. Denote by $d$ the distance function and by $\mathrm{Exp}:TM\rightarrow M$ the exponential map on $M$. The unit tangent bundle of $M$ is given by 
\begin{equation*}
	SM:=\left\{ (p,v)\in TM:p\in M  ,\,v \in T_{p}M  \, \mathrm{and}\, |v|=1 \right\}.
\end{equation*}
Define a function $\mu:SM\rightarrow (0,\infty]$ by
\begin{equation*}
	\mu (p,v):=\sup\left\{t>0:d(\mathrm{Exp}_{p}(tv),p)=t   \right\}. 
\end{equation*}
The point $\mathrm{Exp}_{p}(\mu (p,v)v)$ is called a cut point of $p$ if $\mu (p,v)<\infty$. Note that $\mu$ is continuous on which $\mu$ is finite (cf. \cite{CheeEb}, Proposition 5.4).

Denote by $\pi:T^{\perp}\Sigma\rightarrow \Sigma$ the  normal bundle of $\Sigma$ and by $T^{\perp}\Sigma$ for short with 
  \begin{equation*}
	T^{\perp}\Sigma:=\left\{ (x,y):x\in \Sigma  ,\,y \in T_{x}^{\perp}\Sigma \right\}.
\end{equation*}
The second fundamental form $h$ of $\Sigma$ in $M$ is defined by
\begin{equation*}
 h(X,Y)=\bar{D}_{X}Y-D^{\Sigma}_{X}Y,
\end{equation*}
and the mean curvature vector $\mathbf{H}$ at a point $x\in\Sigma$ is given by
\begin{equation*}
	\mathbf{H}(x)=\frac{1}{n}\sum_{i=1}^{n}h(e_{i},e_{i}),
\end{equation*}
where $\{e_{i}\}_{i=1}^{n}$ is an orthonormal basis of $T_{x}\Sigma$. The normal connection $D^{\perp}$  on $T^{\perp}\Sigma$ is defined by
\begin{equation*}
	D_{X}^{\perp}\xi:=\left(\bar{D}_{X}\xi\right)^{\perp},
\end{equation*}
where $\xi\in \Gamma( T^{\perp}\Sigma)$ and $(\cdot)^{\perp}$ denotes the projection to $T^{\perp}\Sigma$. The covariant derivative of $h$, defined via connections $D^{\Sigma}$ and $D^{\perp}$, is given by
\begin{equation}\label{52}
	D_{X}h(Y,Z):=D_{X}^{\perp}(h(Y,Z))-h(D_{X}^{\Sigma}Y,Z)-h(Y,D_{X}^{\Sigma}Z).
\end{equation}
The following is called the Codazzi equation
\begin{equation}\label{53}
	\langle  (D_{X}h)(Y,V),\xi   \rangle-\langle  (D_{Y}h)(X,V),\xi   \rangle=\bar{R}(X,Y,V,\xi).
\end{equation}
 For convenience, we set
\begin{equation*}
 S_{x}^{m-1}:=\{ \xi\in T_{x}^{\perp}\Sigma:|\xi|=1 \}.
\end{equation*}
  The definition of the total absolute curvature  $K^{*}(x)$ at $x\in\Sigma$ given by Chern-Lashof mentioned previously carries over to the Riemannian setting, although it loses some of its geometric intuition. The unit normal bundle of $\Sigma$ is given by
  \begin{equation*}
  	U\Sigma:=\left\{ (x,\xi)\in T^{\perp}\Sigma:x\in \Sigma  ,\,\xi \in S_{x}^{m-1} \right\}.
  \end{equation*}
  \begin{defn}[Weingarten map and the principal curvatures, cf. \cite{Ca}]\label{83}
  	Given $(x,\xi)\in U\Sigma$, the \textbf{ Weingarten map} $S_{\xi}:T_{x}\Sigma\rightarrow T_{x}\Sigma$ is defined by 
  	  \begin{equation*}
  	\langle S_{\xi}(v),w\rangle:=-\langle h(v,w),\xi\rangle,
  	\end{equation*}
  	for $v,w\in T_{x}\Sigma$. Its eigenvalues, ordered increasingly, denoted by
  	$\kappa_{1}(x,\xi),\cdots,\kappa_{n}(x,\xi)$, are called the \textbf{principal curvatures} of $\Sigma$ in direction $(x,\xi)$.
  \end{defn}

The exponential map $\mathrm{Exp}$ on $M$ can induce the definition of the normal exponential map $\mathrm{exp}^{\perp}:T^{\perp}\Sigma\rightarrow M$ on $\Sigma$ which is given by
  \begin{equation*}
\mathrm{exp}_{x}^{\perp}(y):=\mathrm{exp}^{\perp}(x,y):=\mathrm{Exp}(f(x),y) 
  \end{equation*}
for $x\in \Sigma$ and $y\in T_{x}^{\perp}\Sigma$. 
\begin{defn}[Focal point]
Given $(x,\xi)\in U\Sigma$, let $\sigma(t):=\mathrm{exp}^{\perp}(x,t\xi)$ for $t\in[0,\infty)$ be a geodesic. The point $\mathrm{exp}^{\perp}(x,t_{0}\xi)$, $t_{0}\in(0,\infty)$, is called a \textbf{focal point} of $\Sigma$ along $\sigma$ if $(x,t_{0}\xi)$ is a critical point of $\mathrm{exp}^{\perp}$. Moreover, define a function $\rho:U\Sigma\rightarrow (0,\infty]$ by
\begin{equation*}
	\rho(x,\xi):=\sup\left\{t_{0}\in(0,\infty):\mathrm{exp}^{\perp}_{*(x,t\xi)}\, \mathrm{is\, nonsingular\, for\, all\, }t\in[0,t_{0}) \right\} .
\end{equation*}
If $\rho(x,\xi)<\infty$, then $\mathrm{exp}^{\perp}(x,\rho(x,\xi)\xi)$ is called  the \textbf{first focal point} of $\Sigma$ along $\sigma$.
\end{defn}
Note that $\rho$ is continuous on which $\rho$ is finite (\cite{ItTa}, Theorem A).

Next, define a function $\tau_{f}:U\Sigma\rightarrow [0,\infty]$ by
\begin{equation*}
	\tau_{f}(x,\xi):=\sup\left\{t>0:d(\mathrm{exp}_{x}^{\perp}(t\xi),f(\Sigma))=t   \right\}     ,
\end{equation*}
where $d(p,f(\Sigma))=\inf_{x\in\Sigma}d(p,f(x))$ for $p\in M$. If $\tau_{f}(x,\xi)<\infty$, then $\mathrm{exp}^{\perp}_{x}(\tau_{f}(x,\xi)\xi)$ is called a cut point of $\Sigma$. It should be noted that if 
$f$ is an embedding, then
 $\tau_{f}$ is continuous on which $\tau_{f}$ is finite (\cite{ItTa}, Theorem B) and takes non-zero values; if 
$f$ is merely an immersion, neither property necessarily holds. From the definitions of $\mu$ and $\tau_{f}$, it is clear that
\begin{equation}\label{5}
	\tau_{f}(x,\xi)\leq \mu(f(x),\xi)\leq \infty,
\end{equation}
for all $x\in\Sigma$ and $\xi\in S_{x}^{m-1}$. Furthermore, it is well-known that
\begin{equation}\label{3}
	 \tau_{f}(x,\xi)\leq\rho(x,\xi)\leq \infty,
\end{equation}
for all $x\in\Sigma$ and $\xi\in S_{x}^{m-1}$. 
There is a natural pullback bundle $\tilde{\pi}:f^{*}TM\rightarrow \Sigma$ of $TM$ via the immersion $f$, where
\begin{equation*}
	f^{*}TM=\left\{ (x,z):x\in \Sigma  ,\,z \in T_{f(x)}M \right\}.
\end{equation*}
This bundle inherits an exponential map $\mathrm{exp}:f^{*}TM\rightarrow M$ from the exponential map $\mathrm{Exp}$ of $M$. Explicitly, for $(x,z)\in f^{*}TM$, define
\[
\mathrm{exp}_{x}(z):=\mathrm{exp}(x,z):=\mathrm{Exp}(f(x),z).
\]
Denote by $B^{\Sigma}_{x}(r)$ the intrinsic geodesic ball of radius $r$ centered at $x$ in $\Sigma$, by $\mathscr{H}^{0}(X)$ the number of elements in the set $X$.

We end this section with the following remark.
\begin{rem}
	Some examples below illustrate that the relationships among $\tau_{f}$,  $\rho$ and $\mu$   are generally complicated.
	
	\smallskip	
	{\rm (i)} Let $M=\R^{3}$, $\Sigma=\{(x_{1},x_{2},x_{3})\in \R^{3}:x_{1}^{2} +x_{2}^{2}+x_{3}^{2}=1\}$, $x=(0,0,1)$, $\xi=(0,0,1)$, then $\tau_{f}(x,\xi)=\infty$, $\rho(x,\xi)=\infty$ and $\mu(x,\xi)=\infty$;
	
	\smallskip
	{\rm (ii)} Let $M=\R^{3}$, $\Sigma=\{(x_{1},x_{2},x_{3})\in \R^{3}:x_{1}^{2} +x_{2}^{2}+x_{3}^{2}=1\}$, $x=(0,0,1)$, $\xi=(0,0,-1)$, then $\tau_{f}(x,\xi)=1$, $\rho(x,\xi)=1$ and $\mu(x,\xi)=\infty$;
	
	\smallskip
	{\rm (iii)} Let $M=\{(x_{1},x_{2},x_{3},x_{4})\in \R^{4}:x_{1}^{2} +x_{2}^{2}+x_{3}^{2}+x_{4}^{2}=1\}$, $\Sigma=\{(x_{1},x_{2},x_{3},0)\in \R^{4}:x_{1}^{2} +x_{2}^{2}+x_{3}^{2}=1\}$, $x=(1,0,0,0)$, $\xi=(0,0,0,1)$, then $\tau_{f}(x,\xi)=\pi$, $\rho(x,\xi)=\pi$ and $\mu(x,\xi)=\pi$;
	
	\smallskip
	{\rm (iv)} Let $M=\{(x_{1},x_{2},x_{3})\in \R^{3}:x_{1}^{2} +x_{2}^{2}=1\}$, $\Sigma=\{(1,0,x_{3})\in \R^{3}:x_{3}\in\R\}$, $x=(1,0,0)$, $\xi=(0,1,0)$, then $\tau_{f}(x,\xi)=\pi$, $\rho(x,\xi)=\infty$ and $\mu(x,\xi)=\pi$. If one wants $\Sigma$ to be compact, one can consider the flat square torus.
\end{rem}

\section{Comparison theorems}\label{103}

Let $f: \Sigma^{n} \rightarrow M$  be as defined in Section \ref{102}. To derive more powerful results, we introduce a gradient field on $\Sigma$, even though it will not be used in the remaining sections. Let $u$ be a $C^{2}$ function on $\Sigma$. Define a map $\Phi:T^{\perp}\Sigma\rightarrow M$ by
\begin{equation*}
	\Phi(x,y)=\mathrm{exp}_{x}\left(\nabla^{\Sigma}u(x)+y\right),
\end{equation*}
for all $(x,y)\in T^{\perp}\Sigma$. Note that $\Phi$ reduces to the normal exponential map $\mathrm{exp}^\perp$ when $u$ is constant. Prior to computing the differential of $\Phi$ at a fixed point $(\bar{x},\bar{y})\in T^{\perp}\Sigma$, we first construct a local frame field. Let $(U;x^{1},\cdots, x^{n})$ be a local coordinate system around $\bar{x}$ on $\Sigma$ and
$\{\nu_{\alpha}\}_{\alpha=n+1}^{n+m}$
be a local orthonormal frame for the normal bundle $T^{\perp}\Sigma$ over $U$. Set
\begin{equation*}
D^{\Sigma}_{\frac{\partial}{\partial x^{i}}}\frac{\partial}{\partial x^{j}}=\Gamma_{ij}^{k}\frac{\partial}{\partial x^{k}}	,\quad D^{\perp}_{\frac{\partial}{\partial x^{i}}}\nu_{\alpha}=\Gamma_{i\alpha}^{\beta}\nu_{\beta},
\end{equation*}
where $\Gamma_{ij}^{k}$ and $\Gamma_{i\alpha}^{\beta}$ are the connection coefficients of $D^{\Sigma}$ and $D^{\perp}$, respectively. Locally, a normal vector $y$ can be written as $y^{\alpha}\nu_{\alpha}$, therefore, 
\[
(\pi^{-1}(U);x^{1},\cdots, x^{n},y^{n+1},\cdots,y^{n+m})
\]
 is a local coordinate system around $(\bar{x},\bar{y})$ on normal bundle $T^{\perp}\Sigma$ with the natural frame field
\begin{equation*}
	\{X_{1},\cdots,X_{n},\partial/\partial y^{n+1},\cdots,\partial/\partial y^{n+m} \}.
\end{equation*}
A tangent vector $v$ of $\Sigma$ can be written as $v^{i}\partial/\partial x^{i}$, and thus
\begin{equation*}
	(\tilde{\pi}^{-1}(U);x^{1},\cdots, x^{n},v^{1},\cdots, v^{n},y^{n+1},\cdots,y^{n+m})
\end{equation*}
is a local coordinate system on $f^{*}TM$. Furthermore, let 
\[
\gamma(t):=\mathrm{exp}_{\bar{x}}\left(t\nabla^{\Sigma}u(\bar{x})+t\bar{y}\right),\, t\in [0,1],
\]
be the associated geodesic of $\Phi$ at $(\bar{x},\bar{y})$. For convenience, we abbreviate $e_{i}=\partial/\partial x^{i}(\bar{x})$ and $e_{\alpha}=\nu_{\alpha}(\bar{x})$ for $1\leq i\leq n$ and $n+1\leq \alpha\leq n+m$.  Let $\eta_{i}$ and $\eta_{\alpha}$ denote the parallel transports of $e_{i}$ and $e_{\alpha}$ along $\gamma$ to $\gamma(1)$ for $1\leq i\leq n$ and $n+1\leq \alpha\leq n+m$, respectively. We shall use the abbreviation
\begin{align*}
	E_{AB}:=\langle\left(\mathrm{Exp}_{f(\bar{x})}\right)_{*(\nabla^{\Sigma}u(\bar{x})+\bar{y})}\left(  e_{A}\right)  ,\eta_{B}  \rangle,
\end{align*}
for $1\leq A,B\leq n+m$. Here we have used canonically  identifying $T_{\nabla^{\Sigma}u(\bar{x})+\bar{y}}T_{f(\bar{x})}M$ with $T_{f(\bar{x})}M$. Inspired by Cordero-Erausquin, McCann and Schmuckenschläger \cite{CoMcSc}, we get the following result.
\begin{lem}\label{9}
	In the notation above, if $\gamma$ does not contain a cut point of $f(\bar{x})$, then 
\begin{align*}
	\langle \Phi_{*(\bar{x},\bar{y})}\left(X_{i}\right),\eta_{B} \rangle=&\left[\frac{1}{2} \bar{D}^{2}d_{\gamma(1)}^{2} (e_{i},e_{j})  - \langle   h(e_{i},e_{j}),\bar{y}  \rangle +D_{\Sigma}^{2}u(e_{i},e_{j}) \right]g^{jk}(\bar{x})E_{k B}\\
	&+\left[  \frac{1}{2} \bar{D}^{2} d_{\gamma(1)}^{2}(e_{i},\nu_{\alpha})+  \langle   h(e_{i},\nabla^{\Sigma}u(\bar{x})),\nu_{\alpha}  \rangle -\Gamma_{i\alpha}^{\sigma}(\bar{x})\langle \bar{y},\nu_{\sigma}\rangle   \right]\delta^{\alpha\beta}E_{\beta B}
\end{align*}
and 
\begin{align*}
	\langle \Phi_{*(\bar{x},\bar{y})}\left(\partial/\partial y^{\alpha}\right),\eta_{B} \rangle=E_{\alpha B},
\end{align*}
for $1\leq i\leq n$, $n+1\leq \alpha\leq n+m$ and $1\leq B\leq n+m$.
\end{lem}
\begin{proof}
For each $0\leq t\leq 1$, define a map $F_{t}:T^{\perp}\Sigma\rightarrow f^{*}TM$ by
\begin{equation*}
	F_{t}(x,y)=\left(x,t\nabla^{\Sigma}u(x)+ty\right)
\end{equation*}
for all $(x,y)\in T^{\perp}\Sigma$. Observe that $\Phi$ can be decomposed as $\Phi=\mathrm{exp}\circ F_{1}$. This decomposition is inspired by \cite{CoMcSc}, where an analogous construction is presented on the total space $M$: the map $\Phi$ corresponds to $\tilde{\Phi}(p):=\mathrm{Exp}_{p}(\bar{\nabla} \tilde{u}(p))$ and $F_{1}$ corresponds to $\tilde{F}(p):=(p,\bar{\nabla}\tilde{u}(p))$. In that context, the decomposition similarly holds, although it was not explicitly stated in their work. Indeed, we have the following two commutative diagrams
\begin{center}
	\begin{minipage}{0.35\textwidth}
		\centering
		\[
		\begin{tikzcd}
			T^{\perp}\Sigma \arrow[r, "F_{1}"]  \arrow[dr, "\Phi"']   & f^{*}TM \arrow[d, "\mathrm{exp}" ]\\
			& M,
		\end{tikzcd}
		\]
	\end{minipage}
	\begin{minipage}{0.35\textwidth}
		\centering
		\[
		\begin{tikzcd}
			M \arrow[r, "\tilde{F}"]  \arrow[dr, "\tilde{\Phi}"']   & TM \arrow[d, "\mathrm{Exp}" ]\\
			& M.
		\end{tikzcd}
		\]
	\end{minipage}
\end{center}

Let $c: (-\varepsilon,\varepsilon)\rightarrow T^{\perp}\Sigma$ be a curve for some $\varepsilon>0$ denoted by $c(s)=(x(s),y(s))$ such that $c(0)=(\bar{x}, \bar{y})$. Define
a variation	of $\gamma$ by $\gamma_{s}(t):=\mathrm{exp}_{x(s)}\left( t\left( \nabla^{\Sigma}u(x(s))+y(s)  \right) \right)$. Denote  by $J$ the variational vector field along $\gamma$. Clearly, $J$ is a Jacobi field since $\gamma_{s}:[0,1]\rightarrow M$ is a geodesic for each $-\varepsilon<s<\varepsilon$ and 
	\begin{equation*}
	\Phi_{*(\bar{x},\bar{y})}\left(c'(0)\right)=J(1).
\end{equation*}
Since $\Phi(\bar{x},\bar{y})$ is not the cut point of $f(\bar{x})$, the function $d_{\gamma(t)}^{2}$ is differentiable at $x(s)$ for sufficiently  small $s$ and
		\begin{equation*}
	 t\left( \nabla^{\Sigma}u(x(s))+y(s)  \right) =-\bar{\nabla}d_{\gamma(t)}^{2}(x(s))/2+\bar{\nabla}d_{\gamma(t)}^{2}(x(s))/2+t\left( \nabla^{\Sigma}u(x(s))+y(s)  \right).
	\end{equation*}
By definition of the  gradient,
			\begin{equation*}
					\begin{array}{lllll}
\bar{\nabla}d_{\gamma(t)}^{2}(x(s))&=g^{ij}(x(s))\langle
		\bar{\nabla}d_{\gamma(t)}^{2}(x(s)),\frac{\partial}{\partial x^{j}}   \rangle  \frac{\partial}{\partial x^{i}}+\delta^{\alpha\beta}\langle \bar{\nabla}d_{\gamma(t)}^{2}(x(s)),\nu_{\beta}  \rangle \nu_{\alpha}\\
		&=:2w^{i}(x(s),t)\frac{\partial}{\partial x^{i}}+2w^{\alpha}(x(s),t)\nu_{\alpha}
			\end{array}
		\end{equation*}
and 
			\begin{equation*}
 \nabla^{\Sigma}u(x(s))=g^{ij}(x(s))\frac{\partial u}{\partial x^{j}}(x(s))\frac{\partial }{\partial x^{i}}=:z^{i}(x(s))\frac{\partial}{\partial x^{i}}.
\end{equation*}
	Thus, the coordinate of $F_{t}(x(s),y(s))\in 	\tilde{\pi}^{-1}(U)\subset f^{*}TM$ is
				\begin{align*}
	&\left(x^{1}(s),\cdots,  x^{n}(s), -w^{1}(x(s),t)+w^{1}(x(s),t)+ tz^{1}(x(s)),\cdots,\right.\\ &-w^{n}(x(s),t)+w^{n}(x(s),t)+ tz^{n}(x(s)),
 -w^{n+1}(x(s),t)+w^{n+1}(x(s),t) +ty^{n+1}(s),\cdots,\\
 &\left.-w^{n+m}(x(s),t)+w^{n+m}(x(s),t) +ty^{n+m}(s)                           \right).
\end{align*}
	By using of the chain rule, 
					\begin{align*}
J(t)&=\mathrm{exp}_{*(\bar{x},t\nabla^{\Sigma}u(\bar{x})+t\bar{y})}\circ \left(F_{t}\right)_{*(\bar{x},\bar{y})}\left(c'(0)\right)\\
&=\mathrm{exp}_{*(\bar{x},t\nabla^{\Sigma}u(\bar{x})+t\bar{y})}\left(   \frac{dx^{i}}{ds}(0)\frac{\partial}{\partial x^{i}} +\left( -\frac{\partial w^{i}(x(s),t)}{\partial s}+\frac{\partial w^{i}(x(s),t)}{\partial s}\right.\right.\\
&\ \ \ \left.+\left.t\frac{dz^{i}(x(s))}{ds}   \right)\Big|_{s=0}\frac{\partial}{\partial v^{i}}       +   \left( -\frac{\partial w^{\alpha}(x(s),t)}{\partial s}+\frac{\partial w^{\alpha}(x(s),t)}{\partial s} +t\frac{dy^{\alpha}(s)}{ds}   \right)\Big|_{s=0}\frac{\partial}{\partial y^{\alpha}}   \right).
	\end{align*}
	Note that, for all $s$ small enough,  
				\begin{equation*}
	\mathrm{exp}_{x(s)}\left(  -\bar{\nabla}d_{\gamma(t)}^{2}(x(s))/2\right)=\gamma(t),
	\end{equation*} which yields
					\begin{align*}
\mathrm{exp}_{*(\bar{x},t\nabla^{\Sigma}u(\bar{x})+t\bar{y})}\left(   \frac{dx^{i}}{ds}(0)\frac{\partial}{\partial x^{i}} -\frac{\partial w^{i}(x(s),t)}{\partial s}  \Big|_{s=0}\frac{\partial}{\partial v^{i}}  -   \frac{\partial w^{\alpha}(x(s),t)}{\partial s}   \Big|_{s=0}\frac{\partial}{\partial y^{\alpha}}   \right)=0.
	\end{align*}
 Then  the representation of $J$ can be reduced as 
						\begin{align*}
			J(t)&=\mathrm{exp}_{*(\bar{x},t\nabla^{\Sigma}u(\bar{x})+t\bar{y})}\left(   \left( \frac{\partial w^{i}(x(s),t)}{\partial s} +t\frac{dz^{i}(x(s))}{ds}   \right)\Big|_{s=0}\frac{\partial}{\partial v^{i}}       \right.\\
			&\ \ \ +  \left.\left( \frac{\partial w^{\alpha}(x(s),t)}{\partial s} +t\frac{dy^{\alpha}(s)}{ds}   \right)\Big|_{s=0}\frac{\partial}{\partial y^{\alpha}}   \right)\\
			&= \left( \frac{\partial w^{i}(x(s),t)}{\partial s} +t\frac{dz^{i}(x(s))}{ds}   \right)\Big|_{s=0}\left(\mathrm{Exp}_{f(\bar{x})}\right)_{*(t\nabla^{\Sigma}u(\bar{x})+t\bar{y})}\left(  \frac{\partial}{\partial x^{i}}\right)\\
			&\ \ \ + \left( \frac{\partial w^{\alpha}(x(s),t)}{ds} +t\frac{dy^{\alpha}(s)}{ds}   \right)\Big|_{s=0}\left(\mathrm{Exp}_{f(\bar{x})}\right)_{*(t\nabla^{\Sigma}u(\bar{x})+t\bar{y})}\left(  \nu_{\alpha}\right).
	\end{align*}
	 Now for each $1\leq A\leq n+m$, we designate the curve $c$  as the $A$-th coordinate curve, denoted by $c_{A}$. The corresponding Jacobi field is denoted by $J_{A}$. When $1\leq i\leq n$, the coordinate of 
	$c_{i}(s)$ is
							\begin{equation*}
\left( \bar{x}^{1},\cdots, \bar{x}^{i-1},\bar{x}^{i}+s,\bar{x}^{i+1},\cdots,\bar{x}^{n},\bar{y}^{n+1},\cdots,\bar{y}^{n+m}           \right).
		\end{equation*}
When $n+1\leq \alpha\leq n+m$, the coordinate of 
$c_{\alpha}(s)$ is
\begin{equation*}
	\left( \bar{x}^{1},\cdots,\bar{x}^{n},\bar{y}^{n+1},\cdots,\bar{y}^{\alpha-1},\bar{y}^{\alpha}+s,\bar{y}^{\alpha+1},\cdots,\bar{y}^{n+m}          \right).
\end{equation*}
 Consequently, for $1\leq i \leq n$, $1\leq B\leq n+m$, 
\begin{align*}
		&\ \ \ \langle \Phi_{*(\bar{x},\bar{y})}\left(X_{i}\right),\eta_{B} \rangle=\langle J_{i}(1),\eta_{B}  \rangle\\
		&=\frac{d}{ds}\Big|_{s=0}\left[g^{jk}(x(s))\langle
		\bar{\nabla}d_{\gamma(1)}^{2}(x(s))/2+\nabla^{\Sigma}u(x(s)),\partial/\partial x^{j}  \rangle   \right]E_{k B}\\
		&\ \ \ +\frac{d}{ds}\Big|_{s=0}\left[\delta^{\alpha\beta} \langle \bar{\nabla}d_{\gamma(1)}^{2}(x(s))/2,\nu_{\alpha}  \rangle +y^{\beta}(s)\right]E_{\beta B}\\
	&=\left[\langle 	\bar{D}_{e_{i}}	\bar{\nabla}d_{\gamma(1)}^{2}/2,e_{j}  \rangle +  \langle -\nabla^{\Sigma}u(\bar{x})-\bar{y} ,\Gamma_{ij}^{l}(\bar{x})\partial/\partial x^{l}+h(e_{i},e_{j})   \rangle \right.\\
	 &\ \ \ \left.+\langle  D_{e_{i}}^{\Sigma}\nabla^{\Sigma} u,e_{j}  \rangle +\langle \nabla^{\Sigma}u(\bar{x}),  \Gamma_{ij}^{l}(\bar{x})\partial/\partial x^{l}  \rangle    \right]  g^{jk}(\bar{x})E_{k B}\\
	 &\ \ \ +\left[ \langle \bar{D}_{e_{i}}\bar{\nabla}d_{\gamma(1)}^{2}/2,\nu_{\alpha}  \rangle  +  \langle -\nabla^{\Sigma}u(\bar{x})-\bar{y}, -\langle h(e_{i},e_{j}),\nu_{\alpha} \rangle g^{jk}(\bar{x})e_{k} +\Gamma_{i\alpha}^{\sigma}(\bar{x})\nu_{\sigma}  \rangle   \right]\delta^{\alpha\beta}E_{\beta B}\\
	 &=\left[ \bar{D}^{2}d_{\gamma(1)}^{2} (e_{i},e_{j})/2  - \langle   h(e_{i},e_{j}),\bar{y}  \rangle +D_{\Sigma}^{2}u(e_{i},e_{j}) \right]g^{jk}(\bar{x})E_{k B}\\
	 &\ \ \ +\left[   \bar{D}^{2} d_{\gamma(1)}^{2}(e_{i},\nu_{\alpha})/2+  \langle   h(e_{i},\nabla^{\Sigma}u(\bar{x})),\nu_{\alpha}  \rangle -\Gamma_{i\alpha}^{\sigma}(\bar{x})\langle \bar{y},\nu_{\sigma}\rangle   \right]\delta^{\alpha\beta}E_{\beta B},
\end{align*}
 and for $n+1\leq \alpha \leq n+m$, $1\leq B\leq n+m$, 
\begin{align*}
&\ \ \ \langle \Phi_{*(\bar{x},\bar{y})}\left(\partial/\partial y^{\alpha}\right),\eta_{B} \rangle=\langle J_{\alpha}(1),\eta_{B}  \rangle\\
&= \frac{d}{ds}\Big|_{s=0}  \left[   \langle
\bar{\nabla}d_{\gamma(1)}^{2}(x(s))/2+\nabla^{\Sigma}u(x(s)),\partial/\partial x^{j}   \rangle g^{jk}(x(s)) \right]E_{k B}\\
&\ \ \ +\frac{d}{ds}\Big|_{s=0}\left[\delta^{\alpha\beta} \langle \bar{\nabla}d_{\gamma(1)}^{2}(x(s))/2,\nu_{\alpha}  \rangle +y^{\beta}(s)\right]E_{\beta B}\\
&=E_{\alpha B}.
\end{align*}
	These complete the proof of the lemma.
\end{proof}
While the differential of $\Phi$ at a point has been computed, the calculation of its Jacobian determinant necessitates understanding the metric structure of the normal bundle. For this purpose, we recall the canonical metric on this bundle induced by the metric $\bar{g}$ of $M$ and the normal connection $D^{\perp}$ of $T^{\perp}\Sigma$. In fact, the vector $\partial/\partial x^{i}\in T_{\bar{x}}\Sigma$ can be uniquely horizontally lifted to a tangent vector $Z_{i}(\bar{x},\bar{y})$ in $T_{(\bar{x},\bar{y})}T^{\perp}\Sigma$ for each $1\leq i\leq n$ (with respect to $D^{\perp}$). Moreover, $\{Z_{i}(\bar{x},\bar{y})\}_{i=1}^{n}$ forms a basis of an $n$-dimensional linear subspace of $T_{(\bar{x},\bar{y})}T^{\perp}\Sigma$ which is linearly isomorphic to  $T_{(\bar{x},\bar{y})}\Sigma$, known as the horizontal tangent subspace at $(\bar{x},\bar{y})$ and denoted by $\mathcal{H}_{(\bar{x},\bar{y})}$. Therefore, the whole space $T_{(\bar{x},\bar{y})}T^{\perp}\Sigma$ can be decomposed into the sum of $\mathcal{H}_{(\bar{x},\bar{y})}$ and the kernel of $\pi_{*(\bar{x},\bar{y})}$ which is canonically linearly isomorphic to the fiber $T^{\perp}_{\bar{x}}\Sigma$, called the vertical tangent subspace at $(\bar{x},\bar{y})$ and denoted by $\mathcal{V}_{(\bar{x},\bar{y})}$. Consequently, the tangent bundle of $T^{\perp}\Sigma$ can be decomposed as
\begin{align*}
	TT^{\perp}\Sigma=\mathcal{H}\oplus\mathcal{V}
\end{align*}
in the sense of the Whitney sum. With this understanding, one can define the metric on $T^{\perp}\Sigma$ as follows:
\begin{align}\label{10}
	\langle X,Y \rangle:=	\langle \pi_{*(\bar{x},\bar{y})}(X),\pi_{*(\bar{x},\bar{y})}(Y) \rangle _{\bar{g}}+\langle X^{\mathcal{V}},Y^{\mathcal{V}} \rangle _{\bar{g}},
\end{align}
for $X,Y\in T_{(\bar{x},\bar{y})}T^{\perp}\Sigma$, where $(\cdot)^{\mathcal{V}}$ denotes the projection to $\mathcal{V}$. In general, the vector field $Z_{i}$ differs from the natural frame field $X_{i}$ constructed earlier. Actually, the definition of the horizontal lifting immediately yields
\begin{align}\label{11}
Z_{i}(x,y)=X_{i}(x,y)-y^{\beta}\Gamma^{\alpha}_{i\beta}(x)\frac{\partial}{\partial y^{\alpha}},
\end{align}
for $1\leq i\leq n$. By \eqref{10} and \eqref{11},
\begin{align*}
	\langle X_{i}(x,y),\partial/\partial y^{\alpha} \rangle&=y^{\sigma}\Gamma_{i\sigma}^{\alpha}(x),\,	\langle \partial/\partial y^{\alpha},\partial/\partial y^{\beta}\rangle=\delta_{\alpha\beta},\\
	\langle X_{i}(x,y),X_{j}(x,y) \rangle&=	\langle \pi_{*(\bar{x},\bar{y})}(X_{i}),\pi_{*(\bar{x},\bar{y})}(X_{j}) \rangle+\langle y^{\beta}\Gamma^{\alpha}_{i\beta}(x)\frac{\partial}{\partial y^{\alpha}},y^{\sigma}\Gamma^{\tau}_{j\sigma}(x)\frac{\partial}{\partial y^{\tau}} \rangle \\
	&=g_{ij}(x)+\delta_{\alpha\tau}y^{\beta}\Gamma^{\alpha}_{i\beta}(x)y^{\sigma}\Gamma^{\tau}_{j\sigma}(x),
\end{align*}
for $1\leq i\leq n$, $n+1\leq \alpha,\beta\leq n+m$ and $(x,y)\in \pi^{-1}(U)$. A direct computation yields the local expression for the metric $g_{T^{\perp}\Sigma}$ of $T^{\perp}\Sigma$,
\begin{align}\label{13}
	g_{T^{\perp}\Sigma}(x,y)=\delta_{\alpha\beta}( y^{\sigma}\Gamma_{i \sigma}^{\alpha}(x)dx^{i}+dy^{\alpha})\otimes( y^{\tau}\Gamma_{j \tau}^{\beta}(x) dx^{j}  +dy^{\beta})+g_{\Sigma}(x),
\end{align}
for $(x,y)\in \pi^{-1}(U)$.

 From now on, we shall impose the following convention on the frame fields previously introduced. By an appropriate choice, we can assume that $(U;x^{1},\cdots, x^{n})$ is a local normal coordinate chart around $\bar{x}$ on $\Sigma$, and the normal frame $\{\nu_{\alpha}\}_{\alpha=n+1}^{n+m}$ satisfies 
\begin{align*}
D^{\perp}\nu_{\alpha}=0
\end{align*}
at $\bar{x}$ for $n+1\leq\alpha\leq n+m$. Thus, at the point $(\bar{x},y)\in \pi^{-1}(U)$, expression \eqref{13} can be reduced to 
\begin{align*}
	g_{T^{\perp}\Sigma}(\bar{x},y)=\sum_{i=1}^{n}dx^{i}\otimes dx^{i}+\sum_{\alpha=n+1}^{n+m}dy^{\alpha}\otimes dy^{\alpha}.
\end{align*}
Moreover the natural frame field
\begin{equation*}
	\{X_{1},\cdots,X_{n},\partial/\partial y_{n+1},\cdots,\partial/\partial y_{n+m} \}
\end{equation*}
forms an orthonormal basis of $T_{(\bar{x},y)}T^{\perp}\Sigma$ at $(\bar{x},y)$ for each $y\in T_{\bar{x}}^{\perp}\Sigma$. For convenience, we abbreviate 
\begin{align*}
	P_{iB}&:=\langle \Phi_{*(\bar{x},\bar{y})}\left(X_{i}\right),\eta_{B} \rangle,\, P_{\alpha B}:=\langle \Phi_{*(\bar{x},\bar{y})}\left(\partial/\partial y^{\alpha}\right),\eta_{B} \rangle,\\
	Q&:=\frac{1}{2}\bar{D}^{2}d_{\gamma(1)}^{2}(f(\bar{x})) \Big|_{T_{\bar{x}}\Sigma \times T_{\bar{x}}\Sigma} - \langle   h(\bar{x}),\bar{y}  \rangle  +D_{\Sigma}^{2}u(\bar{x}),\\
	S_{i \alpha}&:=\frac{1}{2}\bar{D}^{2} d_{\gamma(1)}^{2}(e_{i},\nu_{\alpha})+  \langle   h(e_{i},\nabla^{\Sigma}u(\bar{x})),\nu_{\alpha}  \rangle +\Gamma_{i\alpha}^{\sigma}\langle \bar{y},\nu_{\sigma}\rangle,
\end{align*}
for $1\leq i\leq n$, $n+1\leq \alpha\leq n+m$ and $1\leq B\leq n+m$.

The following formulas extend the well-established Euclidean results (cf. \cite{Ca}, p. 232; see also \cite{Bre2021}, Lemma 5) to the setting of general Riemannian manifolds.
\begin{lem}\label{4}
	In the notation above, if $\gamma$ does not contain a cut point of $f(\bar{x})$, we obtain a frame-dependent matrix equation of the form
	\begin{align*}
P_{(n+m)\times(n+m)}=\begin{pmatrix}
	Q_{n\times n} & S_{n\times m} \\
	O_{m\times n} & I_{m}
\end{pmatrix}\cdot E_{(n+m)\times(n+m)},
	\end{align*}
where $O_{m \times n}$ and $I_m$ denote the $m \times n$ zero matrix and the $m \times m$ identity matrix, respectively. Moreover, the following frame-independent identity holds:
	\begin{equation*}
	|\mathrm{det}\Phi_{*(\bar{x},\bar{y})}|=	|\mathrm{det}\left(\mathrm{Exp}_{f(\bar{x})}\right)_{*(\nabla^{\Sigma}u(\bar{x})+\bar{y})}|\cdot|\mathrm{det}Q|.
\end{equation*}
\end{lem}
\begin{proof}
The matrix equation is readily derived via the construction of the frame fields and Lemma \ref{9}, and the second identity follows straightforwardly from this equation. This completes the proof of the lemma.
\end{proof}

For the rest of this section, we take $u$ to be constant and focus exclusively on $\Phi=\exp^{\perp}$, with all previous results remaining valid. We proceed to demonstrate some applications of Lemma \ref{4}.
\begin{prop}
	Let $M$ be an $(n+m)$-dimensional complete Riemannian manifold with infinite injectivity radius and $f: \Sigma^{n} \rightarrow M$ be an isometric immersion. Let $(x,\xi)\in U\Sigma$ and $\sigma(t):=\mathrm{exp}^{\perp}_{x}\left(t\xi\right)$ for $t\in\R$ be a geodesic. Then $\sigma(t)$ is a focal point of $\Sigma$ along $\sigma$ if and only if the tensor
	\begin{align*}
	\frac{1}{2}\bar{D}^{2}d_{\sigma(t)}^{2}(f(x))\Big|_{T_{x}\Sigma \times T_{x}\Sigma} -t\langle h(x),\xi \rangle:T_{x}\Sigma \times T_{x}\Sigma\rightarrow\R,
	\end{align*}
	viewed as a linear transformation on $T_{x}\Sigma$, degenerates.
\end{prop}
\begin{proof}
By assumption, $(\mathrm{Exp}_{p})_{*w}$ is non-singular for each $(p,w)\in TM$. Therefore the lemma follows by the definition of focal point and Lemma \ref{4}.
\end{proof}

When $f: \Sigma \rightarrow M$ is merely an immersion, the function $\tau_{f}$ is generally not continuous and may vanish. This poses substantial difficulties in studying the equality case of geometric inequalities for immersed submanifolds. To circumvent these issues, we are inspired by the definition of the Hausdorff measure to employ a localization technique. This allows us to define a modified version of $\tau_{f}$ that is independent of the global geometry of $\Sigma$. Indeed, for all sufficiently small positive number $r$, the restriction $f|_{B^{\Sigma}_{x}(r)}$  is an embedding. On the one hand, for each $(x,\xi)\in U\Sigma$, $ \tau_{f|_{B^{\Sigma}_{x}(r)}}(x,\xi)\in(0,\infty]$ holds for  all small positive number $r$. On the other hand, the function $\tau_{f|_{B^{\Sigma}_{x}(r)}}(x,\xi)$ is non-decreasing in $r$. We now define the modified cut distance function $\tilde{\tau}_{f}:U\Sigma\rightarrow (0,\infty]$  of $\Sigma$  in direction $(x,\xi)\in U\Sigma$ by
\begin{equation*}
\tilde{\tau}_{f}(x,\xi):=\lim_{r\rightarrow 0^{+}}\tau_{f|_{B^{\Sigma}_{x}(r)}}(x,\xi),
\end{equation*}
with the understanding that the limit is said to exist even if the sequence diverges to $+\infty$. Based on Lemma \ref{4}, the following result establishes the connection between $\tilde{\tau}_{f}$, $\mu$ and $\rho$.

Let $(x,\xi)\in U\Sigma$ and define the geodesic $\sigma(t):=\mathrm{exp}^{\perp}_{x}\left(t\xi\right)$ for  $t\in[0,\mu(f(x),\xi))$. For convenience, denote by $T_{t}$ the tensor
\begin{align*}
	\frac{1}{2}\bar{D}^{2}d_{\sigma(t)}^{2}(f(x)) \Big|_{T_{x}\Sigma \times T_{x}\Sigma} -t\langle h(x),\xi \rangle :T_{x}\Sigma \times T_{x}\Sigma\rightarrow\R
\end{align*}
and $A_{t}$ the tensor
\begin{align*}
	\bar{D}^{2}d_{\sigma(t)}(f(x)) \Big|_{T_{x}\Sigma \times T_{x}\Sigma} -\langle h(x),\xi \rangle :T_{x}\Sigma \times T_{x}\Sigma\rightarrow\R.
\end{align*}

\begin{lem}\label{34}
	Let $M$ be a complete Riemannian manifold  and $f: \Sigma \rightarrow M$ an isometric immersion. Let $(x,\xi)\in U\Sigma$ and define the geodesic $\sigma(t):=\mathrm{exp}^{\perp}_{x}\left(t\xi\right)$ for $0\leq t < \infty$. Then the following two assertions hold.

\smallskip	
{\rm (i)} The tensor $T_{t}$
is positive-definite for all $t\in[0,\min\{\mu(f(x),\xi),\rho(x,\xi)\})$.

\smallskip
{\rm (ii)} The function $\tilde{\tau}_{f}$ satisfies $\tilde{\tau}_{f}(x,\xi)=\min\{\mu(f(x),\xi),\rho(x,\xi)\}$ hence is continuous on which $\tilde{\tau}_{f}$ is finite.
\end{lem}
\begin{proof}
	We  prove these two assertions separately.
	
	1\textdegree. Since $\left(\mathrm{Exp}_{f(x)}\right)_{*t\xi}$ is
	non-singular for all $t\in[0,\mu(f(x),\xi))$, by Lemma \ref{4} and the definition of $\rho$ we know that $T_{t}$  is
	non-singular for all $t\in[0,\min\{\mu(f(x),\xi),\rho(x,\xi)\})$. Note that $T_{0}=g_{\Sigma}$ is positive-definite. (i) follows immediately. 
	
2\textdegree. By the definition of $\tilde{\tau}_{f}$ and the inequalities \eqref{5} and \eqref{3}, we obtain, for each $r>0$,
\begin{equation}\label{35}
	\tau_{f|_{B^{\Sigma}_{x}(r)}}(x,\xi)\leq \tilde{\tau}_{f}(x,\xi) \leq \min\{\mu(f(x),\xi),\rho(x,\xi)\}.
\end{equation}
 We argue by contradiction. Suppose that there exists a positive number $t_{0}$ such that
 \[
\tilde{\tau}_{f}(x,\xi)<t_{0}<\min\{\mu(f(x),\xi),\rho(x,\xi)\}.
 \]
 Set $\varphi:=\frac{1}{2}d_{\sigma(t_{0})}^{2}$ on $M$. Moreover, for each $r>0$, denote by $\psi_{r}$ the function $\varphi \circ f|_{B^{\Sigma}_{x}(r)}$. Since $t_{0}<\mu(f(x),\xi)$,
\begin{equation}\label{36}
\nabla^{\Sigma}\psi_{r}(x)=[	\bar{\nabla}\varphi(f(x))]^{\mathrm{tan}}=(-t_{0}\xi)^{\mathrm{tan}}=\vec{0},
\end{equation}
where $(\cdot)^{\mathrm{tan}}$ denotes the projection to $T_{x}\Sigma$. Given $X,Y\in \Gamma(TB^{\Sigma}_{x}(r))$, at $x$ we have
\begin{equation}\label{37}
	\begin{aligned}
		D_{\Sigma}^{2}\psi_{r}(X,Y)&=\langle D_{X}^{\Sigma}\nabla^{\Sigma}\psi_{r},Y \rangle=\langle \bar{D}_{X}(\bar{\nabla}\varphi-(\bar{\nabla}\varphi)^{\perp}  ),Y \rangle\\
		&=\bar{D}^{2}\varphi(X,Y)+\langle (\bar{\nabla}\varphi)^{\perp}  ,h(X,Y) \rangle\\
		&=\bar{D}^{2}\varphi(X,Y)-t_{0}\langle  h(X,Y),\xi  \rangle\\
		&=T_{t_{0}}(X,Y),
	\end{aligned}
\end{equation}
where $(\cdot)^{\perp}$ denotes the projection to $T^{\perp}\Sigma$ and we have used the convention of identifying $f_{*}(X)$ with $X$ and $f_{*}(Y)$ with $Y$. By \eqref{36}, \eqref{37} and assertion (i), one can conclude that $x$ must be a strict local minimizer for the function $\psi_{r}$. Thus, there exists a positive number $r_{0}$ such that  $\psi_{r_{0}}$ attains its minimum at $x$. Immediately, we have $t_{0}\leq 	\tau_{f|_{B^{\Sigma}_{x}(r_{0})}}(x,\xi)$ which contradicts \eqref{35}. Therefore, we have $\tilde{\tau}_{f}(x,\xi)=\min\{\mu(f(x),\xi),\rho(x,\xi)\}$. The arbitrariness of $(x,\xi)$, together with the continuity of $\mu$ and $\rho$, ensures the continuity of $\tilde{\tau}_{f}$ on which $\tilde{\tau}_{f}$ is finite, as claimed.

These complete the proof of Lemma \ref{34}.
\end{proof}

We will use the following results, proved by Bishop.
\begin{lem}[Bishop, cf. \cite{BiCr}, in the proof of Theorem 15, pp.~253--255]\label{39}
	Let $M$ be a complete Riemannian $N$-manifold. Fix a point  $(p,v)\in S M$ and let $\sigma(t):=\mathrm{Exp}_{p}\left(tv\right)$ for $t\in[0,\mu(p,v))$ be a minimal geodesic.
	For a positive integer $l$ satisfying $1\leq l\leq N-1$, consider an orthonormal family of vectors $w_{1},\cdots,w_{l}$
	in $T_{p}M$ perpendicular to $v$. Then the following two assertions hold.
	
	\smallskip
	{\rm (i)} If $M$ satisfies $\mathrm{Ric}^{M}_{l}\geq l\delta$, then 	for $0<t<\mu(p,v)$
	\begin{align}\label{40}
		\sum_{i=1}^{l}\bar{D}^{2}d_{\sigma(t)}(w_{i},w_{i})\leq \frac{l \mathfrak{c}_{\delta}(t)}{\mathfrak{s}_{\delta}(t)}.
	\end{align}

	\smallskip
	{\rm (ii)} If the sectional curvatures of $M$ are bounded above by $\delta$, then 	for $0<t<\mu(p,v)$
	\begin{align*}
		\bar{D}^{2}d_{\sigma(t)}(w_{l},w_{l})\geq \frac{ \mathfrak{c}_{\delta}(t)}{\mathfrak{s}_{\delta}(t)}.
	\end{align*}

\end{lem}

\begin{lem}[Bishop,  cf. \cite{BiCr}, Theorem 15, p. 253]\label{16}
	Let $M$ be a complete Riemannian $N$-manifold and fix  $(p,v)\in S M$. Consider the geodesic $\sigma(t):=\mathrm{Exp}_{p}\left(tv\right)$ for $t\in[0,r]$, on which $p$ has no conjugate points. For a positive integer $l$ satisfying $1\leq l\leq N-1$, take linearly independent vectors  $w_{1},\cdots,w_{l}$
	 in $T_{p}M$ perpendicular to $v$. Define the Jacobi fields $Y_{i}(t):=(\mathrm{Exp}_{p})_{*tv}(tw_{i})$ along the geodesic $\sigma$, for $1\leq i\leq l$. Then the following two assertions hold.

\smallskip
{\rm (i)} If $M$ satisfies $\mathrm{Ric}^{M}_{l}\geq l\delta$, then
\begin{align*}
	\frac{
		|Y_{1}(t)\wedge Y_{2}(t)\wedge \cdots \wedge Y_{l}(t)|}{|w_{1}\wedge w_{2}\wedge \cdots \wedge w_{l}|}\leq \mathfrak{s}_{\delta}(t)^{l}	,
\end{align*}
for $t\in (0,r]$.

\smallskip
{\rm (ii)} If the sectional curvatures of $M$ are bounded above by $\delta$, then
\begin{align*}
	\frac{
		|Y_{1}(t)\wedge Y_{2}(t)\wedge\cdots \wedge Y_{l}(t)|}{|w_{1}\wedge w_{2}\wedge\cdots \wedge w_{l}|}\geq \mathfrak{s}_{\delta}(t)^{l}	,
\end{align*}
for $t\in (0,r]$.
\end{lem}

Before stating our main comparison theorems, we first use Bishop's result to prove a byproduct: establishing the monotonicity in Hessian/Laplacian comparison theorem.
\begin{theorem}\label{17}
	Let $M$ be a complete Riemannian $N$-manifold. Fix a point  $(p,v)\in S M$ and let $\sigma(t):=\mathrm{Exp}_{p}\left(tv\right)$ for $t\in[0,\mu(p,v))$ be a minimal geodesic.
For a positive integer $l$ satisfying $1\leq l\leq N-1$, consider an orthonormal family of vectors $w_{1},\cdots,w_{l}$
 in $T_{p}M$ perpendicular to $v$. Then the following two assertions hold.
	
	\smallskip
	{\rm (i)} If $M$ satisfies $\mathrm{Ric}^{M}_{l}\geq l\delta$, then 
	\begin{align}
&\frac{d}{dt}	\sum_{i=1}^{l}\bar{D}^{2}d_{\sigma(t)}(w_{i},w_{i})\leq \frac{d}{dt}\frac{l \mathfrak{c}_{\delta}(t)}{\mathfrak{s}_{\delta}(t)}<0,\label{41}\\
&\frac{\sum_{i=1}^{l}\bar{D}^{2}d_{\sigma(t)}(w_{i},w_{i})}{l \mathfrak{c}_{\delta}(t)/\mathfrak{s}_{\delta}(t)}\leq \frac{\sum_{i=1}^{l}\bar{D}^{2}d_{\sigma(s)}(w_{i},w_{i})}{l \mathfrak{c}_{\delta}(s)/\mathfrak{s}_{\delta}(s)},\label{42}
	\end{align}
	for $0<s\leq t<\mu(p,v)$.
	
	\smallskip
	{\rm (ii)} If the sectional curvatures of $M$ are bounded above by $\delta$, then
	\begin{align*}
	&0>\frac{d}{dt}	\bar{D}^{2}d_{\sigma(t)}(w_{l},w_{l})\geq \frac{d}{dt}\frac{ \mathfrak{c}_{\delta}(t)}{\mathfrak{s}_{\delta}(t)},\\
	&\frac{\bar{D}^{2}d_{\sigma(t)}(w_{l},w_{l})}{l \mathfrak{c}_{\delta}(t)/\mathfrak{s}_{\delta}(t)}\geq \frac{\bar{D}^{2}d_{\sigma(s)}(w_{l},w_{l})}{l \mathfrak{c}_{\delta}(s)/\mathfrak{s}_{\delta}(s)},
\end{align*}
for $0<s\leq t<\mu(p,v)$.
\end{theorem}
\begin{proof}
Let $\{\bar{E}_{j}\}_{j=1}^{N}$ be an orthonormal basis of $T_{p}M$ such that $\bar{E}_{N}=v$ and  $\bar{E}_{i}=w_{i}$ for $1\leq i\leq l$, and $\{E_{j}(s)\}_{j=1}^{N}$ be the parallel transports of $\{\bar{E}_{j}\}_{j=1}^{N}$ along $\sigma$ to $\sigma(s)$ for $0\leq s < \mu(p,v)$. Since there are no conjugate points to $p$ along $\sigma$, for each $0\leq t < \mu(p,v)$ and each $1\leq i\leq l$, there exists a unique normal Jacobi field $J^{t}_{i}$ along $\sigma$ such that
	\begin{align}\label{19}
	J^{t}_{i}(0):=J^{t}_{i}(\sigma(0))=w_{i}
\end{align}
 and
 	\begin{align}\label{20}
 J^{t}_{i}(t):=J^{t}_{i}(\sigma(t))=0.
 \end{align}
 Express $J^{t}_{i}$ in the basis $\{E_{j}\}_{j=1}^{N}$ as
	\begin{align}\label{24}
J^{t}_{i}(s)=\sum_{j=1}^{N-1}\lambda_{ij}(t,s)E_{j}(s)
\end{align}
for $1\leq i\leq l$, $t\in (0,\mu(p,v))$
 and  $s\in [0,\mu(p,v))$. We claim that
 \[
\lambda_{ij}\in C^{\infty}( (0,\mu(p,v))\times [0,\mu(p,v))).
 \]
 Indeed, we can define normal Jacobi fields $\varphi_{1},\cdots,\varphi_{N-1},\psi_{1},\cdots,\psi_{N-1}$  along $\sigma$ with initial value conditions
	\begin{align}\label{21}
	\varphi_{j}(0)=\bar{E}_{j},\quad\bar{D}_{v}	\varphi_{j}=0 ,\quad\psi_{j}(0)=0,\quad\bar{D}_{v}\psi_{j}=\bar{E}_{j},
\end{align}
for $1\leq j\leq N-1$. Note that $\{\varphi_{1},\cdots,\varphi_{N-1},\psi_{1},\cdots,\psi_{N-1}\}$ forms a basis solution system of all normal Jacobi fields along $\sigma$. Therefore, for each $0<t<\mu(p,v)$ there exist $a_{ij}(t)$ and $b_{ij}(t)$ such that 
	\begin{align}\label{18}
	J^{t}_{i}(s)=\sum_{j=1}^{N-1}(a_{ij}(t)\varphi_{j}(s)+b_{ij}(t)\psi_{j}(s)),
\end{align}
for all $0\leq s<\mu(p,v)$. On the one hand, by \eqref{19}, \eqref{21} and \eqref{18}, 
	\begin{align*}
\sum_{j=1}^{N-1}a_{ij}(t)\bar{E}_{j}=\bar{E}_{i}
\end{align*}
which yields $a_{ij}(t)=\delta_{ij}$, for $1\leq i\leq l$, $1\leq j\leq N-1$ and $t\in (0,\mu(p,v))$. Therefore 
	\begin{align}\label{22}
	J^{t}_{i}(s)=\varphi_{i}(s)+\sum_{j=1}^{N-1}b_{ij}(t)\psi_{j}(s).
\end{align}
On the other hand, by \eqref{20}, \eqref{21} and \eqref{18},
	\begin{align}\label{23}
	\varphi_{i}(t)+\sum_{j=1}^{N-1}b_{ij}(t)\psi_{j}(t)=0.
\end{align}
It follows from the construction of $\psi_{j}$ that
the vectors $\psi_{1}(s),\cdots,\psi_{N-1}(s)$ form a basis for the orthogonal complement of $E_{N}(s)$ in $T_{\sigma(s)}M$ for each $0<s<\mu(p,v)$. Thus, equation \eqref{23} implies that $b_{ij}(t)$ is precisely the $j$-th component of the vector $-\varphi_{i}(t)$ with respect to the basis $\{\psi_{j}(t)\}_{j=1}^{N-1}$,  and hence $b_{ij}\in C^{\infty}(0,\mu(p,v))$ which together with   \eqref{22} implies $\lambda_{ij}\in C^{\infty}( (0,\mu(p,v))\times [0,\mu(p,v)))$.
 
Differentiating \eqref{24} gives us
\begin{align}
	&\bar{D}_{\sigma'(s)}J_{i}^{t}=\sum_{j=1}^{N-1}\frac{\partial\lambda_{ij}}{\partial s}(t,s)E_{j}(s),\label{25}\\
	&\frac{d}{dt}J_{i}^{t}(s)=\sum_{j=1}^{N-1}\frac{\partial\lambda_{ij}}{\partial t}(t,s)E_{j}(s)\label{26},
\end{align}
for $0<t<\mu(p,v)$ and $0\leq s<\mu(p,v)$. These together with
\eqref{24} can yield
\begin{equation}\label{29}
	\begin{aligned}
\frac{\partial}{\partial t}\langle \bar{D}_{\sigma'(s)}J_{i}^{t},\bar{D}_{\sigma'(s)}J_{i}^{t} \rangle&=\frac{\partial}{\partial t}\left( \sum_{j=1}^{N-1}\frac{\partial\lambda_{ij}}{\partial s}(t,s)^{2} \right)\\
&=2\sum_{j=1}^{N-1}\frac{\partial\lambda_{ij}}{\partial s}(t,s)\frac{\partial^{2}\lambda_{ij}}{\partial s \partial t}(t,s)\\
&=2\langle \bar{D}_{\sigma'(s)}J_{i}^{t},\bar{D}_{\sigma'(s)}\frac{d}{dt}J_{i}^{t} \rangle,
	\end{aligned}
\end{equation}
 for $1\leq i\leq l$, $t\in (0,\mu(p,v))$ and $s\in [0,\mu(p,v))$. Since $J_{i}^{t}(0)\equiv w_{i}$ for all $t\in (0,\mu(p,v))$, it follows that
\begin{align}\label{27}
	\frac{d}{dt}J_{i}^{t}(0)=0.
\end{align}
Furthermore, note that $J_{i}^{t}(t)=0$  for all $t\in (0,\mu(p,v))$, which implies $\lambda_{ij}(t,t)\equiv 0$. Differentiating this with respect to $t$ along the line $s=t$ yields
\begin{align*}
	\frac{\partial \lambda_{ij}}{\partial t}(t,t)+\frac{\partial \lambda_{ij}}{\partial s}(t,t)=0,
\end{align*}
 for $1\leq i\leq l$, $1\leq j\leq N-1$ and $t\in (0,\mu(p,v))$, which together with \eqref{25} and \eqref{26} gives
 \begin{align} \label{28}
 	\bar{D}_{\sigma'(t)}J_{i}^{t}+\frac{dJ_{i}^{t}}{dt}(t)=0,
 \end{align}
 for $1\leq i\leq l$ and $t\in (0,\mu(p,v))$.

 By the relationship between the Hessian of the distance function and the index form, we can get
\begin{align*}
	\bar{D}^{2}d_{\sigma(t)}(w_{i},w_{i})&=I_{0}^{t}(J_{i}^{t}|_{[0,t]},J_{i}^{t}|_{[0,t]})\\
	&=\int_{0}^{t}(\langle \bar{D}_{\sigma'(s)}J_{i}^{t},\bar{D}_{\sigma'(s)}J_{i}^{t} \rangle-\bar{R}(J_{i}^{t}(s),\sigma'(s),\sigma'(s),J_{i}^{t}(s))ds.
\end{align*}
Combining with \eqref{29}, \eqref{27} and \eqref{28}, differentiating the above equation with respect to $t$ yields
\begin{equation}\label{30}
	\begin{aligned}
&\frac{d}{dt}	\bar{D}^{2}d_{\sigma(t)}(w_{i},w_{i})\\
=&\langle \bar{D}_{\sigma'(t)}J_{i}^{t},\bar{D}_{\sigma'(t)}J_{i}^{t} \rangle-\bar{R}(J_{i}^{t}(t),\sigma'(t),\sigma'(t),J_{i}^{t}(t))\\
&+\int_{0}^{t}\frac{\partial}{\partial t}(\langle \bar{D}_{\sigma'(s)}J_{i}^{t},\bar{D}_{\sigma'(s)}J_{i}^{t} \rangle-\bar{R}(J_{i}^{t}(s),\sigma'(s),\sigma'(s),J_{i}^{t}(s)))ds\\
=&|\bar{D}_{\sigma'(t)}J_{i}^{t}|^{2}+\int_{0}^{t}\frac{\partial}{\partial t}\langle \bar{D}_{\sigma'(s)}J_{i}^{t},\bar{D}_{\sigma'(s)}J_{i}^{t} \rangle-2\bar{R}(J_{i}^{t}(s),\sigma'(s),\sigma'(s),\frac{d}{dt}J_{i}^{t}(s))ds\\
=&|\bar{D}_{\sigma'(t)}J_{i}^{t}|^{2}+2I_{0}^{t}(J_{i}^{t}|_{[0,t]},\frac{dJ_{i}^{t}}{dt}|_{[0,t]})\\
=&|\bar{D}_{\sigma'(t)}J_{i}^{t}|^{2}+2\langle \bar{D}_{\sigma'(t)}J_{i}^{t}, \frac{dJ_{i}^{t}}{dt}(t)\rangle-2\langle \bar{D}_{\sigma'(0)}J_{i}^{t}, \frac{d}{dt}J_{i}^{t}(0)\rangle\\
=&-|\bar{D}_{\sigma'(t)}J_{i}^{t}|^{2}.
	\end{aligned}
\end{equation}
We proceed to prove these two assertions separately.

1\textdegree. Assume that $\mathrm{Ric}^{M}_{l}\geq l\delta$. Summing equation \eqref{30} over $i$ from $1$ to $l$, by the Cauchy-Schwarz inequality and the arithmetic-geometric mean inequality, we obtain
\begin{equation}\label{31}
	\begin{aligned}
	\frac{d}{dt}\sum_{i=1}^{l}	\bar{D}^{2}d_{\sigma(t)}(w_{i},w_{i})
		&=-\sum_{i=1}^{l}|\bar{D}_{\sigma'(t)}J_{i}^{t}|^{2}\\
		&\leq -l^{-1} \left(\sum_{i=1}^{l}|\bar{D}_{\sigma'(t)}J_{i}^{t}|\right)^{2}\\
		&\leq -l |\bar{D}_{\sigma'(t)}J_{1}^{t}\wedge\bar{D}_{\sigma'(t)}J_{2}^{t}\wedge\cdots\wedge\bar{D}_{\sigma'(t)}J_{l}^{t}|^{\frac{2}{l}}.
	\end{aligned}
\end{equation}
Since $J_{i}^{t}(0)=w_{i}$,  $|J_{1}^{t}(0)\wedge J_{2}^{t}(0)\wedge\cdots\wedge J_{l}^{t}(0)|=1$. By assertion (i) of Lemma \ref{16} and in combination with \eqref{31}, we obtain \eqref{41}. By \eqref{40} and \eqref{41}, 
							\begin{align*}
\frac{d}{dt}\frac{\sum_{i=1}^{l}\bar{D}^{2}d_{\sigma(t)}(w_{i},w_{i})}{l \mathfrak{c}_{\delta}(t)/\mathfrak{s}_{\delta}(t)}&=\frac{\frac{l \mathfrak{c}_{\delta}(t)}{\mathfrak{s}_{\delta}(t)}\frac{d}{dt}\sum_{i=1}^{l}\bar{D}^{2}d_{\sigma(t)}(w_{i},w_{i})-\sum_{i=1}^{l}\bar{D}^{2}d_{\sigma(t)}(w_{i},w_{i})\frac{d}{dt}\frac{l \mathfrak{c}_{\delta}(t)}{\mathfrak{s}_{\delta}(t)}}{(l \mathfrak{c}_{\delta}(t)/\mathfrak{s}_{\delta}(t))^{2}}\\
&\leq \frac{\frac{l \mathfrak{c}_{\delta}(t)}{\mathfrak{s}_{\delta}(t)}\frac{d}{dt}\frac{l \mathfrak{c}_{\delta}(t)}{\mathfrak{s}_{\delta}(t)}-\frac{l \mathfrak{c}_{\delta}(t)}{\mathfrak{s}_{\delta}(t)}\frac{d}{dt}\frac{l \mathfrak{c}_{\delta}(t)}{\mathfrak{s}_{\delta}(t)}}{(l \mathfrak{c}_{\delta}(t)/\mathfrak{s}_{\delta}(t))^{2}}=0,
\end{align*}
which implies \eqref{42}. These prove assertion (i).

2\textdegree. Assume that the sectional curvatures of $M$ are bounded above by $\delta$. Similarly, assertion (ii) holds by an argument analogous to the proof of assertion (i).

These complete the proof.
\end{proof}


 We now use Bishop's result (Lemma \ref{16}) to prove the following comparison theorem and Theorem \ref{57}.
\begin{theorem}\label{38}
	Let $f: \Sigma^n \rightarrow M^{n+m}$ be an isometric immersion of an $n$-manifold into a complete Riemannian $(n+m)$-manifold. Let $(x,\xi)\in U\Sigma$ and define the geodesic $\sigma(t):=\mathrm{exp}^{\perp}_{x}\left(t\xi\right)$ for $0\leq t < \infty$. Consider a second such situation $\underline{f}:\underline{ \Sigma}^n\rightarrow \M^{n+m}_{\delta}$ etc. Then the following two assertions hold.

\smallskip
{\rm (i)} If the sectional curvatures of $M$ are bounded below by $\delta$ and $\kappa_{i}(x,\xi)\leq \underline{\kappa}_{i}(\underline{x},\underline{\xi})$ for each $1\leq i \leq n$, then
\begin{align*}
	&|\mathrm{det} T_{t}|\leq |\mathrm{det} \underline{T}_{t}|=t^{n}\prod_{i=1}^{n}\left(\frac{ \mathfrak{c}_{\delta}(t)}{\mathfrak{s}_{\delta}(t)}+\underline{\kappa}_{i}(\underline{x},\underline{\xi})\right),\\
	&\frac{|\mathrm{det} T_{t}|}{|\mathrm{det} \underline{T}_{t}|}\leq \frac{|\mathrm{det} T_{s}|}{|\mathrm{det} \underline{T}_{s}|},
\end{align*}
for $0<s\leq t<\tilde{\tau}_{f}(x,\xi)$. Moreover,  $\tilde{\tau}_{f}(x,\xi)\leq \tilde{\tau}_{\underline{f}}(\underline{x},\underline{\xi})$.

	\smallskip
{\rm (ii)} If $M$ satisfies $\mathrm{Ric}^{M}_{n}\geq n\delta$, $\underline{\Sigma}$ is  umbilical at $\underline{x}$ for the normal $\underline{\xi}$ and $\langle 	\mathbf{H}(x),\xi \rangle\geq  \langle 	\underline{\mathbf{H}}(\underline{x}),\underline{\xi} \rangle$, then 
\begin{align*}
	&|\mathrm{det} T_{t}|\leq |\mathrm{det} \underline{T}_{t}|=t^{n}\left(\frac{ \mathfrak{c}_{\delta}(t)}{\mathfrak{s}_{\delta}(t)}-\langle 	\underline{\mathbf{H}}(\underline{x}),\underline{\xi} \rangle\right)^{n},\\
	&\frac{|\mathrm{det} T_{t}|}{|\mathrm{det} \underline{T}_{t}|}\leq \frac{|\mathrm{det} T_{s}|}{|\mathrm{det} \underline{T}_{s}|},
\end{align*}
for $0<s\leq t<\tilde{\tau}_{f}(x,\xi)$. Moreover, $\tilde{\tau}_{f}(x,\xi)\leq \tilde{\tau}_{\underline{f}}(\underline{x},\underline{\xi})$.
\end{theorem}
\begin{proof}
Fix a $t_{0}\in(0,\tilde{\tau}_{f}(x,\xi))$ and choose an orthonormal basis $\{e_{i} \}_{i=1}^{n}\subset T_{x}\Sigma$ such that $A_{t_{0}}$ is diagonalized. Let us write, $A_{t_{0}}(e_{i},e_{j})=\lambda_{i}\delta_{ij}$ for $ 1\leq i,j \leq n$. By Lemma \ref{34},  $\lambda_{i}>0$ for $ 1\leq i \leq n$. For each $t\in(0,\tilde{\tau}_{f}(x,\xi))$, let  $J^{t}_{i}$ be the Jacobi field along  $\sigma$ such that
$
	J^{t}_{i}(0)=e_{i}
$
and
$
	J^{t}_{i}(t)=0.
$
 By the arithmetic-geometric mean inequality,
\begin{equation*}
	\begin{aligned}
		\frac{d}{dt}\Big|_{t=t_{0}}\log|\mathrm{det}A_{t}|&=\sum_{i=1}^{n}\lambda_{i}^{-1}\frac{d}{dt}\Big|_{t=t_{0}}\bar{D}^{2}d_{\sigma(t)}(e_{i},e_{j})\\
		&=-\sum_{i=1}^{n}\lambda_{i}^{-1}|\bar{D}_{\sigma'(t_{0})}J_{i}^{t_{0}}|^{2}\\
		&\leq -n\left(\prod_{i=1}^{n} \lambda_{i}^{-1}|\bar{D}_{\sigma'(t_{0})}J_{i}^{t_{0}}|^{2} \right)^{\frac{1}{n}}\\
		&=-n|\mathrm{det}A_{t_{0}}|^{-\frac{1}{n}} \left(\prod_{i=1}^{n} |\bar{D}_{\sigma'(t_{0})}J_{i}^{t_{0}}| \right)^{\frac{2}{n}}   .
	\end{aligned}
\end{equation*}
Thus 
\begin{equation}\label{46}
	\begin{aligned}
		\frac{d}{dt}\Big|_{t=t_{0}}\log|\mathrm{det}T_{t}|\leq \frac{n}{t_{0}}-n|\mathrm{det}A_{t_{0}}|^{-\frac{1}{n}} \left(\prod_{i=1}^{n} |\bar{D}_{\sigma'(t_{0})}J_{i}^{t_{0}}| \right)^{\frac{2}{n}}   .
	\end{aligned}
\end{equation}
We proceed to prove these two assertions separately.

1\textdegree. Assume that  the sectional curvatures of $M$ are bounded below by $\delta$ and $\kappa_{i}(x,\xi)\leq \underline{\kappa}_{i}(\underline{x},\underline{\xi})$ for each $1\leq i \leq n$. Let $\{E_{i}\}_{i=1}^n$ be the principal directions of $\Sigma$ with respect to the normal $(x,\xi)$. By Hadamard’s inequality (which states that the
determinant of a positive definite matrix is at most the product of its diagonal entries) and Hessian comparison theorem (see also Lemma \ref{39}),
\begin{equation*}
	\begin{aligned}
		|\mathrm{det}A_{t}|&\leq\prod_{i=1}^{n}(\bar{D}^{2}d_{\sigma(t)}(E_{i},E_{i})+\kappa_{i}(x,\xi))    \\
		&\leq \prod_{i=1}^{n}\left(\frac{ \mathfrak{c}_{\delta}(t)}{\mathfrak{s}_{\delta}(t)}+\kappa_{i}(x,\xi)\right)\\
		&\leq \prod_{i=1}^{n}\left(\frac{ \mathfrak{c}_{\delta}(t)}{\mathfrak{s}_{\delta}(t)}+\underline{\kappa}_{i}(\underline{x},\underline{\xi})\right)=|\mathrm{det}\underline{A}_{t}|.
	\end{aligned}
\end{equation*}
Thus
\begin{equation*}
|\mathrm{det} T_{t}|=t^{n}|\mathrm{det} A_{t}|\leq t^{n}|\mathrm{det} \underline{A}_{t}|=|\mathrm{det} \underline{T}_{t}|=t^{n}\prod_{i=1}^{n}\left(\frac{ \mathfrak{c}_{\delta}(t)}{\mathfrak{s}_{\delta}(t)}+\underline{\kappa}_{i}(\underline{x},\underline{\xi})\right).
\end{equation*}
By \eqref{46} and Lemma \ref{16},
\begin{equation*}
	\begin{aligned}
		\frac{d}{dt}\Big|_{t=t_{0}}\log|\mathrm{det}T_{t}|\leq \frac{n}{t_{0}}-n|\mathrm{det}A_{t_{0}}|^{-\frac{1}{n}} \left(\prod_{i=1}^{n}\mathfrak{s}_{\delta}(t_{0})^{-1} \right)^{\frac{2}{n}} =	\frac{d}{dt}\Big|_{t=t_{0}}\log|\mathrm{det}\underline{T}_{t}|  .
	\end{aligned}
\end{equation*}
Note that $|\mathrm{det} T_{0}|=|\mathrm{det} \underline{T}_{0}|=1$. Therefore,
\[
	\frac{|\mathrm{det} T_{t}|}{|\mathrm{det} \underline{T}_{t}|}\leq \frac{|\mathrm{det} T_{s}|}{|\mathrm{det} \underline{T}_{s}|},
\]
for $0<s\leq t<\tilde{\tau}_{f}(x,\xi)$. Clearly,  $\tilde{\tau}_{f}(x,\xi)\leq \tilde{\tau}_{\underline{f}}(\underline{x},\underline{\xi})$. These prove assertion (i).

2\textdegree. Assume that $M$ satisfies $\mathrm{Ric}^{M}_{n}\geq n\delta$, $\underline{\Sigma}$ is  umbilical at $\underline{x}$ for the normal $\underline{\xi}$ and $\langle 	\mathbf{H}(x),\xi \rangle\geq  \langle 	\underline{\mathbf{H}}(\underline{x}),\underline{\xi} \rangle$. By the arithmetic-geometric mean inequality and Lemma \ref{39}, 
\begin{equation*}
	\begin{aligned}
		|\mathrm{det}A_{t_{0}}|&\leq\left(\frac{\sum_{i=1}^{n}(\bar{D}^{2}d_{\sigma(t)}(e_{i},e_{i})+\kappa_{i}(x,\xi)) }{n}  \right)^{n} \\
		&\leq \left(\frac{ \mathfrak{c}_{\delta}(t_{0})}{\mathfrak{s}_{\delta}(t_{0})}-\langle \mathbf{H}(x),\xi \rangle\right)^{n}\\
		&\leq \prod_{i=1}^{n}\left(\frac{ \mathfrak{c}_{\delta}(t)}{\mathfrak{s}_{\delta}(t)}+\underline{\kappa}_{i}(\underline{x},\underline{\xi})\right)=|\mathrm{det}\underline{A}_{t_{0}}|.
	\end{aligned}
\end{equation*}
Thus 
\begin{equation*}
|\mathrm{det} T_{t}|\leq |\mathrm{det} \underline{T}_{t}|=t^{n}\left(\frac{ \mathfrak{c}_{\delta}(t)}{\mathfrak{s}_{\delta}(t)}-\langle 	\underline{\mathbf{H}}(\underline{x}),\underline{\xi} \rangle\right)^{n}.
\end{equation*}
By \eqref{46} and Lemma \ref{16},
\begin{equation*}
	\begin{aligned}
		\frac{d}{dt}\Big|_{t=t_{0}}\log|\mathrm{det}T_{t}|\leq \frac{n}{t_{0}}-n|\mathrm{det}A_{t_{0}}|^{-\frac{1}{n}} \left(\mathfrak{s}_{\delta}(t_{0})^{-n} \right)^{\frac{2}{n}} =	\frac{d}{dt}\Big|_{t=t_{0}}\log|\mathrm{det}\underline{T}_{t}|  .
	\end{aligned}
\end{equation*}
Note that $|\mathrm{det} T_{0}|=|\mathrm{det} \underline{T}_{0}|=1$. Therefore,
\[
\frac{|\mathrm{det} T_{t}|}{|\mathrm{det} \underline{T}_{t}|}\leq \frac{|\mathrm{det} T_{s}|}{|\mathrm{det} \underline{T}_{s}|},
\]
for $0<s\leq t<\tilde{\tau}_{f}(x,\xi)$. Clearly,  $\tilde{\tau}_{f}(x,\xi)\leq \tilde{\tau}_{\underline{f}}(\underline{x},\underline{\xi})$.  These prove assertion (ii).

These complete the proof.
\end{proof}

\emph{Proof of the sufficiency part of Theorem \ref{68}:} We prove these two assertions separately.

1\textdegree. Assume that  the sectional curvatures of $M$ are bounded below by $\delta$ and $\kappa_{i}(x,\xi)\leq \underline{\kappa}_{i}(\underline{x},\underline{\xi})$ for each $1\leq i \leq n$. By Theorem \ref{16} and Theorem \ref{38}, 
\[
	|\mathrm{det}\ (\mathrm{exp} ^{\perp})_{*(x,t\xi)}|=	|\mathrm{det} (\mathrm{Exp}_{f(x)} )_{*(t\xi)}|\cdot|\mathrm{det} T_{t}|\leq \left(\mathfrak{s}_{\delta}(t)/t\right)^{m-1}\prod_{i=1}^{n}\left( \mathfrak{c}_{\delta}(t)+\mathfrak{s}_{\delta}(t)\underline{\kappa}_{i}(\underline{x},\underline{\xi}) \right)
\]
and
\[
\frac{d}{dt}\mathrm{log}	|\mathrm{det}\ (\mathrm{exp} ^{\perp})_{*(x,t\xi)}|=\frac{d}{dt}\mathrm{log}	|\mathrm{det} (\mathrm{Exp}_{f(x)} )_{*(t\xi)}|+\frac{d}{dt}\mathrm{log}|\mathrm{det} T_{t}|.
\]
Therefore assertion (i) follows from Theorem \ref{38} and the rigidity in Bishop's volume distortion comparison theorem.

2\textdegree. Assume that $M$ satisfies $\mathrm{Ric}^{M}_{n}\geq n\delta$, $\underline{\Sigma}$ is  umbilical at $\underline{x}$ for the normal $\underline{\xi}$ and $\langle 	\mathbf{H}(x),\xi \rangle\geq  \langle 	\underline{\mathbf{H}}(\underline{x}),\underline{\xi} \rangle$. Similarly, assertion (ii) holds by an argument analogous to the proof of assertion (i).

These complete the proof.
\hfill$\Box$\\

We end this section with two remarks.
\begin{rem}
$(1)$ In Theorem \ref{57}, if one only needs a Heintze-Karcher-type estimate (i.e., a pointwise upper bound estimate for $|\mathrm{det}\ (\mathrm{exp} ^{\perp})_{*(x,t\xi)}|$), one may take $\underline{\kappa}_{i}(\underline{x},\underline{\xi})=\kappa_{i}(x,\xi) $ , for $1\leq i \leq n$, in (i) and, correspondingly, $\langle 	\underline{\mathbf{H}}(\underline{x}),\underline{\xi} \rangle=\langle 	\mathbf{H}(x),\xi \rangle$ in (ii);

$(2)$ As we did not improve the Heintze-Karcher comparison theorem under an upper bound on the sectional curvature, we did not use assertion (ii) of Lemma \ref{39} nor assertion (ii) of Lemma \ref{16} in Theorems \ref{38} and \ref{57}. In assertion (i) of Lemma \ref{39} and assertion (i) of Lemma \ref{16}, one may take $l=1$; consequently they cover the case of a lower bound on the sectional curvature. The reader should be careful not to confuse these curvature conditions.
\end{rem}

\section{Proof of inequality \eqref{54}}\label{104}
Let  $(M^{n+m},\bar{g})$ be a complete noncompact Riemannian manifold with nonnegative sectional curvature and Euclidean volume growth. Let $\Sigma$ be a closed $n$-dimensional Riemannian manifold  and $f: \Sigma^{n} \rightarrow M$ an isometric immersion.

Throughout the remainder of this paper, we shall use the following three sets. For each positive number $r_{0}$, define a subset of $M$ as
\begin{equation*}
	\Omega_{r_{0}}:=\left\{ p\in M: 0\leq d(p,f(\Sigma))\leq r_{0} \right\}
\end{equation*}
and two subsets of $T^{\perp}\Sigma$ as
\begin{equation*}
	U_{r_{0}}:=\left\{(x,r\xi)\in T^{\perp}\Sigma: x\in \Sigma,\, \xi\in S^{m-1}_{x},\, 0\leq r  \leq\min\{\tau_{f}(x,\xi),r_{0}\}\right\}
\end{equation*}
and
\begin{equation*}
	\tilde{U}_{r_{0}}:=\left\{(x,r\xi)\in T^{\perp}\Sigma: x\in \Sigma,\, \xi\in S^{m-1}_{x},\, 0\leq r  \leq\min\{\tilde{\tau}_{f}(x,\xi),r_{0}\}\right\}.
\end{equation*}
Moreover, throughout Sections~\ref{104}--\ref{106}, we divide $S_{x}^{m-1}$ into three parts as
$S_{x}^{m-1}=L_{x}\cup M_{x}\cup N_{x}$ for each $x\in\Sigma$ with
							\begin{equation*}
	L_{x}:=\left\{ \xi\in S_{x}^{m-1}:\kappa_{1}(x,\xi)>0 \right\},
\end{equation*}
							\begin{equation*}
	M_{x}:=\left\{ \xi\in S_{x}^{m-1}:\kappa_{1}(x,\xi)=0 \right\},
\end{equation*}
							\begin{equation*}
	N_{x}:=\left\{ \xi\in S_{x}^{m-1}:\kappa_{1}(x,\xi)<0 \right\}.
\end{equation*}

\begin{lem}\label{56}
For each $r_{0}>0$, $ \mathrm{exp}^{\perp}(	U_{r_{0}})=\mathrm{exp}^{\perp}(\tilde{U}_{r_{0}})=	\Omega_{r_{0}}$.				
\end{lem}
\begin{proof}
	Fix a positive number $r_{0}$. For each $p\in \Omega_{r_{0}}\setminus f(\Sigma)$, the compactness of $\Sigma$  implies that the function $x\mapsto d(p,f(x))$ on
	 $\Sigma$ attains its minimum at some point $x_{0}\in \Sigma$. A standard variational argument then shows that there exists a unit normal vector $\xi_{0}\in T^{\perp}_{x_{0}}\Sigma$ and a minimal geodesic $\gamma:[0,r]\rightarrow M$ given by $\gamma(t):=\mathrm{exp}^{\perp}(x_{0},t\xi_{0})$, which connects $\gamma(0)=f(x_{0})$ and $\gamma(r)=p$. Moreover, we have $ 0\leq r  \leq\min\{\tau_{f}(x_{0},\xi_{0}),r_{0}\}$. Therefore, $(x_{0},r\xi_{0})\in U_{r_{0}}\subset \tilde{U}_{r_{0}}$. Thus, $\Omega_{r_{0}} \subset  \mathrm{exp}^{\perp}(U_{r_{0}})\subset\mathrm{exp}^{\perp}(\tilde{U}_{r_{0}})	$.	Clearly, 	$\mathrm{exp}^{\perp}(U_{r_{0}})\subset\mathrm{exp}^{\perp}(\tilde{U}_{r_{0}})\subset \Omega_{r_{0}}$. The lemma follows.			
\end{proof}

\emph{Proof of \eqref{54}:} 
Using  Lemma \ref{56}, Theorem  \ref{57} and applying the area formula to the map $\mathrm{exp}^{\perp}|_{	\tilde{U}_{r_{0}}}$, 
							\begin{equation}\label{58}
	\begin{array}{lllll}
		|\Omega_{r_{0}}|&\leq \int_{T^{\perp}\Sigma}1_{\tilde{U}_{r_{0}}}(x,r\xi)|\mathrm{det}\ (\mathrm{exp}^{\perp})_{*(x,r\xi)}| d\mathrm{vol}_{T^{\perp}\Sigma}(x,r\xi)\\
		&\leq\int_{T^{\perp}\Sigma}1_{\tilde{U}_{r_{0}}}(x,r\xi)\prod_{i=1}^{n}\left(1+r\kappa_{i}(x,\xi) \right)d\mathrm{vol}_{T^{\perp}\Sigma}(x,r\xi)\\
		&=\int_{\Sigma}\int_{L_{x}}\int_{0}^{\min\{\tilde{\tau}_{f}(x,\xi),r_{0}\}}r^{m-1}\prod_{i=1}^{n}\left(1+r\kappa_{i}(x,\xi) \right)drd\xi d\mathrm{vol}_{\Sigma}(x)\\
		&\ \ \ +\int_{\Sigma}\int_{M_{x}}\int_{0}^{\min\{\tilde{\tau}_{f}(x,\xi),r_{0}\}}r^{m-1}\prod_{i=1}^{n}\left(1+r\kappa_{i}(x,\xi) \right)drd\xi d\mathrm{vol}_{\Sigma}(x)\\
			&\ \ \ +\int_{\Sigma}\int_{N_{x}}\int_{0}^{\min\{\tilde{\tau}_{f}(x,\xi),r_{0}\}}r^{m-1}\prod_{i=1}^{n}\left(1+r\kappa_{i}(x,\xi) \right)drd\xi d\mathrm{vol}_{\Sigma}(x).
	\end{array}
\end{equation}
By Theorem \ref{38}, for $(x,\xi)\in\left\{  (x,\xi)\in U\Sigma:x\in \Sigma,\, \xi \in N_{x}      \right\}$, \[\tilde{\tau}_{f}(x,\xi)\leq-\frac{1}{\kappa_{1}(x,\xi)}.\]
 Consequently, the second and third terms on the right-hand side of inequality \eqref{58} are both of lower order than $r_{0}^{n+m}$. After rearrangement, we can obtain
							\begin{equation*}
								\begin{array}{lllll}
		|\Omega_{r_{0}}|&\leq\int_{\Sigma^{+}}\int_{L_{x}}\int_{0}^{\min\{\tilde{\tau}_{f}(x,\xi),r_{0}\}}r^{n+m-1}\prod_{i=1}^{n}\kappa_{i}(x,\xi) drd\xi d\mathrm{vol}_{\Sigma}(x)\\
		&\ \ \ +\mathrm{lowerorder\ terms\ of}\ r_{0}^{n+m}.
			\end{array}
\end{equation*}
Note that for each $(x,\xi)\in\left\{  (x,\xi)\in U\Sigma:x\in \Sigma^{+},\, \xi \in L_{x}      \right\}$, the antipodal point $-\xi\in S_{x}^{m-1}$ of $\xi$ also satisfies
							\begin{equation*}
		\left|\prod_{i=1}^{n}\kappa_{i}(x,-\xi)\right|=\left|\prod_{i=1}^{n}\kappa_{i}(x,\xi)\right|=\prod_{i=1}^{n}\kappa_{i}(x,\xi)	=	\left|\mathrm{det}   \langle h(\cdot,\cdot),-\xi \rangle\right|.
\end{equation*}
Then 
						\begin{equation*}
	\begin{array}{lllll}
		|\Omega_{r_{0}}|&\leq\frac{r_{0}^{n+m}}{2(n+m)}\int_{\Sigma^{+}}  \int_{S_{x}^{m-1}}	\left|\prod_{i=1}^{n}\kappa_{i}(x,\xi)\right|d\xi d\mathrm{vol}_{\Sigma}(x)\\
		&\ \ \ +\mathrm{lowerorder\ terms\ of}\ r_{0}^{n+m}\\
		&=\frac{r_{0}^{n+m}}{2(n+m)} \int_{\Sigma^{+}}K^{*}(x)d\mathrm{vol}(x)    \\
				&\ \ \ +\mathrm{lowerorder\ terms\ of}\ r_{0}^{n+m}.
	\end{array}
\end{equation*}
Dividing both sides by $r_{0}^{n+m}\omega_{n+m}$ and letting $r_{0}\rightarrow \infty$ can give us 
\begin{equation*}
	\int_{\Sigma}K^{*}(x)d\mathrm{vol}(x)\geq 	\int_{\Sigma^{+}}K^{*}(x)d\mathrm{vol}(x)\geq 2 \mathrm{AVR}(\bar{g}) \left| \mathbb{S}^{n+m-1} \right|
\end{equation*}
which completes the proof.
\hfill$\Box$

\section{Proof of the necessity part of Theorem \ref{67}}\label{105}
Let  $(M^{n+m},\bar{g})$ be a complete noncompact Riemannian manifold with nonnegative sectional curvature and Euclidean volume growth. Let $\Sigma$ be a closed $n$-dimensional Riemannian manifold  and $f: \Sigma^{n} \rightarrow M$ be an isometric immersion satisfying
	\begin{equation}\label{70}
	\int_{\Sigma}K^{*}(x)d\mathrm{vol}_{\Sigma}(x)= 2 \mathrm{AVR}(\bar{g}) \left| \mathbb{S}^{n+m-1} \right|.
\end{equation}

Set $\Sigma_{0}:=\Sigma \setminus \Sigma^{+}$.
Note that	$L_x$ is open in $S_x^{m-1}$ and	$\Sigma_1$ is open in $\Sigma$.

\begin{lem}\label{51}
	$\Sigma$ is connected.
\end{lem}
\begin{proof}
	Suppose otherwise that $\Sigma$ has $N$ connected components, denoted by $\Sigma^{(1)}, \cdots, \Sigma^{(N)}$. Then the assumption and the argument in Section \ref{104} yield
	\begin{equation*}
		\begin{array}{lllll}
		2 \mathrm{AVR}(\bar{g}) \left| \mathbb{S}^{n+m-1} \right|=\int_{\Sigma}K^{*}d\mathrm{vol}_{\Sigma}=	\sum_{k=1}^{N}\displaystyle\int_{\Sigma^{(k)}}K^{*}d\mathrm{vol}_{\Sigma}\geq N\cdot 2 \mathrm{AVR}(\bar{g}) \left| \mathbb{S}^{n+m-1} \right|,
		\end{array}
	\end{equation*}
	which implies $N=1$, as claimed.
\end{proof}

	\begin{lem}\label{78}
	Assume that $x\in\Sigma_{0}$ and $\xi\in S_{x}^{m-1}$. Then the linear transformation $\langle h(\cdot,\cdot),\xi\rangle$ necessarily has a zero eigenvalue.
\end{lem}
\begin{proof}	
	Since
	\begin{equation*}
		2 \mathrm{AVR}(\bar{g}) \left| \mathbb{S}^{n+m-1} \right|\leq	\int_{\Sigma^{+}}K^{*}(x)d\mathrm{vol}(x)\leq	\int_{\Sigma}K^{*}(x)d\mathrm{vol}(x)=	2 \mathrm{AVR}(\bar{g})\left| \mathbb{S}^{n+m-1} \right|,
	\end{equation*}
	it is clear that
	\begin{equation}\label{72}
		\int_{ \Sigma_{0}}K^{*}(x)d\mathrm{vol}(x)=0.
	\end{equation}
	
	We claim that $K^{*}(x)=0$ for each $x\in\Sigma_{0}$.
	It can be proved by contradiction. Suppose that there exist an $\bar{x}\in \Sigma_{0}$ satisfying $K^{*}(\bar{x})>0$.
	Clearly, $\bar{x}$ is not a boundary point of $\Sigma^{+}$ and $\bar{x}\notin \Sigma^{+}$. By continuity there exists an  open neighborhood  $W\subset \Sigma$ of $\bar{x}$ such that  $W\subset \Sigma_{0}$ and $K^{*}(x)>0$ for each $x\in W$ which contradicts \eqref{72}. Thus we have $K^{*}(x)=0$ for each $x\in\Sigma_{0}$. Then the lemma follows by the definition of $K^{*}(x)$.
\end{proof}

\begin{lem}\label{60}
	For all $(x,\xi)\in U\Sigma$ satisfying $x\in\Sigma^{+}$ and $\xi \in L_{x}$,  $\tilde{\tau}_{f}(x,\xi)=\infty$. 
\end{lem}
\begin{proof}
	We argue by contradiction. Suppose that there exists a point $x_{0}\in\Sigma^{+}$ and a unit normal vector $\xi_{0}\in L_{x_{0}}$  such that $\tilde{\tau}_{f}(x_{0},\xi_{0})<\infty$. By continuity, there exists a positive number
	\[
\varepsilon<\frac{1}{4}\int_{S_{x_{0}}^{m-1}}	\left|\mathrm{det}\langle h(\cdot,\cdot),\xi \rangle    \right|d\xi
	\]
	and an open neighborhood $V\subset U\Sigma$ of $(x_{0},\xi_{0})$ such that  $x\in\Sigma^{+}$, $\xi\in L_{x}$, $\tilde{\tau}_{f}(x,\xi)<\tilde{\tau}_{f}(x_{0},\xi_{0})+\varepsilon$ and
	\[
\int_{V_{x}}	\left|\mathrm{det}\langle h(\cdot,\cdot),\xi \rangle    \right|d\xi>\varepsilon,
	\]
 for all $(x,\xi)\in V$, where $V_{x}=\{\xi\in S_{x}^{m-1}:(x,\xi)\in V \}$. Similar to Section \ref{104}, for each $r_{0}>\tilde{\tau}_{f}(x_{0},\xi_{0})+\varepsilon$,
								\begin{align*}
			|\Omega_{r_{0}}|&\leq\int_{\Sigma^{+}}\int_{L_{x}}\int_{0}^{\min\{\tilde{\tau}_{f}(x,\xi),r_{0}\}}r^{n+m-1}\prod_{i=1}^{n}\kappa_{i}(x,\xi) drd\xi d\mathrm{vol}_{\Sigma}(x)\\
			&\ \ \ +\mathrm{lowerorder\ terms\ of}\ r_{0}^{n+m}\\
			&\leq\int_{\pi(V)}  \int_{L_{x}\setminus V_{x}}\int_{0}^{r_{0}}r^{n+m-1}\prod_{i=1}^{n}\kappa_{i}(x,\xi) drd\xi d\mathrm{vol}_{\Sigma}(x)\\
				&\ \ \ +\int_{\pi(V)}  \int_{ V_{x}}\int_{0}^{\tau_{f}(x_{0},\xi_{0})+\varepsilon}r^{n+m-1}\prod_{i=1}^{n}\kappa_{i}(x,\xi) drd\xi d\mathrm{vol}_{\Sigma}(x)\\
			&\ \ \ +\int_{\Sigma^{+}\setminus\pi(V)}\int_{L_{x}}\int_{0}^{r_{0}}r^{n+m-1}\prod_{i=1}^{n}\kappa_{i}(x,\xi) drd\xi
			d\mathrm{vol}_{\Sigma}(x)\\
&\ \ \ +\mathrm{lowerorder\ terms\ of}\ r_{0}^{n+m}\\
&=\int_{\pi(V)}  \int_{L_{x}\setminus V_{x}}\int_{0}^{r_{0}}r^{n+m-1}\prod_{i=1}^{n}\kappa_{i}(x,\xi) drd\xi d\mathrm{vol}_{\Sigma}(x)\\
	&\ \ \ +\int_{\Sigma^{+}\setminus\pi(V)}\int_{L_{x}}\int_{0}^{r_{0}}r^{n+m-1}\prod_{i=1}^{n}\kappa_{i}(x,\xi) drd\xi
d\mathrm{vol}_{\Sigma}(x)\\
			&\ \ \ +\mathrm{lowerorder\ terms\ of}\ r_{0}^{n+m}\\
	&\leq\frac{1}{2}\int_{\pi(V)}  \int_{L_{x}\cup(-L_{x})\setminus (V_{x}\cup (-V_{x}))}\int_{0}^{r_{0}}r^{n+m-1}	\left|\mathrm{det}\langle h(\cdot,\cdot),\xi \rangle    \right| drd\xi d\mathrm{vol}_{\Sigma}(x)\\
	&\ \ \ \frac{1}{2}\int_{\Sigma^{+}\setminus\pi(V)}  \int_{L_{x}\cup(-L_{x})}\int_{0}^{r_{0}}r^{n+m-1}	\left|\mathrm{det}\langle h(\cdot,\cdot),\xi \rangle    \right| drd\xi d\mathrm{vol}_{\Sigma}(x)\\
		&\ \ \ +\mathrm{lowerorder\ terms\ of}\ r_{0}^{n+m}\\
		&\leq\frac{r_{0}^{n+m}}{2(n+m)}\int_{\Sigma^{+}}  \int_{S_{x}^{m-1}}	\left|\mathrm{det}\langle h(\cdot,\cdot),\xi \rangle    \right| d\xi d\mathrm{vol}_{\Sigma}(x)\\
				&\ \ \ -\frac{r_{0}^{n+m}}{n+m}\int_{\pi(V)}  \int_{V_{x}}	\left|\mathrm{det}\langle h(\cdot,\cdot),\xi \rangle    \right| d\xi d\mathrm{vol}_{\Sigma}(x)\\
		&\ \ \ +\mathrm{lowerorder\ terms\ of}\ r_{0}^{n+m}\\
				&\leq\frac{r_{0}^{n+m}}{2(n+m)}\int_{\Sigma^{+}}  \int_{S_{x}^{m-1}}	\left|\mathrm{det}\langle h(\cdot,\cdot),\xi \rangle    \right| d\xi d\mathrm{vol}_{\Sigma}(x)\\
		&\ \ \ -\frac{\varepsilon r_{0}^{n+m}}{n+m}|\pi(V)|+\mathrm{lowerorder\ terms\ of}\ r_{0}^{n+m}.
	\end{align*}
		Dividing by $r_{0}^{n+m}\omega_{n+m}$ and letting $r_{0}\rightarrow \infty$ can give us 
	\begin{equation*}
 \mathrm{AVR}(\bar{g}) \leq \frac{1}{2\left| \mathbb{S}^{n+m-1} \right|} \int_{\Sigma}K^{*}(x)d\mathrm{vol}(x)-\frac{\varepsilon |\pi(V)|}{\left| \mathbb{S}^{n+m-1} \right|}<\frac{1}{2\left| \mathbb{S}^{n+m-1} \right|} \int_{\Sigma}K^{*}(x)d\mathrm{vol}(x),
	\end{equation*}
		which is a contradiction. The lemma follows.
\end{proof}

\begin{lem}\label{59}
	For all $(x,y)\in T^{\perp}\Sigma$ satisfying $x\in \Sigma^{+}$ and $y/|y|\in L_{x}$,
	\begin{equation*}
	|\mathrm{det}\ (\mathrm{exp}^{\perp})_{*(x,y)}|=\prod_{i=1}^{n}\left(1+|y|\kappa_{i}(x,y/|y|) \right).
	\end{equation*}
\end{lem}
	\begin{proof}
		Suppose that there exists a point $(x_{0},y_{0})\in T^{\perp}\Sigma$ such that $x_{0}\in \Sigma^{+}$ , $y_{0}/|y_{0}|\in L_{x_{0}}$ and 
			\[
		|\mathrm{det}\ (\mathrm{exp}^{\perp})_{*(x_{0},y_{0})}|<\prod_{i=1}^{n}\left(1+r\kappa_{i}(x_{0},y_{0}/|y_{0}|) \right). 
			\]
			By continuity, there exists a positive number $\varepsilon<|y_{0}|$ and an open neighborhood $Z\subset T^{\perp}\Sigma$ of $(x_{0},y_{0})$ such that   $x\in \Sigma^{+}$, $y/|y|\in L_{x}$, $0<|y_{0}|-\varepsilon<|y|\leq|y_{0}|+\varepsilon$,
			\[
			|\mathrm{det}\ (\mathrm{exp}^{\perp})_{*(x,y)}|<(1-\varepsilon)\prod_{i=1}^{n}\left(1+|y|\kappa_{i}(x,y/|y|) \right)		
			\]
		 and 
			 \[
			  \int_{ Z'_{x}}	\left|\mathrm{det}\langle h(\cdot,\cdot),\xi \rangle    \right| d\xi>\varepsilon,
			 \]
			 for all $(x,y)\in Z$, where $Z'_{x}:=\{y/|y|\in S_{x}^{m-1}:(x,y)\in Z \}$.  For each $(x,y)\in Z$ and $r\in(|y_{0}|+\varepsilon,+\infty)$, Theorem \ref{57} implies 
			 				\begin{align*}
				|\mathrm{det}\ (\mathrm{exp}^{\perp})_{*(x,ry/|y|)}|&\leq\left( \prod_{i=1}^{n}\frac{1+r\kappa_{i}(x,y/|y|) }{1+|y|\kappa_{i}(x,y/|y|)}\right)	|\mathrm{det}\ \mathrm{exp}^{\perp}_{*(x,y)}|\\
				&<\left( \prod_{i=1}^{n}\frac{1+r\kappa_{i}(x,y/|y|) }{1+|y|\kappa_{i}(x,y/|y|)}\right)	(1-\varepsilon)\prod_{i=1}^{n}\left(1+|y|\kappa_{i}(x,y/|y|) \right)\\
				&=(1-\varepsilon)\prod_{i=1}^{n}\left(1+r\kappa_{i}(x,y/|y|) \right).
			 		\end{align*}
			 Combined with this estimate, an argument analogous to that in Section \ref{104} gives us, for each $r_{0}>|y_{0}|+\varepsilon$, 
				\begin{align*}
				|\Omega_{r_{0}}|
				&\leq\int_{\Sigma^{+}\setminus\pi(Z)} \int_{L_{x}}\int_{0}^{r_{0}}r^{n+m-1}	\left|\mathrm{det}\langle h(\cdot,\cdot),\xi \rangle    \right| drd\xi d\mathrm{vol}_{\Sigma}(x)\\
			&\ \ \ +\int_{\pi(Z)} \int_{L_{x}\setminus Z'_{x}}\int_{0}^{r_{0}}r^{n+m-1}	\left|\mathrm{det}\langle h(\cdot,\cdot),\xi \rangle    \right| drd\xi d\mathrm{vol}_{\Sigma}(x)\\
			&\ \ \ +\int_{\pi(Z)} \int_{ Z'_{x}}\int_{0}^{|y_{0}|+\varepsilon}r^{n+m-1}	\left|\mathrm{det}\langle h(\cdot,\cdot),\xi \rangle    \right| drd\xi d\mathrm{vol}_{\Sigma}(x)\\
						&\ \ \ +\int_{\pi(Z)} \int_{ Z'_{x}}\int_{|y_{0}|+\varepsilon}^{r_{0}}r^{n+m-1}(1-\varepsilon)	\left|\mathrm{det}\langle h(\cdot,\cdot),\xi \rangle    \right| drd\xi d\mathrm{vol}_{\Sigma}(x)\\
				&\ \ \ +\mathrm{lowerorder\ terms\ of}\ r_{0}^{n+m}\\
				&=\frac{r_{0}^{n+m}}{2(n+m)} \int_{\Sigma^{+}\setminus \pi(Z)}K^{*}(x)d\mathrm{vol}(x)+\frac{r_{0}^{n+m}}{2(n+m)} \int_{\pi(Z)}K^{*}(x)d\mathrm{vol}(x)    \\
					&\ \ \ -\frac{\varepsilon r_{0}^{n+m}}{n+m}  \int_{\pi(Z)} \int_{ Z'_{x}}	\left|\mathrm{det}\langle h(\cdot,\cdot),\xi \rangle    \right| d\xi d\mathrm{vol}_{\Sigma}(x)              \\
				&\ \ \ +\mathrm{lowerorder\ terms\ of}\ r_{0}^{n+m}\\
				&\leq \frac{r_{0}^{n+m}}{2(n+m)} \int_{\Sigma}K^{*}(x)d\mathrm{vol}(x)-\frac{\varepsilon^{2} r_{0}^{n+m}|\pi(Z)|}{n+m} +\mathrm{lowerorder\ terms\ of}\ r_{0}^{n+m}.
			\end{align*}
		Dividing by $r_{0}^{n+m}\omega_{n+m}$ and letting $r_{0}\rightarrow \infty$ yields  
			\begin{equation*}
				\mathrm{AVR}(\bar{g}) \leq \frac{1}{2\left| \mathbb{S}^{n+m-1} \right|} \int_{\Sigma}K^{*}(x)d\mathrm{vol}(x)-\frac{\varepsilon^{2} |\pi(Z)|}{\left| \mathbb{S}^{n+m-1} \right|}<\frac{1}{2\left| \mathbb{S}^{n+m-1} \right|} \int_{\Sigma}K^{*}(x)d\mathrm{vol}(x).
			\end{equation*}
			It is a contradiction. The lemma follows.
\end{proof}
	
\begin{lem}\label{79}
	Assume that $x\in \Sigma^{+}$ and $y/|y|\in L_{x}$. Let $\gamma(t):=\mathrm{exp}^{\perp}_{x}(ty)$ for all $t\in [0,\infty)$. Then the sectional curvature of $M$ is equal to $\delta$ for all tangent planes containing $\gamma'(t)$.
\end{lem}
\begin{proof}
	By Lemma \ref{59} and Theorem \ref{57}, the lemma follows.
\end{proof}
	
	\begin{lem}\label{61}
	Assume that $x\in \Sigma^{+}$. Then	$L_{x}$ is exactly an open hemisphere contained in $S_{x}^{m-1}$.
	\end{lem}
	\begin{proof}
We prove this lemma by contradiction. Suppose that there exists a point $x_{0}\in\Sigma^{+}$ such that $L_{x_{0}}$ is not an open hemisphere contained in $S_{x_{0}}^{m-1}$. By the definition of $L_{x_{0}}$, we can find two unit vectors $v_{1},v_{2}\in T_{x_{0}}\Sigma$ such that $h(v_{1},v_{1})\neq \vec{0}$, $h(v_{2},v_{2})\neq \vec{0}$, 
				\begin{equation*}
-1<\langle h(v_{1},v_{1})/|h(v_{1},v_{1})|, h(v_{2},v_{2})/h(v_{2},v_{2})\rangle <1 
\end{equation*}
and 
				\begin{equation*}
L_{x_{0}}\subset \bigcap_{i=1}^{2}\left\{ \xi\in S_{x_{0}}^{m-1}:\langle h(v_{i},v_{i}),\xi \rangle< 0 \right\}.
	\end{equation*}
	We conclude that
					\begin{equation*}
P_{x_{0}}:=\left\{ \xi\in S_{x_{0}}^{m-1}:\langle h(v_{1},v_{1}),\xi \rangle> 0 \right\}\cap \left\{ \xi\in S_{x_{0}}^{m-1}:\langle h(v_{2},v_{2}),\xi \rangle< 0 \right\}
\end{equation*}
is a non-empty open set of $S_{x_{0}}^{m-1}$ and
					\begin{equation*}
	-L_{x_{0}}\subset \bigcap_{i=1}^{2}\left\{ \xi\in S_{x_{0}}^{m-1}:\langle h(v_{i},v_{i}),\xi \rangle> 0 \right\}.
\end{equation*}
Thus $P_{x_{0}}\cap( L_{x_{0}}\cup(-L_{x_{0}}))=\emptyset$.

 Consequently, by continuity, there exists a positive number $\varepsilon$ and an open neighborhood $W$ of $x_{0}$ contained in $\Sigma^{+}$ such that  
					\begin{equation*}
\int_{W}\int_{L_{x}\cup(-L_{x})}\left|\mathrm{det}\langle h(\cdot,\cdot),\xi \rangle    \right|d\xi d\mathrm{vol}_{\Sigma}(x) \leq\int_{W}(K^{*}(x)-\varepsilon)d\mathrm{vol}_{\Sigma}(x).
\end{equation*}
Similar to Section \ref{104}, for each $r_{0}>0$, 
\begin{align*}
	|\Omega_{r_{0}}|&\leq\frac{r_{0}^{n+m}}{2(n+m)}\int_{\Sigma^{+}}  \int_{L_{x}\cup(-L_{x})}	\left|\mathrm{det}\langle h(\cdot,\cdot),\xi \rangle    \right| d\xi d\mathrm{vol}_{\Sigma}(x)\\
	&\ \ \ +\mathrm{lowerorder\ terms\ of}\ r_{0}^{n+m}\\
	&\leq\frac{r_{0}^{n+m}}{2(n+m)}\left[\int_{\Sigma^{+}\setminus W} K^{*}(x)  d\mathrm{vol}_{\Sigma}(x)+\int_{W} (K^{*}(x)-\varepsilon) d\mathrm{vol}_{\Sigma}(x)   \right]\\
	&\ \ \ +\mathrm{lowerorder\ terms\ of}\ r_{0}^{n+m}\\
		&=\frac{r_{0}^{n+m}}{2(n+m)}\int_{\Sigma^{+}} K^{*}(x)  d\mathrm{vol}_{\Sigma}(x)	-\frac{\varepsilon r_{0}^{n+m}}{2(n+m)}|W| \\
	&\ \ \ +\mathrm{lowerorder\ terms\ of}\ r_{0}^{n+m}.
\end{align*}
Dividing by $r_{0}^{n+m}\omega_{n+m}$ and letting $r_{0}\rightarrow \infty$ can give us 
\begin{equation*}
	\mathrm{AVR}(\bar{g}) \leq \frac{1}{2\left| \mathbb{S}^{n+m-1} \right|} \int_{\Sigma}K^{*}(x)d\mathrm{vol}(x)-\frac{\varepsilon |W|}{2\left| \mathbb{S}^{n+m-1} \right|}<\frac{1}{2\left| \mathbb{S}^{n+m-1} \right|} \int_{\Sigma}K^{*}(x)d\mathrm{vol}(x),
\end{equation*}
which is a contradiction. The lemma follows.
\end{proof}
	By Lemma \ref{61}, for $x\in\Sigma^{+}$ there exists a unit vector $\xi_{x}\in S_{x}^{m-1}$ such that
	\[
L_{x}=\{\xi \in  S_{x}^{m-1}:\langle \xi,\xi_{x} \rangle >0
	\}. 
	\]
Next, we give the following property of the second fundamental form $h$.
	\begin{lem}\label{77}
For each $x\in\Sigma^{+}$, the image of the second fundamental form $h:T_{x}\Sigma\times T_{x}\Sigma\rightarrow T_{x}^{\perp}\Sigma$ is  $\{c\xi_{x}:c\in \R  \}$.
\end{lem}
\begin{proof}	
	Suppose that there exists a point $x_{0}\in\Sigma^{+}$, a unit normal vector $y_{0}\in T_{x_{0}}^{\perp}\Sigma$ and two  non-collinear unit tangent vectors $v_{0},w_{0}\in T_{x_{0}}\Sigma$ such that $\langle \xi_{x_{0}},y_{0}    \rangle =0$ and
	$\langle h(v_{0},w_{0}),y_{0}    \rangle >0$. Define a symmetric bilinear form $I:T_{x_{0}}\Sigma\times T_{x_{0}}\Sigma\rightarrow \R$ by 
					\begin{equation*}
	\phi(v,w):=\langle h(v,w),y_{0}    \rangle
	\end{equation*}
	for all $v,w\in T_{x_{0}}\Sigma$. It is clear that
						\begin{equation*}
		\phi(v_{0}+w_{0},v_{0}+w_{0})-\phi(v_{0},v_{0})-\phi(w_{0},w_{0})=2\phi(v_{0},w_{0})>0.
	\end{equation*}
	Thus in the three terms on the left hand side of the above equality, at least one of them does not vanish. Without loss of generality, we may assume that
							\begin{equation*}
		\phi\left(\frac{v_{0}+w_{0}}{|v_{0}+w_{0}|},\frac{v_{0}+w_{0}}{|v_{0}+w_{0}|}\right)\neq 0.
	\end{equation*}
Then by Lemma \ref{61} and the definition of $\xi_{x_{0}}$, it follows that $\langle \xi_{x_{0}},y_{0} \rangle \neq 0$. This contradicts the choice of $y_{0}$. The lemma follows.
\end{proof}
Since Lemma \ref{77} holds, we can define $\mathcal{U}$ by \eqref{84}. Moreover, Lemmas \ref{61} and \ref{77} imply that the $\{-\xi_{x} \}$ coincides with 	$
\{h(v,v)/|h(v,v)|: v\in T_{x}\Sigma,\,|v|=1   \} 
$
for each $x\in \Sigma^{+}$ and that $\mathcal{U}=\{(x,r\xi)\in T^{\perp}\Sigma: x\in \Sigma^{+},\, r>0,\,\xi\in L_{x}\}$.

	\begin{lem}\label{63}
	Assume that $(x_{1},r_{1}\xi)\in \mathcal{U}$ with $x_{1}\in \Sigma^{+},\, r_{1}>0$ and $\xi_{1} \in S_{x}^{m-1}$. Then for every $x_{2} \in \Sigma$ and every positive number $r_{2} < r_{1}$,  $\mathrm{exp}^{\perp}(x_{1}, r_{1}\xi_{1}) \notin B_{f(x_{2})}^{M}(r_{2})$.
\end{lem}
\begin{proof}
	Suppose that $\mathrm{exp}^{\perp}(x_{1}, r_{1}\xi_{1}) \in B_{f(x_{2})}^{M}(r_{2})$. By Lemma~\ref{60}, there exist a positive number $r_3<r_1-r_2$ and an open neighborhood $W\subset\mathcal U$ of $(x_1,r_1\xi_1)$ such that
\begin{itemize}
		\item $W$ is diffeomorphic via $\exp^\perp$ to $B_{\exp^\perp(x_1,r_1\xi_1)}^M(r_3)$;
		\item $B_{\exp^\perp(x_1,r_1\xi_1)}^M(r_3)$ is contained in $B_{f(x_2)}^M(r_2)$;
		\item  $d(f(\bar x),f(\tilde x))<r_1-r_2-r_3$, for each $\bar x,\tilde x\in \pi(W)$.
\end{itemize}
	For each $q=\mathrm{exp}^{\perp}(x, r\xi)\in B^{M}_{\mathrm{exp}^{\perp}(x_{1}, r_{1}\xi_{1})}(r_{3})$ with $(x, r\xi)\in W$,
	by the triangle inequality, it follows that
	\begin{align*}
		d(q,f(x_{1}))&\geq d(\mathrm{exp}^{\perp}(x_{1}, r_{1}\xi_{1}),f(x))-d(\mathrm{exp}^{\perp}(x_{1}, r_{1}\xi_{1}),q)\\
		&>d(\mathrm{exp}^{\perp}(x_{1}, r_{1}\xi_{1}),f(x_{1}))-d(f(x_{1}),f(x))-r_{3}\\
		&=r_{1}-r_{3}-d(f(x_{1}),f(x))\\
		&>r_{r_{2}}>d(q,f(x_{2})).
	\end{align*}
	Thus we have $\tau_{f}(x,\xi)<r_{2}$ for each $(x, r\xi)\in W$.
	Applying the same argument as in Lemma \ref{60} can lead us to a contradiction. 	Therefore, $\mathrm{exp}^{\perp}(x_{1}, r_{1}\xi_{1}) \notin B_{f(x_{2})}^{M}(r_{2})$. The lemma follows.
\end{proof}

\begin{lem}\label{64}
	Assume that $(x,y)\in \mathcal{U}$. Then 
	$\tau_{f}(x,y/|y|)=\infty$.
\end{lem}
\begin{proof}
	We still prove this lemma by contradiction. Suppose that there exists a point $(x_0, y_0) \in \mathcal{U}$ such that $\tau_f(x_0, y_0/|y_0|) < \infty$.  Denote by $\xi_{0}$ the unit vector $y_{0}/|y_{0}|$. It is necessary to consider the following two cases.
	
	\noindent \textbf{Case 1:} $\tau_f(x_0, \xi_{0})= 0$.
	
	We claim that $f$ must have a self-intersection at $x_0$. Indeed, otherwise there exists a positive number $r_0$ such that the restriction of $f$ to $f^{-1}(B_{f(x_{0})}^{M}(r_{0}))$ is an embedding, and hence $\tau_f(x_0, \xi_{0}) > 0$. It's a contradiction.
	
	At $x_0$, $f$ must have only tangential self-intersections and no transverse self-intersections. Suppose that there exists $\bar{x}\in\Sigma$ such that $f$ intersects transversely at $x_0$ and $\bar{x}$. Then there exist $y\in T_{x_0}^{\perp}\Sigma$ and $v\in T_{\bar{x}}\Sigma$ such that $(x_0,y)\in\mathcal{U}$ and $\langle y,v\rangle>0$. This, together with the asymptotic behavior and Lemma \ref{63}, leads to a contradiction. Therefore, at $x_0$, $f$ can only have tangential self-intersections.
	
	We set $f(x_{0})=p$ and claim that the set $f^{-1}(p)$ must be finite. If not, since $\Sigma$ is compact and $f^{-1}(p)$ is closed in  $\Sigma$, we can find a sequence of distinct points $\{x_{k}\}_{k=0}^{\infty} \subset \Sigma$ such that $f(x_k)=p$ and $\lim_{k\to\infty} x_{k} = x_{\infty} \in \Sigma$. By continuity, $f(x_{\infty})=p$. However, since $f$ is an immersion, this contradicts the fact that $x_{\infty}$ is an accumulation point. Hence the set $f^{-1}(p)$ must be finite. 
	
	Without loss of generality, we may assume that $f^{-1}(p) =\{  x_{0},x_{1},\cdots,x_{N}\}$. Choose a sufficiently small positive number $\bar{r}$ such that $f^{-1}(B^{M}_{p}(\bar{r}))=\cup_{k=0}^{N}V_{k}$, where $V_{k}$ is an open neighborhood of $x_{k}$ in $\Sigma$ and $f|_{V_{k}}$ is an embedding for $k = 0, 1, \cdots, N$. Immediately, we have $\tau_{f|_{V_{k}}}(x_{k},\xi_{0})>0$ for $k = 0, 1, \cdots, N$. Set
	\[
	C:=\min \{ \frac{\bar{r}}{2},\tau_{f|_{V_{0}}}(x_{0},\xi_{0}),\tau_{f|_{V_{1}}}(x_{1},\xi_{0}) ,\cdots,\tau_{f|_{V_{N}}}(x_{N},\xi_{0})    \}.
	\]
	Then for each positive number t with $0<t<C$, 
	\begin{align*}
		d(\mathrm{exp}^{\perp}_{x_{0}}(t\xi_{0}),f(\Sigma))&=\inf_{x\in\Sigma}d(\mathrm{exp}^{\perp}_{x_{0}}(t\xi_{0}),f(x))\\
		&=\min\{ \inf_{x\in\Sigma\setminus(\bigcup_{k=0}^{N}V_{k})}d(\mathrm{exp}^{\perp}_{x_{0}}(t\xi_{0}),f(x)),\inf_{x\in\bigcup_{k=0}^{N}V_{k}}d(\mathrm{exp}^{\perp}_{x_{0}}(t\xi_{0}),f(x))    \}.
	\end{align*}
	Since 
	\[
	d(\mathrm{exp}^{\perp}_{x_{0}}(t\xi_{0}),f(x))\geq \frac{\bar{r}}{2}>t
	\]
	for each $x\in\Sigma\setminus(\cup_{k=0}^{N}V_{k})$ and
	\[
	d(\mathrm{exp}^{\perp}_{x_{0}}(t\xi_{0}),f(x))=t
	\]
	for each $x\in\cup_{k=0}^{N}V_{k}$, it can be seen that 
	$d(\mathrm{exp}^{\perp}_{x_{0}}(t\xi_{0}),f(\Sigma))=t$
	for each positive number t with $0<t<C$. Hence $\tau_f(x_0, \xi_{0})\geq C>0$. It's a contradiction.
	
	\noindent \textbf{Case 2:} $0<\tau_f(x_0, \xi_{0})<  \infty$.
	
	For a positive number $r_{0}>\tau_f(x_0, \xi_{0})$, there exist $\bar{x}\in \Sigma$, $\bar{\xi}\in S_{\bar{x}}^{m-1}$ and  a positive number $r_{1}<r_{0}$ such that $\mathrm{exp}^{\perp}(\bar{x},r_{1}\bar{\xi})=\mathrm{exp}^{\perp}(x_{0},r_{0}\xi_{0})$ which contradicts with Lemma \ref{63}.
	
	Thus we obtain 	$\tau_{f}(x,y/|y|)=\infty$ for  $(x,y)\in \mathcal{U}$. The lemma follows.
\end{proof}

\begin{lem}\label{73}
Suppose $V$ is a nonempty open set contained in $\Sigma^{+}$ such that $f|_{V}$ is an embedding. Let
$\mathcal{V}:=\{(x,y)\in T^{\perp}\Sigma: x\in V, y\in T_{x}^{\perp}\Sigma,\langle \xi_{x},y \rangle >0  \}$. Then $\mathrm{exp}^{\perp}|_{\mathcal{V}}:\mathcal{V}\rightarrow \mathrm{exp}^{\perp}(\mathcal{V})$  is a diffeomorphism.
\end{lem}
\begin{proof}
	Since Lemma \ref{60} holds, we only need to show that $\mathrm{exp}^{\perp}|_{\mathcal{V}}$ is injective. We argue by contradiction. Fix a point $p\in \mathrm{exp}^{\perp}(\mathcal{V})$ and suppose there exist $x_{1}\in V, r_{1}>0,\xi_{1}\in S_{x_{1}}^{m-1}$ and $x_{2}\in V, r_{2}>0,\xi_{2}\in S_{x_{2}}^{m-1}$ such that $(x_{1},r_{1}\xi_{1}),(x_{2},r_{2}\xi_{2})\in \mathcal{V}$ and
	\[
	\mathrm{exp}^{\perp}(x_{1},r_{1}\xi_{1})=\mathrm{exp}^{\perp}(x_{2},r_{2}\xi_{2})=p.
	\]
	 Clearly, $\mathcal{V}\subset \mathcal{U}$. By Lemma \ref{63}, we have $r_{1}=r_{2}$. Next, by the definition of $\tau_{f}$ and Lemma \ref{64}, we obtain $f(x_{1})=f(x_{2})$ and $\xi_{1}=\xi_{2}$. This contradicts the fact that $f|_{V}$ is an embedding. Therefore, $\mathrm{exp}^{\perp}|_{\mathcal{V}}$ is injective, and the lemma follows.
\end{proof}

\begin{lem}\label{69}
	The map $f|_{\Sigma^{+}}$ is an embedding.
\end{lem}
\begin{proof}
We argue by contradiction again. Suppose that there exist two distinct points $x_{1},x_{2}\in\Sigma^{+}$ such that $f(x_{1})=f(x_{2})=p\in M$. We claim that $f$ does not have a transverse self-intersection at $x_{1}$ and $x_{2}$. Otherwise, there exist $y\in T_{x_{1}}^{\perp}\Sigma$ and $v\in T_{x_{2}}\Sigma$ such that $(x_{1},y)\in\mathcal{U}$ and $\langle y,v\rangle>0$. This, together with the asymptotic behavior and Lemma \ref{63}, leads to a contradiction.
Hence the self-intersection must be tangential, and two cases are considered.

\noindent \textbf{Case 1:} $\xi_{x_{1}}\neq - \xi_{x_{2}}$.

There exists $\xi_{0}\in S_{x_{1}}^{m-1}=S_{x_{2}}^{m-1}$ such that $(x_{1},r\xi)\in\mathcal{U}$ and $(x_{2},r\xi)\in\mathcal{U}$ for each $r>0$. For a fixed positive number $r_{0}$, by Lemma \ref{60}, there exist an open neighborhood $W_{1}\subset \mathcal{U}$ of $(x_{1},r_{0}\xi_{0})$, an open neighborhood $W_{2}\subset \mathcal{U}$ of $(x_{2},r_{0}\xi_{0})$ and a positive number $\varepsilon<r_{0}$ such that
\begin{itemize}
	\item $W_1$ and $W_2$ are homeomorphic to $\mathbb{B}^{n+m}$;
		\item  $\pi(W_{1})$ and $\pi(W_{2})$ are disjoint and  homeomorphic to $\mathbb{B}^{n}$;
	\item For $k=1,2$, $W_k$ is diffeomorphic via $\exp^{\perp}$ to $\exp^{\perp}(W_k)$;
	
	\item $\exp^{\perp}(W_1)=\exp^{\perp}(W_2)=:W$;
	\item The set $V_{k}:=\{(x,r_{0}\xi)\in W_{k}:x\in \Sigma^{+},\xi \in S_{x}^{m-1}\}$ is homeomorphic to 
	$\mathbb{B}^{n+m-1}$ and $\partial \exp^{\perp}(V_{k}) \subset \partial W$ for $k=1,2$.
\end{itemize}
Note that each $V_1$ and $V_2$ are embedded hypersurfaces in $\mathcal U$. By the above properties, for each $k=1,2$, the image $\exp^\perp(V_k)$ is an $(n+m-1)$-dimensional embedded hypersurface in $M$ and splits $W$ into three pairwise disjoint parts:
\[
W=\exp^\perp(V_{k,+})\cup \exp^\perp(V_k)\cup \exp^\perp(V_{k,-}),
\]
where 
\[
\mathrm{exp}^{\perp}(V_{k,+}):=\{ (x,t\xi)\in W_{k}:x\in \Sigma^{+},\xi \in S_{x}^{m-1},t>r_{0} \},
\]
\[
\mathrm{exp}^{\perp}(V_{k,-}):=\{ (x,t\xi)\in W_{k}:x\in \Sigma^{+},\xi \in S_{x}^{m-1},t<r_{0} \}.
\]
By Lemma \ref{63}, it follows that
\[
\mathrm{exp}^{\perp}(V_{1})=\mathrm{exp}^{\perp}(V_{2}),\quad \mathrm{exp}^{\perp}(V_{1,+})=\mathrm{exp}^{\perp}(V_{2,+}),\quad\mathrm{exp}^{\perp}(V_{1,-})=\mathrm{exp}^{\perp}(V_{2,-}).
\]
For a fixed point $q\in \mathrm{exp}^{\perp}(V_{1})=\mathrm{exp}^{\perp}(V_{2})$, there exist points $(\tilde{x}_{1},r_{0}\xi_{1})\in W_{1}$ and $(\tilde{x}_{2},r_{0}\xi_{2})\in W_{2}$ such that $ \mathrm{exp}^{\perp}(\tilde{x}_{1},r_{0}\xi_{1})=\mathrm{exp}^{\perp}(\tilde{x}_{2},r_{0}\xi_{2})=q$. For each $k=1,2$, let $\gamma_k:[0,\infty)\to M$ be the geodesic $\gamma_k(t):=\exp^\perp(\tilde x_k,t\xi_k)$, which is minimal by Lemma~\ref{60}. Fix a positive number $t_0>r_0$ such that $\gamma_1(t_0)\in W$, and let $\tilde\gamma_2$ be a minimal geodesic joining $f(\tilde x_2)$ to $\gamma_1(t_0)$. On one hand, the triangle inequality gives
\[
t_0 = L(\gamma_2|_{[0,r_0]}) + L(\gamma_1|_{[r_0,t_0]}) \ge L(\tilde\gamma_2),
\]
where $L$ is the length functional. On the other hand, Lemma~\ref{63} yields
\[
L(\tilde\gamma_2) \ge t_0.
\]
Hence $L(\tilde\gamma_2)=t_0$, the concatenation of $\gamma_2|_{[0,r_0]}$ and $\gamma_1|_{[r_0,t_0]}$ forms a geodesic. By uniqueness of geodesics, it follows that
\[
f(\tilde x_1)=f(\tilde x_2),\qquad \xi_1=\xi_2.
\]
Therefore, we have found a diffeomorphism $\varphi:\pi(V_{1}) \rightarrow \pi(V_{2})$ such that $f(\varphi(x))=f(x)$ for each $x\in \pi(V_{1})$.

For each positive number $r_{0}$, we define two subsets of $\mathcal{U}$ as follows: for $k=1,2$, let
\[
X_{r_{0}}^{k}:=\{(x,r\xi): x\in V_{k},\,0<r\leq r_{0},\, \xi\in S_{x}^{m-1},\, \langle \xi_{x},\xi \rangle >0  \}.
\]
From the previous discussion, we have $\mathrm{exp}^{\perp}(X_{r_{0}}^{1})=\mathrm{exp}^{\perp}(X_{r_{0}}^{2})$ for each positive number $r_{0}$. By Lemma \ref{73}, $\mathrm{exp}^{\perp}|_{X_{r_{0}}^{1}}$ is a diffeomorphism. On the one hand, applying the change of variables formula to this map yields
		\begin{equation*}
	\begin{array}{lllll}
		|\mathrm{exp}^{\perp}(X_{r_{0}}^{1})|
		&=\int_{V_{1}}\int_{L_{x}}\int_{0}^{r_{0}}r^{n+m-1}\prod_{i=1}^{n}\kappa_{i}(x,\xi)drd\xi d\mathrm{vol}_{\Sigma}(x)\\
		&\ \ \ +\mathrm{lowerorder\ terms\ of}\ r_{0}^{n+m}.
	\end{array}
\end{equation*}
Dividing both sides of the above equality by $r_{0}^{n+m}\omega_{n+m}$ and letting  $r_{0}\rightarrow \infty$, 
		\begin{equation}\label{74}
	\begin{array}{lllll}
	\lim_{r_{0}\rightarrow\infty}  \frac{|\mathrm{exp}^{\perp}(X_{r_{0}}^{1})|}{r_{0}^{n+m}\omega_{n+m}}=\frac{\int_{V_{1}}K^{*}(x)d\mathrm{vol}(x)}{2|\mathbb{S}^{n+m-1}|}>0.
	\end{array}
\end{equation}
Note that $\mathrm{exp}^{\perp}(X_{r_{0}}^{1})\subset \Omega_{r_{0}}$ and
$\mathscr{H}^{0}((\mathrm{exp}^{\perp})^{-1}(p)\cap U_{r_{0}})\geq 2$ for each $p\in \mathrm{exp}^{\perp}(X_{r_{0}}^{1})$. On the other hand, as in Section \ref{104}, applying the area formula to the map $\mathrm{exp}^{\perp}|_{U_{r_{0}}}$, 
for each $r_{0}>0$,
\begin{equation*}
	\begin{array}{lllll}
		&\ \ \ |\Omega_{r_{0}}|+| X_{r_{0}}|\\
		&\leq\int_{\Sigma^{+}}\int_{L_{x}}\int_{0}^{r_{0}}r^{n+m-1}\prod_{i=1}^{n}\kappa_{i}(x,\xi) drd\xi d\mathrm{vol}_{\Sigma}(x)\\
		&\ \ \ +\mathrm{lowerorder\ terms\ of}\ r_{0}^{n+m}.
	\end{array}
\end{equation*}
Dividing both sides of the above equality by $r_{0}^{n+m}\omega_{n+m}$ and letting  $r_{0}\rightarrow \infty$, we obtain
\begin{equation*}
	\begin{array}{lllll}
		\mathrm{AVR}(\bar{g})+\lim_{r_{0}\rightarrow\infty}  \frac{|\mathrm{exp}^{\perp}(X_{r_{0}}^{1})|}{r_{0}^{n+m}\omega_{n+m}}\leq\frac{\int_{\Sigma^{+}}K^{*}(x)d\mathrm{vol}(x)}{2|\mathbb{S}^{n+m-1}|}\leq\mathrm{AVR}(\bar{g}),
	\end{array}
\end{equation*}
which contradicts \eqref{74}.
Therefore, this case cannot occur.

\noindent \textbf{Case 2:} $\xi_{x_{1}}= - \xi_{x_{2}}$.

In fact, $\Sigma^{+}$ has properties similar to those of convex hypersurfaces in Euclidean space; by introducing local coordinates, we immediately obtain that this case contradicts Lemma~\ref{63}.

Thus the map $f|_{\Sigma^{+}}$ must be an embedding. The lemma follows.
\end{proof}



\emph{Proof of the necessity part of Theorem \ref{67}:} 
By Lemmas \ref{73} and \ref{69}, we conclude that $\mathrm{exp}^{\perp}|_{\mathcal{U}}:\mathcal{U}\rightarrow \mathrm{exp}^{\perp}(\mathcal{U})$ is a diffeomorphism. Combining this fact with Lemmas \ref{51}, \ref{78}, \ref{61}, \ref{77}, and \ref{69}, it remains only to compute the pulled‑back metric on $\mathcal{U}$.

Fix a point $(\bar{x},\bar{y})\in\mathcal{U}$. We shall use the frame fields established prior to Lemma \ref{9}. By Lemmas \ref{9} and \ref{79}, 
\begin{align*}
	&\langle (\mathrm{exp}^{\perp})_{*(\bar{x},\bar{y})}\left(X_{i}\right),\eta_{j} \rangle=g_{ij}(\bar{x})-\langle h(e_{i},e_{j}),\bar{y} \rangle ,\\
	&\langle (\mathrm{exp}^{\perp})_{*(\bar{x},\bar{y})}\left(X_{i}\right),\eta_{\beta} \rangle=\bar{y}^{\sigma}\Gamma_{i \sigma}^{\beta}(\bar{x}),\\
	&\langle (\mathrm{exp}^{\perp})_{*(\bar{x},\bar{y})}\left(\partial/\partial y^{\alpha}\right),\eta_{j} \rangle=0,\\
	&\langle (\mathrm{exp}^{\perp})_{*(\bar{x},\bar{y})}\left(\partial/\partial y^{\alpha}\right),\eta_{\beta} \rangle=\delta_{\alpha}^{\beta}.
\end{align*}
For each $1\leq A\leq n+m$, let  $\bar{\omega}^{A}$ be  the dual $1$-form of $\eta_{A}$. Note that
\begin{equation*}
	\bar{g}_{AB}(\mathrm{exp}^{\perp}_{\bar{x}}\bar{y})=\langle \eta_{A},\eta_{B} \rangle =\langle e_{A},e_{B} \rangle,
\end{equation*}
and
\begin{equation*}
	\bar{g}^{ij}(\mathrm{exp}^{\perp}_{\bar{x}}\bar{y})=g^{ij}(\bar{x}) ,\,	\bar{g}^{i\alpha}(\mathrm{exp}^{\perp}_{\bar{x}}\bar{y})=g^{i\alpha }(\bar{x})=0,\,\bar{g}^{\alpha\beta}(\mathrm{exp}^{\perp}_{\bar{x}}\bar{y})=\delta_{\alpha\beta}
\end{equation*}
for  $1\leq i,j\leq n$ and $n+1\leq \alpha,\beta\leq n+m$. 
These yield
\begin{align*}
	(\mathrm{exp}^{\perp})^{*}\bar{\omega}^{i}&= \bar{g}^{iA} (\mathrm{exp}^{\perp}_{\bar{x}}\bar{y})\langle (\mathrm{exp}^{\perp})_{*(\bar{x},\bar{y})}\left(X_{j}\right),\eta_{A} \rangle  dx^{j}+ \bar{g}^{iA}(\mathrm{exp}^{\perp}_{\bar{x}}\bar{y}) \langle (\mathrm{exp}^{\perp})_{*(\bar{x},\bar{y})}\left(\partial/\partial y^{\beta}\right),\eta_{A} \rangle dy^{\beta}  \\
	&=\left(\delta_{j}^{i}-g^{ik}\langle h(e_{j},e_{k}),\bar{y} \rangle \right)dx^{j},
\end{align*}
\begin{align*}
	(\mathrm{exp}^{\perp})^{*}\bar{\omega}^{\alpha}&= \bar{g}^{\alpha A}(\mathrm{exp}^{\perp}_{\bar{x}}\bar{y}) \langle (\mathrm{exp}^{\perp})_{*(\bar{x},\bar{y})}\left(X_{j}\right),\eta_{A} \rangle  dx^{j}+ \bar{g}^{\alpha A} (\mathrm{exp}^{\perp}_{\bar{x}}\bar{y})\langle (\mathrm{exp}^{\perp})_{*(\bar{x},\bar{y})}\left(\partial/\partial y^{\beta}\right),\eta_{A} \rangle dy^{\beta}  \\
	&=\bar{y}^{\beta}\Gamma_{i \beta}^{\alpha}(\bar{x})dx^{i}+dy^{\alpha}
\end{align*}
and
\begin{align*}
	\bar{g}(\mathrm{exp}^{\perp}_{\bar{x}}\bar{y})&=\langle \eta_{A},\eta_{B} \rangle \bar{\omega}^{A}\otimes\bar{\omega}^{B}=\langle e_{A},e_{B} \rangle \bar{\omega}^{A}\otimes\bar{\omega}^{B}\\
	&=g_{ij}(\bar{x}) \bar{\omega}^{i}\otimes\bar{\omega}^{j}+\delta_{\alpha\beta}\bar{\omega}^{\alpha}\otimes\bar{\omega}^{\beta}.
\end{align*}
Now we calculate the pulled back metric as follows:
\begin{align*}
		[(\mathrm{exp}^{\perp})^{*}\bar{g}](\bar{x},\bar{y})&=g_{ij}(\bar{x}) (\mathrm{exp}^{\perp})^{*}\bar{\omega}^{i}\otimes(\mathrm{exp}^{\perp})^{*}\bar{\omega}^{j}+\delta_{\alpha\beta}(\mathrm{exp}^{\perp})^{*}\bar{\omega}^{\alpha}\otimes(\mathrm{exp}^{\perp})^{*}\bar{\omega}^{\beta}\\
	&=g_{ij}(\bar{x})\left(\delta_{p}^{i}-g^{iq}\langle h(e_{p},e_{q}),\bar{y} \rangle \right)\left(\delta_{l}^{j}-g^{jk}\langle h(e_{l},e_{k}),\bar{y} \rangle \right)dx^{p}\otimes dx^{l}\\
	&\ \ \ +\delta_{\alpha\beta}( \bar{y}^{\sigma}\Gamma_{i \sigma}^{\alpha}(\bar{x})dx^{i}+dy^{\alpha})\otimes( \bar{y}^{\tau}\Gamma_{j \tau}^{\beta}(\bar{x})\otimes dx^{j}  +dy^{\beta})\\
		&=\left[g_{ij}(\bar{x})-2 \langle h(e_{i},e_{j}),\bar{y} \rangle  +g^{kl}(\bar{x})\langle h(e_{i},e_{k}),\bar{y} \rangle \langle h(e_{j},e_{l}),\bar{y} \rangle \right] dx^{i}\otimes dx^{j}  \\
	&\ \ \ +\delta_{\alpha\beta}( \bar{y}^{\sigma}\Gamma_{i \sigma}^{\alpha}(\bar{x})dx^{i}+dy^{\alpha})\otimes( \bar{y}^{\tau}\Gamma_{j \tau}^{\beta}(\bar{x})\otimes dx^{j}  +dy^{\beta}).
\end{align*}
Recalling the canonical metric  (see \eqref{13}) on $\pi^{-1}(U)\subset T^{\perp}\Sigma$, we obtain
\begin{align*}
	g_{T^{\perp}\Sigma}(\bar{x},\bar{y})=\delta_{\alpha\beta}( \bar{y}^{\sigma}\Gamma_{i \sigma}^{\alpha}(\bar{x})dx^{i}+dy^{\alpha})\otimes( \bar{y}^{\tau}\Gamma_{j \tau}^{\beta}(\bar{x}) dx^{j}  +dy^{\beta})+g_{\Sigma}(\bar{x}).
\end{align*}
Immediately,
\begin{align*}
		[(\mathrm{exp}^{\perp})^{*}\bar{g}](\bar{x},\bar{y})
	=\left[g^{kl}(\bar{x})\langle h(e_{i},e_{k}),\bar{y} \rangle \langle h(e_{j},e_{l}),\bar{y} \rangle -2 \langle h(e_{i},e_{j}),\bar{y} \rangle \right] dx^{i}\otimes dx^{j} +g_{T^{\perp}\Sigma} .
\end{align*}
Since Lemma \ref{77} holds, one can choose an orthonormal basis $\{a_{i}\}_{i=1}^{n}$ of $T_{\bar{x}}\Sigma$ with dual basis $\{\omega^{i}\}_{i=1}^{n}$ of  $T_{\bar{x}}^{*}\Sigma$  such that $\langle h(a_{i},a_{j}),-\xi_{\bar{x}} \rangle=\kappa_{i}(\bar{x},\xi_{\bar{x}})\delta_{ij}$ for $1\leq i,j \leq n$ and $\kappa_{1}(\bar{x},\xi_{\bar{x}})\leq \kappa_{2}(\bar{x},\xi_{\bar{x}})\leq \cdots\leq \kappa_{n}(\bar{x},\xi_{\bar{x}})$.
Let $\{a_{i}^{j}\}_{i,j=1}^{n}$ be the transformation matrix between basis $\{a_{i}\}_{i=1}^{n}$ and basis $\{e_{i}\}_{i=1}^{n}$ with its inverse matrix $\{b_{i}^{j}\}_{i,j=1}^{n}$ that is
\begin{align*}
	e_{i}=b_{i}^{j}a_{j},\,a_{i}=a_{i}^{j}e_{j}.
\end{align*}
Then we conclude that 
\begin{align*}
	dx^{i}=a_{j}^{i}\omega^{j},\,g_{ij}(\bar{x})=\langle e_{i},e_{j}\rangle=\sum_{k=1}^{n}b_{i}^{k}b_{j}^{k},\,g^{ij}(\bar{x})=\sum_{k=1}^{n}a_{k}^{i}a_{k}^{j}.
	\end{align*}
Therefore, for each $(x,y)\in\mathcal{U}$, the pulled back metric on $\mathcal{U}$ is given by
\begin{align*}
	[(\mathrm{exp}^{\perp})^{*}\bar{g}](x,y)=\sum_{i=1}^{n}\left[\left(1+\kappa_{i}(x,\xi_{x})\langle \xi_{x},y \rangle \right)^{2}-1\right]\omega^{i}\otimes \omega^{i}+g_{T^{\perp}\Sigma}(x,y),
\end{align*}
	for each $(x,y)\in \mathcal{U}$.

These complete the proof.
\hfill$\Box$

\section{Proof of the sufficiency part of Theorem \ref{67}}\label{106}
Let  $(M^{n+m},\bar{g})$ be a complete noncompact Riemannian manifold with nonnegative sectional curvature and Euclidean volume growth. Let $\Sigma$ be a closed $n$-dimensional Riemannian manifold  and $f: \Sigma^{n} \rightarrow M$ be an isometric immersion satisfying conditions (i), (ii) and (iii) in Theorem \ref{67}.

\emph{Proof of the sufficiency part of Theorem \ref{67}.} 
By the assumption, we immediately have $\mathcal{U}=\{(x,r\xi)\in T^{\perp}\Sigma: x\in \Sigma^{+},\, \xi\in L_{x},\, r>0\}$ and $\tilde{\tau}_{f}(x,y/|y|)=\infty$ for $(x,y)\in \mathcal{U}$. On the one hand, since the normal exponential map $\mathrm{exp}^{\perp}|_{\mathcal{U}}:\mathcal{U}\rightarrow \mathrm{exp}^{\perp}(\mathcal{U})$ is a diffeomorphism, for each $r_{0}>0$,
\begin{equation}\label{65}
	|\mathrm{exp}^{\perp}(\mathcal{U}\cap U_{r_{0}})|=\int_{\Sigma^{+}}\int_{L_{x}}\int_{0}^{r_{0}}r^{m-1}\prod_{i=1}^{n}\left(1+r\kappa_{i}(x,\xi) \right)drd\xi d\mathrm{vol}_{\Sigma}(x).
\end{equation}
On the other hand, applying the area formula to the the map $\mathrm{exp}^{\perp}|_{U_{r_{0}}}$, 
for each $r_{0}>0$,
\begin{equation}\label{66}
	\begin{array}{lllll}
		&\ \ \ \int_{\Omega_{r_{0}}}\mathscr{H}^{0}((\mathrm{exp}^{\perp})^{-1}(p)\cap U_{r_{0}})d\mathrm{vol}_{M}(p)\\
		&=\int_{\Sigma}\int_{S_{x}^{m-1}}\int_{0}^{\min\{\tau_{f}(x,\xi),r_{0}\}}r^{m-1}|\mathrm{det}\ (\mathrm{exp}^{\perp})_{*(x,r\xi)}|drd\xi d\mathrm{vol}_{\Sigma}(x)\\
		&\leq\int_{\Sigma^{+}}\int_{L_{x}}\int_{0}^{r_{0}}r^{m-1}\prod_{i=1}^{n}\left(1+r\kappa_{i}(x,\xi) \right)drd\xi d\mathrm{vol}_{\Sigma}(x)\\
		&\ \ \ +\int_{\Sigma}\int_{M_{x}}\int_{0}^{\min\{\tau_{f}(x,\xi),r_{0}\}}r^{m-1}\prod_{i=1}^{n}\left(1+r\kappa_{i}(x,\xi) \right)drd\xi d\mathrm{vol}_{\Sigma}(x)\\
		&\ \ \ +\int_{\Sigma}\int_{N_{x}}\int_{0}^{\min\{\tau_{f}(x,\xi),r_{0}\}}r^{m-1}\prod_{i=1}^{n}\left(1+r\kappa_{i}(x,\xi) \right)drd\xi d\mathrm{vol}_{\Sigma}(x).
	\end{array}
\end{equation}
Note that for each $p\in \Omega_{r_{0}}$ we have $\mathscr{H}^{0}((\mathrm{exp}^{\perp})^{-1}(p)\cap U_{r_{0}})\geq 1$. By \eqref{65} and \eqref{66},
\begin{equation*}
	\begin{array}{lllll}
		&\ \ \ \int_{\Omega_{r_{0}}\setminus\mathrm{exp}^{\perp}(\mathcal{U}\cap U_{r_{0}})}\mathscr{H}^{0}((\mathrm{exp}^{\perp})^{-1}(p)\cap U_{r_{0}})d\mathrm{vol}_{M}(p)\\
		&\ \ \ + \int_{\mathrm{exp}^{\perp}(\mathcal{U}\cap U_{r_{0}})}(\mathscr{H}^{0}((\mathrm{exp}^{\perp})^{-1}(p)\cap U_{r_{0}})-1)d\mathrm{vol}_{M}(p)\\
		&\leq\int_{\Sigma}\int_{M_{x}}\int_{0}^{\min\{\tau_{f}(x,\xi),r_{0}\}}r^{m-1}\prod_{i=1}^{n}\left(1+r\kappa_{i}(x,\xi) \right)drd\xi d\mathrm{vol}_{\Sigma}(x)\\
		&\ \ \ +\int_{\Sigma}\int_{N_{x}}\int_{0}^{\min\{\tau_{f}(x,\xi),r_{0}\}}r^{m-1}\prod_{i=1}^{n}\left(1+r\kappa_{i}(x,\xi) \right)drd\xi d\mathrm{vol}_{\Sigma}(x).
	\end{array}
\end{equation*}
Since the right-hand side of the above inequality is of lower order than $r_0^{n+m}$, it follows that $\int_{\Omega_{r_{0}}\setminus\mathrm{exp}^{\perp}(\mathcal{U}\cap U_{r_{0}})}\mathscr{H}^{0}((\mathrm{exp}^{\perp})^{-1}(p)\cap U_{r_{0}})d\mathrm{vol}_{M}(p)$ is also of lower order than $r_0^{n+m}$. Therefore, it is obvious that
\[
\int_{\Sigma^{+}}K^{*}(x)d\mathrm{vol}_{\Sigma}(x)= 2 \mathrm{AVR}(\bar{g}) \left| \mathbb{S}^{n+m-1} \right|,
\]
which together with condition (iii) yields
\[
\int_{\Sigma}K^{*}(x)d\mathrm{vol}_{\Sigma}(x)= 2 \mathrm{AVR}(\bar{g}) \left| \mathbb{S}^{n+m-1} \right|.
\]
This completes the proof.
\hfill$\Box$\\

\section{Proof of inequality \eqref{55}}\label{107}
	Let $(M^{n+m},\bar{g})$ $(n\geq2)$ be a complete noncompact Riemannian manifold with $\mathrm{Ric}_{n}^{M}\geq 0$ and Euclidean volume growth. Let $\Sigma$ be a closed $n$-dimensional  Riemannian manifold  and $f: \Sigma^{n} \rightarrow M$ be an isometric immersion.

Throughout Sections~\ref{107}--\ref{108}, $\Sigma$ is divided into two parts as $\Sigma=\Sigma_{0}\cup \Sigma_{+}$ with
\begin{equation*}
	\begin{array}{lllll}
		\Sigma_{0}:=\left\{  x\in\Sigma: \mathbf{H}(x)= \vec{0}\right\},\\
		\Sigma_{+}:=\left\{  x\in\Sigma: \mathbf{H}(x)\neq \vec{0}\right\}.
	\end{array}
\end{equation*}
Write $\mathbf{H}(x)=|\mathbf{H}(x)|\mathbf{e}(x)$ with $\mathbf{e}(x)\in S_{x}^{m-1}$ and $y=r\xi$ with $\xi\in S_{x}^{m-1}$ and $r> 0$ for $x\in\Sigma_{+}$ and $y\in T_{x}^{\perp}\Sigma$, $y\neq \vec{0}$. Moreover, we divide $S_{x}^{m-1}$ into two parts as
$S_{x}^{m-1}=E_{x}^{1}\cup E_{x}^{2}$ for each $x\in\Sigma_{+}$ with
\begin{equation*}
	\begin{array}{lllll}
		&E_{x}^{1}:=\left\{  \xi\in S_{x}^{m-1}: \langle \mathbf{e}(x),\xi \rangle< 0\right\},\\
		&E_{x}^{2}:=\left\{  \xi\in S_{x}^{m-1}: \langle \mathbf{e}(x),\xi \rangle\geq0\right\}.
	\end{array}
\end{equation*}

\begin{lem}\label{91}
For each $x\in\Sigma_{+}$, 
\[
\int_{E_{x}^{1}}\langle -\mathbf{e}(x),\xi\rangle ^{n}d\xi=\frac{(n+m)\omega_{n+m}}{\left| \mathbb{S}^{n} \right|}.
\]
\end{lem}
\begin{proof}

In the case $m=1$, the integral equals $1$, and the identity $(n+1)\omega_{n+1}=|\mathbb{S}^{n}|$ gives the required result immediately. For $m\geq 2$, the set $E_{x}^{1}$ is an $(m-1)$-dimensional hemisphere. Parametrizing by polar coordinates and invoking the Beta function
\[
B(p,q)=2\int_{0}^{\frac{\pi}{2}}\cos^{2p-1}\theta\,\sin^{2q-1}\theta\,d\theta
=\frac{\Gamma(p)\Gamma(q)}{\Gamma(p+q)},\qquad p,q>0,
\]
together with the standard Gamma function expression for sphere areas,
\[
\left| \mathbb{S}^{N} \right|=2\pi^{\frac{N+1}{2}}/\Gamma\left(\frac{N+1}{2}\right),
\]
yields
\[
\begin{aligned}
	\int_{E_{x}^{1}}\langle -\mathbf{e}(x),\xi\rangle ^{n}\,d\xi
	&=\left| \mathbb{S}^{m-2} \right|
	\int_{0}^{\frac{\pi}{2}}\cos^{n}\theta\,\sin^{m-2}\theta\,d\theta \\
	&=\frac{1}{2}\left| \mathbb{S}^{m-2} \right|
	B\!\left(\frac{n+1}{2},\frac{m-1}{2}\right) \\
	&=\frac{\left| \mathbb{S}^{n+m-1} \right|}{\left| \mathbb{S}^{n} \right|}.
\end{aligned}
\]
 Finally, applying $\left| \mathbb{S}^{n+m-1} \right|=(n+m)\omega_{n+m}$ completes the proof.
\end{proof}

\emph{Proof of \eqref{55}:} 
Since $\mathrm{exp}^{\perp}(\tilde{U}_{r_{0}})=	\Omega_{r_{0}}$, applying the area formula to the map $\mathrm{exp}^{\perp}|_{	\tilde{U}_{r_{0}}}$ and Theorem \ref{57}, 
\begin{equation*}
	\begin{array}{lllll}
		|\Omega_{r_{0}}|&\leq \int_{T^{\perp}\Sigma}1_{\tilde{U}_{r_{0}}}(x,y)|\mathrm{det}\ \mathrm{exp}^{\perp}_{*(x,y)}| d\mathrm{vol}_{T^{\perp}\Sigma}(x,y)\\
		&\leq\int_{T^{\perp}\Sigma}1_{\tilde{U}_{r_{0}}}(x,y)\left(1-\langle \mathbf{H}(x),y \rangle \right)^{n}d\mathrm{vol}_{T^{\perp}\Sigma}(x,y).
	\end{array}
\end{equation*}
Note that for each $x\in \Sigma_{+}$ and $\xi\in E_{x}^{2}$,
\begin{equation*}
	1-r|\mathbf{H}(x)|\langle \mathbf{e}(x),\xi \rangle\leq 1.
\end{equation*}
Consequently, 
\begin{equation*}
	\begin{array}{lllll}
		&\ \ \ \int_{T^{\perp}\Sigma}1_{\tilde{U}_{r_{0}}}(x,y)\left(1-\langle \mathbf{H}(x),y \rangle \right)^{n} d\mathrm{vol}_{T^{\perp}\Sigma}(x,y)\\
		&\leq\int_{\Sigma_{0}}\int_{S_{x}^{m-1}}\int_{0}^{\min\left\{\tilde{\tau}_{f}(x,\xi),r_{0}\right\}}r^{m-1}drd\xi d\mathrm{vol}_{\Sigma}(x)\\
		&\ \ \ +\int_{\Sigma_{+}}\int_{E_{x}^{1}}\int_{0}^{\min\left\{\tilde{\tau}_{f}(x,\xi),r_{0}\right\}}\left( 1-r|\mathbf{H}(x)|\langle \mathbf{e}(x),\xi \rangle\right)^{n}r^{m-1}drd\xi d\mathrm{vol}_{\Sigma}(x)\\
		&\ \ \ +\int_{\Sigma_{+}}\int_{E_{x}^{2}}\int_{0}^{\min\left\{\tilde{\tau}_{f}(x,\xi),r_{0}\right\}}r^{m-1}drd\xi d\mathrm{vol}_{\Sigma}(x).
	\end{array}
\end{equation*}
Extending the upper limit of the $r$-integral to $r_{0}$ and extracting the terms of lower order in $r_{0}^{n+m}$ gives 
\begin{equation*}
	\begin{array}{lllll}
		|\Omega_{r_{0}}|&\leq\int_{\Sigma_{+}}   \int_{E_{x}^{1}}\int_{0}^{r_{0}}  \left|\mathbf{H}(x)\right|^{n}\langle -\mathbf{e}(x),\xi\rangle ^{n} r^{n+m-1}dr d\xi d\mathrm{vol}_{\Sigma}(x)\\
		&\ \ \ +\mathrm{lowerorder\ terms\ of}\ r_{0}^{n+m}\\
		&=\frac{r_{0}^{n+m}}{n+m}\int_{\Sigma_{+}}\left|\mathbf{H}(x)\right|^{n}\int_{E_{x}^{1}}\langle -\mathbf{e}(x),\xi\rangle ^{n}d\xi  d\mathrm{vol}_{\Sigma}(x)\\
		&\ \ \ +\mathrm{lowerorder\ terms\ of}\ r_{0}^{n+m}\\
		&=\frac{r_{0}^{n+m}\omega_{n+m}}{\left| \mathbb{S}^{n} \right|}\int_{\Sigma}\left|\mathbf{H}\right|^{n}d\mathrm{vol}_{\Sigma}+\mathrm{lowerorder\ terms\ of}\ r_{0}^{n+m},
	\end{array}
\end{equation*}
where the last equality follows from Lemma~\ref{91}.
Dividing by $r_{0}^{n+m}\omega_{n+m}$ and letting $r_{0}\rightarrow \infty$ can give us
\begin{equation*}
	\int_{\Sigma}\left|\mathbf{H}\right|^{n}d\mathrm{vol}_{\Sigma}\geq	\mathrm{AVR}(\bar{g})\left| \mathbb{S}^{n} \right|				,
\end{equation*}
which completes the proof.
\hfill$\Box$

\section{Proof of the sufficiency part of Theorem \ref{68}}\label{108}
	Let $(M^{n+m},\bar{g})$ $(n\geq2)$ be a complete noncompact Riemannian manifold with $\mathrm{Ric}_{n}^{M}\geq 0$ and Euclidean volume growth. Let $\Sigma$ be a closed $n$-dimensional  Riemannian manifold  and $f: \Sigma^{n} \rightarrow M$ be an isometric immersion satisfying
\begin{equation}\label{71}
	\begin{array}{lllll}
		\displaystyle\int_{\Sigma}\left|\mathbf{H}\right|^{n} d\mathrm{vol}_{\Sigma}= \mathrm{AVR}(\bar{g})\left| \mathbb{S}^{n} \right|.
	\end{array}
\end{equation}

 \begin{lem}\label{49}
 	 $\Sigma$ is connected.
 \end{lem}
 \begin{proof}
 Suppose otherwise that $\Sigma$ has $N$ connected components, denoted by $\Sigma^{(1)}, \cdots, \Sigma^{(N)}$. Then the assumption and the argument in Section \ref{107} yield
 		\begin{equation*}
 	\begin{array}{lllll}
 	\mathrm{AVR}(\bar{g})\left| \mathbb{S}^{n} \right|=\displaystyle\int_{\Sigma}\left|\mathbf{H}\right|^{n} d\mathrm{vol}_{\Sigma}=	\sum_{k=1}^{N}\displaystyle\int_{\Sigma^{(k)}}\left|\mathbf{H}\right|^{n} d\mathrm{vol}_{\Sigma}\geq N\cdot \mathrm{AVR}(\bar{g})\left| \mathbb{S}^{n} \right|,
 	\end{array}
 \end{equation*}
 which implies $N=1$, as claimed.
\end{proof}

\begin{lem}\label{48}
	For all $(x,\xi)\in U\Sigma$ satisfying $\mathbf{H}(x)\neq 0$ and $\langle \mathbf{H}(x),\xi \rangle< 0$, $\tilde{\tau}_{f}(x,\xi)=\infty$. 
\end{lem}
\begin{proof}
	Suppose that there exists a point $x_{0}\in\Sigma$ and a unit normal vector $\xi_{0}\in S_{x}^{m-1}$ satisfying $\mathbf{H}(x_{0})\neq 0$ and  $\langle \mathbf{H}(x_{0}),\xi_{0} \rangle< 0$ such that $\tilde{\tau}_{f}(x_{0},\xi_{0})<\infty$. By continuity, there exists a positive number $\varepsilon<(n+m)\omega_{n+m}/\left| \mathbb{S}^{n} \right|$ and an open neighborhood $W\subset U\Sigma$ of $(x_{0},\xi_{0})$ such that  $\mathbf{H}(x)\neq 0$, $\tilde{\tau}_{f}(x,\xi)<\tilde{\tau}_{f}(x_{0},\xi_{0})+\varepsilon<\infty$, $\langle \mathbf{H}(x),\xi \rangle< 0$ and
	\[
	\int_{W_{x}}\langle -\mathbf{e}(x),\xi\rangle ^{n}d\xi>\varepsilon,
	\]
	 for all $(x,\xi)\in W$, where $W_{x}=\{\xi\in S_{x}^{m-1}:(x,\xi)\in W \}$. Similar to Section \ref{107}, for each $r_{0}>\tilde{\tau}_{f}(x_{0},\xi_{0})+\varepsilon$,
\begin{align*}
			|\Omega_{r_{0}}|&\leq\int_{\Sigma_{0}}\int_{S_{x}^{m-1}}\int_{0}^{\min\left\{\tilde{\tau}_{f}(x,\xi),r_{0}\right\}}r^{m-1}drd\xi d\mathrm{vol}_{\Sigma}(x)\\
			&\ \ \ +\int_{\pi(W)}\int_{E_{x}^{1}\setminus W_{x}}\int_{0}^{\min\left\{\tilde{\tau}_{f}(x,\xi),r_{0}\right\}}\left( 1-r|\mathbf{H}(x)|\langle \mathbf{e}(x),\xi \rangle\right)^{n}r^{m-1}drd\xi d\mathrm{vol}_{\Sigma}(x)\\
		&\ \ \ +\int_{\pi(W)}\int_{ W_{x}}\int_{0}^{\min\left\{\tilde{\tau}_{f}(x,\xi),r_{0}\right\}}\left( 1-r|\mathbf{H}(x)|\langle \mathbf{e}(x),\xi \rangle\right)^{n}r^{m-1}drd\xi d\mathrm{vol}_{\Sigma}(x)\\
			&\ \ \ +\int_{\Sigma_{+}\setminus\pi(W)}\int_{ E_{x}^{1}}\int_{0}^{\min\left\{\tilde{\tau}_{f}(x,\xi),r_{0}\right\}}\left( 1-r|\mathbf{H}(x)|\langle \mathbf{e}(x),\xi \rangle\right)^{n}r^{m-1}drd\xi d\mathrm{vol}_{\Sigma}(x)\\
			&\ \ \ +\int_{\Sigma_{+}}\int_{E_{x}^{2}}\int_{0}^{\min\left\{\tilde{\tau}_{f}(x,\xi),r_{0}\right\}}r^{m-1}drd\xi d\mathrm{vol}_{\Sigma}(x)\\
			&\leq\int_{\Sigma_{0}}\int_{S_{x}^{m-1}}\int_{0}^{r_{0}}r^{m-1}drd\xi d\mathrm{vol}_{\Sigma}(x)\\
			&\ \ \ +\int_{\pi(W)}\int_{E_{x}^{1}\setminus W_{x}}\int_{0}^{r_{0}}\left( 1-r|\mathbf{H}(x)|\langle \mathbf{e}(x),\xi \rangle\right)^{n}r^{m-1}drd\xi d\mathrm{vol}_{\Sigma}(x)\\
			&\ \ \ +\int_{\pi(W)}\int_{ W_{x}}\int_{0}^{\tilde{\tau}_{f}(x_{0},\xi_{0})+\varepsilon}\left( 1-r|\mathbf{H}(x)|\langle \mathbf{e}(x),\xi \rangle\right)^{n}r^{m-1}drd\xi d\mathrm{vol}_{\Sigma}(x)\\
			&\ \ \ +\int_{\Sigma_{+}\setminus\pi(W)}\int_{ E_{x}^{1}}\int_{0}^{r_{0}}\left( 1-r|\mathbf{H}(x)|\langle \mathbf{e}(x),\xi \rangle\right)^{n}r^{m-1}drd\xi d\mathrm{vol}_{\Sigma}(x)\\
			&\ \ \ +\int_{\Sigma_{+}}\int_{E_{x}^{2}}\int_{0}^{r_{0}}r^{m-1}drd\xi d\mathrm{vol}_{\Sigma}(x)\\
			&=\frac{r_{0}^{n+m}}{n+m}\int_{\pi(W)}   \int_{E_{x}^{1}\setminus W_{x}}  \left|\mathbf{H}(x)\right|^{n}\langle -\mathbf{e}(x),\xi\rangle ^{n} d\xi d\mathrm{vol}_{\Sigma}(x)\\
			&\ \ \ +\frac{r_{0}^{n+m}}{n+m}\int_{\Sigma_{+}\setminus\pi(W)}   \int_{E_{x}^{1}}  \left|\mathbf{H}(x)\right|^{n}\langle -\mathbf{e}(x),\xi\rangle ^{n}  d\xi d\mathrm{vol}_{\Sigma}(x)\\
			&\ \ \ +\mathrm{lowerorder\ terms\ of}\ r_{0}^{n+m}\\
			&\leq \frac{r_{0}^{n+m}}{n+m}\left(\frac{(n+m)\omega_{n+m}}{\left| \mathbb{S}^{n} \right|}-\varepsilon\right)\int_{\pi(W)}\left|\mathbf{H}\right|^{n}d\mathrm{vol}_{\Sigma}+\frac{r_{0}^{n+m}\omega_{n+m}}{\left| \mathbb{S}^{n} \right|}\int_{\Sigma_{+}\setminus\pi(W)}\left|\mathbf{H}\right|^{n}d\mathrm{vol}_{\Sigma}\\
						&\ \ \ +\mathrm{lowerorder\ terms\ of}\ r_{0}^{n+m}.
\end{align*}
	Dividing by $r_{0}^{n+m}\omega_{n+m}$ and letting $r_{0}\rightarrow \infty$ can lead to
	\begin{equation*}
		\mathrm{AVR}(\bar{g})\left| \mathbb{S}^{n} \right|					\leq \int_{\Sigma}\left|\mathbf{H}\right|^{n}d\mathrm{vol}_{\Sigma}-\frac{\varepsilon\left| \mathbb{S}^{n} \right|	}{(n+m)\omega_{n+m}}\int_{\pi(W)}\left|\mathbf{H}\right|^{n}d\mathrm{vol}_{\Sigma}<\int_{\Sigma}\left|\mathbf{H}\right|^{n}d\mathrm{vol}_{\Sigma},
	\end{equation*}
	which is a contradiction. The lemma follows.
\end{proof}

\begin{lem}\label{62}
For all $(x,y)\in T^{\perp}\Sigma$ satisfying $x\in \Sigma_{+}$ and $y/|y|\in E_{x}^{1}$,
\begin{equation*}
|\mathrm{det}\ (\mathrm{exp}^{\perp})_{*(x,y)}|=\left(1-\langle \mathbf{H}(x),y \rangle \right)^{n}.
\end{equation*}
\end{lem}
\begin{proof}
The proof is highly similar to that of Lemma \ref{59}. We omit it.
\end{proof}

\begin{lem}\label{81}
	Assume that $x\in \Sigma_{+}$ and $y/|y|\in E_{x}^{1}$. Let $\gamma(t):=\mathrm{exp}^{\perp}_{x}(ty)$ for all $t\in [0,\infty)$. Then the sectional curvature of $M$ is equal to $\delta$ for all tangent planes containing $\gamma'(t)$	and $x$ is an umbilical point, i.e., 	for all	$	X,Y\in T_{x}\Sigma, y \in T_{x}^{\perp}\Sigma$,
		\begin{equation*}
		\langle h(X,Y),y \rangle =\langle X,Y\rangle \langle \mathbf{H}(x),y \rangle  .
	\end{equation*}
\end{lem}
\begin{proof}
By Lemma \ref{62} and Theorem \ref{57}, the lemma follows.
\end{proof}

\begin{lem}\label{82}
  	We have $\Sigma_{+}=\Sigma$ and $D^{\perp}\mathbf{H}(x)=0$ for all $x\in \Sigma$. Moreover, $|\mathbf{H}|$ is a non-zero constant.
\end{lem}
\begin{proof}
	Fix a point $x\in\Sigma_{+}$. Let $\{E_{i}\}_{i=1}^{n}$ be a local orthonormal tangent frame and $\xi$ a local unit normal vector field around $x$ with $e_{i}:=E_{i}(x)$ and $y:=\xi(x)$.
	We claim that 
	\begin{align}\label{86}
		\sum_{i=1}^{n}\bar{R}(e_{n},e_{i},e_{i},y)=0.
	\end{align}
 When $m=1$, Lemma \ref{81} yields $\mathrm{Ric}^{M}(y,y)=0$. Together with this fact and the condition $\mathrm{Ric}^{M}\geq 0$, it follows immediately that \eqref{86} holds.
	When $m\geq 2$, there exists a unit normal vector $z\in S_{x}^{m-1}$ such that $\langle z,y\rangle=0$. For convenience,	set  $e_{0}:=z$ and $P:= \{e_{i}\}_{i=1}^{n-1}\cup\{ z\}=\{e_{i}\}_{i=0}^{n-1}$. Denote by $P^{\perp}$ the orthogonal complement of $P$ in $T_{f(x)}M$.
	Under the condition $\mathrm{Ric}^{M}_{n}\geq 0$, the tensor
	\[
	T:=\sum_{i=0}^{n-1}\bar{R}(\cdot,e_{i},e_{i},\cdot):P^{\perp}\times P^{\perp}\rightarrow\R
	\]
is symmetric and positive semi-definite. By Lemma \ref{81}, $T(y,y)=0$. Combining these two facts yields 
		\begin{equation}\label{87}
	T(u,y)=0,
	\end{equation}
for each $u\in P^{\perp}$. In particular, since $e_{n}\in P^{\perp}$, taking $u=e_{n}$ in \eqref{87} gives us 
	\begin{equation}\label{88}
	T(e_{n},y)=\sum_{i=0}^{n}\bar{R}(e_{n},e_{i},e_{i},y)=0.
	\end{equation}
	By Lemma \ref{81},   		\begin{equation*}
		2\bar{R}(e_{n},z,z,y)=\bar{R}(e_{n}+y,z,z,e_{n}+y)-\bar{R}(e_{n},z,z,e_{n})-\bar{R}(y,z,z,y)=0,
	\end{equation*} which together with \eqref{88} yields \eqref{86} immediately.

	By \eqref{52} and \eqref{53}, 
\begin{align*}
	\bar{R}(e_{n},e_{i},e_{i},y)=&\langle D_{e_{n}}^{\perp}(h(E_{i},E_{i})),y \rangle-\langle h(D_{e_{n}}^{\Sigma}E_{i},e_{i}),y\rangle-\langle h(e_{i},D_{E_{n}}^{\Sigma}E_{i}),y\rangle\\
	&-\langle D_{e_{i}}^{\perp}(h(E_{n},E_{i})),y \rangle+\langle h(D_{e_{i}}^{\Sigma}E_{n},e_{i}),y\rangle+\langle h(e_{n},D_{e_{i}}^{\Sigma}E_{i}),y\rangle.
\end{align*}
Summing the above identity over $i=1$ to $n$, together with \eqref{86}, yields 
\begin{align*}
	0=&n\langle D_{e_{n}}^{\perp}\mathbf{H},y \rangle-\sum_{i=1}^{n}\langle D_{e_{n}}^{\Sigma}E_{i},e_{i}\rangle\langle \mathbf{H},y\rangle-\sum_{i=1}^{n}\langle e_{i},D_{e_{n}}^{\Sigma}E_{i}\rangle\langle \mathbf{H},y\rangle\\
	&-\sum_{i=1}^{n}\langle D_{e_{i}}^{\perp}(\langle X,E_{i}\rangle\mathbf{H}),y\rangle
	+\sum_{i=1}^{n}\langle D_{e_{i}}^{\Sigma}X,e_{i}\rangle \langle \mathbf{H},y\rangle+\sum_{i=1}^{n}\langle e_{n},D_{e_{i}}^{\Sigma}E_{i}\rangle\langle \mathbf{H},y\rangle.
\end{align*}
Note that $\langle D_{e_{n}}^{\Sigma}E_{i},e_{i}\rangle=0$ and
\begin{align*}
	&\ \ \ \sum_{i=1}^{n}\langle D_{e_{i}}^{\perp}(\langle X,E_{i}\rangle\mathbf{H}),y\rangle\\
	&=\sum_{i=1}^{n}\langle e_{n},e_{i}\rangle\langle  D_{e_{i}}^{\perp}\mathbf{H},y\rangle
	+\sum_{i=1}^{n}(\langle D_{e_{i}}^{\Sigma}X,e_{i}\rangle+\langle e_{n},D_{e_{i}}^{\Sigma}E_{i}\rangle)\langle \mathbf{H},y\rangle\\
	&=\langle D_{e_{n}}^{\perp}\mathbf{H},y \rangle+\sum_{i=1}^{n}(\langle D_{e_{i}}^{\Sigma}X,e_{i}\rangle+\langle e_{n},D_{e_{i}}^{\Sigma}E_{i}\rangle)\langle \mathbf{H},y\rangle.
\end{align*}
Consequently, 
			\begin{equation*}
	(n-1)\langle D_{e_{n}}^{\perp}\mathbf{H},y \rangle=0.
\end{equation*}
Since $n\geq 2$,
\begin{equation*}
	\langle D_{e_{n}}^{\perp}\mathbf{H},y \rangle=0.
\end{equation*}
By the arbitrariness of $e_{n},y$ and $x$,  it follows that $	D^{\perp}\mathbf{H}(x) =0$ for all $x\in \Sigma_{+}$. Immediately, we obtain that $|\mathbf{H}|$ is a non-zero constant  on each connected component of $\Sigma_{+}$. Since $\Sigma$ is connected, $\Sigma_{0}$ is empty, i.e., $\Sigma_{+}=\Sigma$. Thus $|\mathbf{H}|$ is a non-zero constant on $\Sigma$. The lemma follows.
\end{proof}

\begin{lem}\label{50}
	Assume that $(x_{1},r_{1}\xi)\in \mathcal{W}$ with $x_{1}\in \Sigma,\, r_{1}>0$ and $\xi_{1} \in S_{x}^{m-1}$. Then for every $x_{2} \in \Sigma$ and every positive number $r_{2} < r_{1}$,  $\mathrm{exp}^{\perp}(x_{1}, r_{1}\xi_{1}) \notin B_{f(x_{2})}^{M}(r_{2})$.
\end{lem}
\begin{proof}
	The proof is highly similar to that of Lemma \ref{63}. We omit it.
\end{proof}

\begin{lem}\label{75}
	Assume that $(x,y)\in \mathcal{W}$. Then 
	$\tau_{f}(x,y/|y|)=\infty$.
\end{lem}
\begin{proof}
	The proof is highly similar to that of Lemma \ref{64}. We omit it.
\end{proof}

\begin{lem}\label{76}
	Suppose $V$ is a nonempty open set contained in $\Sigma$ such that $f|_{V}$ is an embedding. Define
	$\mathcal{V}:=\{(x,y)\in T^{\perp}\Sigma: x\in V,\, y\in T_{x}^{\perp}\Sigma,\, \langle y,-\mathbf{H}(x) \rangle > 0  \}$. Then $\mathrm{exp}^{\perp}|_{\mathcal{V}}:\mathcal{V}\rightarrow \mathrm{exp}^{\perp}(\mathcal{V})$  is a diffeomorphism.
\end{lem}
\begin{proof}
Since Lemma \ref{75} holds, we can prove this lemma by a similar argument as in Lemma \ref{73}. We omit it.
\end{proof}

\begin{lem}\label{80}
 The map $f$ is an embedding.
\end{lem}
\begin{proof}
Since Lemmas \ref{48} and \ref{50} hold, we can prove this lemma by a similar argument as in Lemma \ref{69}. We omit it.
\end{proof}

\emph{Proof of the sufficiency part of Theorem \ref{68}:} 
By Lemmas \ref{76} and \ref{80}, we conclude that $\mathrm{exp}^{\perp}|_{\mathcal{W}}:\mathcal{W}\rightarrow \mathrm{exp}^{\perp}(\mathcal{W})$ is a diffeomorphism. Combining this fact with Lemmas \ref{49}, \ref{82} and \ref{80}, it remains only to compute the pulled‑back metric on $\mathcal{W}$.

Fix a point $(\bar{x},\bar{y})\in\mathcal{W}$. We shall use the frame fields established prior to Lemma \ref{9}. By Lemmas \ref{9} and \ref{81}, 
\begin{align*}
	&\langle (\mathrm{exp}^{\perp})_{*(\bar{x},\bar{y})}\left(X_{i}\right),\eta_{j} \rangle=g_{ij}(\bar{x})\left(1-\langle \mathbf{H}(\bar{x}),\bar{y} \rangle \right),\\
	&\langle (\mathrm{exp}^{\perp})_{*(\bar{x},\bar{y})}\left(X_{i}\right),\eta_{\beta} \rangle=\bar{y}^{\sigma}\Gamma_{i \sigma}^{\beta}(\bar{x}),\\
	&\langle (\mathrm{exp}^{\perp})_{*(\bar{x},\bar{y})}\left(\partial/\partial y^{\alpha}\right),\eta_{j} \rangle=0,\\
			&\langle (\mathrm{exp}^{\perp})_{*(\bar{x},\bar{y})}\left(\partial/\partial y^{\alpha}\right),\eta_{\beta} \rangle=\delta_{\alpha}^{\beta}.
\end{align*}
For each $1\leq A\leq n+m$, let  $\bar{\omega}^{A}$ be  the dual $1$-form of $\eta_{A}$. Note that 
			\begin{equation*}
\bar{g}_{AB}(\mathrm{exp}^{\perp}_{\bar{x}}\bar{y})=\langle \eta_{A},\eta_{B} \rangle =\langle e_{A},e_{B} \rangle
\end{equation*}
and
			\begin{equation*}
	\bar{g}^{ij}(\mathrm{exp}^{\perp}_{\bar{x}}\bar{y})=g^{ij}(\bar{x}) ,\,	\bar{g}^{i\alpha}(\mathrm{exp}^{\perp}_{\bar{x}}\bar{y})=g^{i\alpha }(\bar{x})=0,\,\bar{g}^{\alpha\beta}(\mathrm{exp}^{\perp}_{\bar{x}}\bar{y})=\delta_{\alpha\beta},
\end{equation*}
for  $1\leq i,j\leq n$ and $n+1\leq \alpha,\beta\leq n+m$. 
These yield
\begin{align*}
	(\mathrm{exp}^{\perp})^{*}\bar{\omega}^{i}&= \bar{g}^{iA}(\mathrm{exp}^{\perp}_{\bar{x}}\bar{y}) \langle (\mathrm{exp}^{\perp})_{*(\bar{x},\bar{y})}\left(X_{j}\right),\eta_{A} \rangle  dx^{j}+ \bar{g}^{iA}(\mathrm{exp}^{\perp}_{\bar{x}}\bar{y}) \langle (\mathrm{exp}^{\perp})_{*(\bar{x},\bar{y})}\left(\partial/\partial y^{\beta}\right),\eta_{A} \rangle dy^{\beta}  \\
	&=\left(1-\langle \mathbf{H}(\bar{x}),\bar{y} \rangle \right)dx^{i},
\end{align*}
\begin{align*}
	(\mathrm{exp}^{\perp})^{*}\bar{\omega}^{\alpha}&= \bar{g}^{\alpha A}(\mathrm{exp}^{\perp}_{\bar{x}}\bar{y}) \langle (\mathrm{exp}^{\perp})_{*(\bar{x},\bar{y})}\left(X_{j}\right),\eta_{A} \rangle  dx^{j}+ \bar{g}^{\alpha A} (\mathrm{exp}^{\perp}_{\bar{x}}\bar{y})\langle (\mathrm{exp}^{\perp})_{*(\bar{x},\bar{y})}\left(\partial/\partial y^{\beta}\right),\eta_{A} \rangle dy^{\beta}  \\
	&=\bar{y}^{\beta}\Gamma_{i \beta}^{\alpha}(\bar{x})dx^{i}+dy^{\alpha}
\end{align*}
and
		\begin{align*}
\bar{g}(\mathrm{exp}^{\perp}_{\bar{x}}\bar{y})&=\langle \eta_{A},\eta_{B} \rangle \bar{\omega}^{A}\otimes\bar{\omega}^{B}=\langle e_{A},e_{B} \rangle \bar{\omega}^{A}\otimes\bar{\omega}^{B}\\
&=g_{ij}(\bar{x}) \bar{\omega}^{i}\otimes\bar{\omega}^{j}+\delta_{\alpha\beta}\bar{\omega}^{\alpha}\otimes\bar{\omega}^{\beta}.
\end{align*}
Thus the pulled back metric
\begin{align*}
[	(\mathrm{exp}^{\perp})^{*}\bar{g}](\bar{x},\bar{y})&=g_{ij}(\bar{x}) (\mathrm{exp}^{\perp})^{*}\bar{\omega}^{i}\otimes(\mathrm{exp}^{\perp})^{*}\bar{\omega}^{j}+\delta_{\alpha\beta}(\mathrm{exp}^{\perp})^{*}\bar{\omega}^{\alpha}\otimes(\mathrm{exp}^{\perp})^{*}\bar{\omega}^{\beta}\\
	&=\delta_{\alpha\beta}( \bar{y}^{\sigma}\Gamma_{i \sigma}^{\alpha}(\bar{x})dx^{i}+dy^{\alpha})\otimes( \bar{y}^{\tau}\Gamma_{j \tau}^{\beta}(\bar{x})\otimes dx^{j}  +dy^{\beta})\\
	&\ \ \ +	g_{ij}(\bar{x})\left(1-\langle \mathbf{H}(\bar{x}),\bar{y} \rangle \right)^{2}dx^{i}\otimes dx^{j}.
\end{align*}
Recalling the canonical metric (see \eqref{13}) on $\pi^{-1}(U)\subset T^{\perp}\Sigma$, we obtain
\begin{align*}
g_{T^{\perp}\Sigma}(\bar{x},\bar{y})=\delta_{\alpha\beta}( \bar{y}^{\sigma}\Gamma_{i \sigma}^{\alpha}(\bar{x})dx^{i}+dy^{\alpha})\otimes( \bar{y}^{\tau}\Gamma_{j \tau}^{\beta}(\bar{x})\otimes dx^{j}  +dy^{\beta})+g_{\Sigma}(\bar{x}).
\end{align*}
Therefore, the pulled back metric on $\mathcal{W}$ is given by
\begin{align*}
	[(\mathrm{exp}^{\perp})^{*}\bar{g}](x,y)=	\left[\left(1-\langle \mathbf{H}(x),y \rangle \right)^{2}-1\right]g(x)+g_{T^{\perp}\Sigma}(x,y),
\end{align*}
	for each $(x,y)\in \mathcal{W}$.

These complete the proof.
\hfill$\Box$\\

\end{document}